\documentclass[1 [leqno,11pt]{amsart}
\usepackage{amssymb, amsmath}
\allowdisplaybreaks[4]
\usepackage{xcolor}
\usepackage[hidelinks]{hyperref}
\hypersetup{colorlinks={true},linkcolor={blue},citecolor=blue}

\let\f=\frac

\let\wt=\widetilde

\let\D=\Delta
\let\p=\partial

\def\ka{\kappa}

\def\v{{\rm v}}

\def\ff{\frak{f}}
\def\fg{\frak{g}}

\def\oo{\infty}
\def\ff{\frak{f}}
\def\B{{\rm B}}

\def\rC{{\rm C}}

\def\Z{\mathop{\mathbb Z\kern 0pt}\nolimits}
\def\N{\mathop{\mathbb N\kern 0pt}\nolimits}
\def\Q{\mathop{\mathbb Q\kern 0pt}\nolimits}
\def\R{\mathop{\mathbb R\kern 0pt}\nolimits}

\def\dive{{\mathop{\rm div}\nolimits}\,}

\def\A{{\rm A}}
\def\h{{\rm h}}

\def\vv{{\rm v}}
\def\fg{\frak{g}}
\def\na{\nabla}
\def\rb{{\rm b}}

\def\dB{\dot{B}}
\def\dH{\dot{H}}
\def\dDv{\dot{\Delta}^\vv}
\def\dSv{\dot{S}^\vv}
\def\dTv{\dot{T}^\vv}
\def\dRv{\dot{R}^\vv}

\def\dBl{\dot{B}_{{\rm lo}}}
\def\dHh{\dot{H}_{{\rm hi}}}

\def\eqdefa{\buildrel\hbox{\footnotesize def}\over =}

\newcommand{\Rmnum}[1]{\uppercase\expandafter{\romannumeral #1} }

\newcommand{\beq}{\begin{equation}}
\newcommand{\eeq}{\end{equation}}
\newcommand{\ben}{\begin{eqnarray}}
\newcommand{\een}{\end{eqnarray}}
\newcommand{\beno}{\begin{eqnarray*}}
\newcommand{\eeno}{\end{eqnarray*}}

 \numberwithin{equation}{section}
\newcommand{\andf}{\quad\hbox{and}\quad}
\newcommand{\with}{\quad\hbox{with}\quad}

\newtheorem{defi}{Definition}[section]
\newtheorem{thm}{Theorem}[section]
\newtheorem{lem}{Lemma}[section]
\newtheorem{rmk}{Remark}[section]

\newtheorem{prop}{Proposition}[section]

\begin{document}

\title[Stability of solutions to the 3-D anisotropic Navier-Stokes equations]
{On the stability of solutions to the 3-D anisotropic Navier-Stokes equations in the critical space}

\author[N. Burq]{Nicolas Burq}
\address[N. Burq]
{Laboratoire de Math\'ematiques d'Orsay, Universit\'e Paris-Saclay, Orsay, France, CNRS, UMR 8628 $\&$ Institut Universitaire de
France.}
\email{nicolas.burq@universite-paris-saclay.fr}
\author[N. Liu]{Ning Liu}
\address[N. Liu]
 {Department of Mathematics, New York University Abu Dhabi, Saadiyat Island, P.O. Box 129188, Abu Dhabi, United Arab Emirates.} \email{liuning09281@126.com $\&$ nl2983@nyu.edu  }
\author[P. Zhang]{Ping Zhang}%
\address[P. Zhang]
 {State Key Laboratory of Mathematical Sciences, Academy of Mathematics $\&$ Systems Science, The Chinese Academy of
	Sciences, Beijing 100190, China, and School of Mathematical Sciences,
	University of Chinese Academy of Sciences, Beijing 100049, China. }
\email{zp@amss.ac.cn}

\date{\today}

\begin{abstract} 
In this paper, we study the continuous dependence of solutions to the three-dimensional incompressible anisotropic Navier–Stokes equations $(ANS)$ with initial data in the critical space \smash{$\dot{B}^{0,\frac12}$}. We prove that the data-to-solution map is continuous in this setting. Due to the critical regularity of the problem,  the classical Bona-Smith argument~\cite {BoSm} (see also the abstract result~\cite{ABITZ}) does not appear to apply because a key ingredient (the Lipschitz continuity at lower regularity) does not appear to hold. As a consequence, we develop a new direct Fourier-weighted method within the Besov framework.
\end{abstract}
\maketitle

\noindent {\sl Keywords:} Anisotropic Navier-Stokes equations, anisotropic Littlewood-Paley theory, continuous dependence, critical spaces.

\vskip 0.2cm
\noindent {\sl AMS Subject Classification (2000):} 35Q30, 76D03

\setcounter{equation}{0}

\section{Introduction}
In this paper, we investigate the continuity of the data-to-solution map for the following three-dimensional incompressible anisotropic Navier–Stokes equations with initial data in the critical Besov space:
\begin{equation*}
(ANS)\quad \left\{
\begin{array}{l}
\displaystyle \partial_t u + u \cdot \nabla u - \Delta_{\mathrm{h}} u = -\nabla p, \qquad (t,x) \in \mathbb{R}^+ \times \mathbb{R}^3, \\[4pt]
\displaystyle \operatorname{div} u = 0, \\[4pt]
\displaystyle u|_{t=0} = u_0,
\end{array}
\right.
\end{equation*}
where $\Delta_{\mathrm{h}} \eqdefa \partial_1^2 + \partial_2^2$, $u$ denotes the velocity of the fluid, and $p$ is the scalar pressure function that enforces the divergence-free condition on the velocity field. We refer to \cite{CDGG, Pedlovsky} for the physical background of this system.

\smallskip

Let us first recall the following anisotropic Sobolev spaces.

\begin{defi}\label{S1def1}
For $s, s' \in \mathbb{R}$, we define the anisotropic Sobolev space $H^{s,s'}(\mathbb{R}^3)$ to be the space of tempered distributions $f$ such that
\[
\|f\|_{H^{s,s'}} \eqdefa \| \langle \xi_{\mathrm{h}} \rangle^s \langle \xi_3 \rangle^{s'} \widehat{f}(\xi) \|_{L^2} < \infty,
\]
where $\widehat{f}$ denotes the Fourier transform of $f$ and $\xi_{\mathrm{h}} = (\xi_1, \xi_2)$. Here and throughout the paper, we always  denote $\langle \xi \rangle \eqdefa (1 + |\xi|^2)^{\frac12}$.
\end{defi}

For initial data in $H^{0,s}$ with $s \in (1/2, 1)$, Chemin et al. \cite{CDGG} first proved that system $(ANS)$ has a local solution $u \in C([0, T^*); H^{0,s})$ with $\nabla_{\mathrm{h}} u \in L^2(0, T^*; H^{0,s})$ for some maximal existence time $T^*$. If, in addition,
\begin{equation}\label{smallCDGG}
\|u_0\|_{L^2}^{1 - \frac{1}{2s}} \|u_0\|_{\dot{H}^{0,s}}^{\frac{1}{2s}} \leq c
\end{equation}
for some sufficiently small constant $c$, then $T^* = \infty$. Nevertheless, due to the lack of dissipation in the vertical variable in $(ANS),$  the authors \cite{CDGG} {required} one more derivative in the $x_3$ direction to prove uniqueness of such a solution. Iftimie \cite{Ift} later resolved the uniqueness issue for $(ANS)$ with initial data in $H^{0,s}$ for $s > \frac12$.

Notice that, just as for the classical Navier–Stokes system, system $(ANS)$ possesses the following scaling invariance property:
\begin{equation}\label{NSscaling}
u_\lambda(t,x) \eqdefa \lambda u(\lambda^2 t, \lambda x), \qquad u_{0,\lambda}(x) \eqdefa \lambda u_0(\lambda x),
\end{equation}
which means that if $u$ is a solution of $(ANS)$ on $[0,T]$ with initial data $u_0$, then $u_\lambda$ defined by \eqref{NSscaling} is also a solution of $(ANS)$ on $[0, T/\lambda^2]$ with initial data $u_{0,\lambda}$. The scaling \eqref{NSscaling} determines the minimal regularity required of the initial data for local well-posedness of $(ANS)$. A function space is called critical if its norm is invariant under the scaling transformation $u_0 \mapsto u_{0,\lambda}$.

For the convenience of the reader, we recall the anisotropic dyadic operators from \cite{BCD}. Let $\chi(\tau)$ and $\varphi(\tau)$ be smooth functions such that
\begin{align*}
&\operatorname{Supp} \varphi \subset \Bigl\{ \tau \in \mathbb{R} : \frac34 \leq |\tau| \leq \frac83 \Bigr\}, \quad \text{and} \quad \forall \tau > 0,\ \sum_{j \in \mathbb{Z}} \varphi(2^{-j} \tau) = 1; \\
&\operatorname{Supp} \chi \subset \Bigl\{ \tau \in \mathbb{R} : |\tau| \leq \frac43 \Bigr\}, \quad \text{and} \quad \forall \tau \geq 0,\ \chi(\tau) + \sum_{j \geq 0} \varphi(2^{-j} \tau) = 1,
\end{align*}
and define the anisotropic  dyadic operators as follows:
\begin{equation}\label{S1eq1}
\begin{split}
&\dot{\Delta}_{j}^{\mathrm{v}} a \eqdefa \mathcal{F}^{-1}\bigl( \varphi(2^{-j} |\xi_3|) \widehat{a} \bigr)\andf 
\dot{S}_{j}^{\mathrm{v}} a \eqdefa \mathcal{F}^{-1}\bigl( \chi(2^{-j} |\xi_3|) \widehat{a} \bigr), \quad \text{for } j \in \mathbb{Z},
\end{split}
\end{equation}
where $\mathcal{F}^{-1} a$ denotes the inverse Fourier transform of $a$.

\begin{defi}\label{anibesov}
We define $\dot{B}^{0,\frac12}(\mathbb{R}^3)$ as the set of homogeneous tempered distributions $a$ such that
\[
\|a\|_{\dot{B}^{0,\frac12}} \eqdefa \sum_{j \in \mathbb{Z}} 2^{\frac{j}{2}} \| \dot{\Delta}_{j}^{\mathrm{v}} a \|_{L^2} < \infty.
\]
\end{defi}

Paicu \cite{Paicu} proved the local existence of solutions to $(ANS)$ for any solenoidal vector field $u_0 \in \dot{B}^{0,\frac12}$, as well as global existence for sufficiently small initial data in $\dot{B}^{0,\frac12}$. Chemin and the third author \cite{CZ07} introduced the critical anisotropic Besov space with negative index, $\mathcal{B}^{-\frac12,\frac12}_4$, and established global existence of solutions to $(ANS)$ for sufficiently small initial data in that space. We mention that, due to the special structure of the solutions constructed in \cite{Paicu, CZ07}, uniqueness was obtained by establishing $H^{0,-\frac12}$ estimates for the difference between two solutions. We refer to \cite{LPZ} and the references therein for further well-posedness results for $(ANS).$  We emphasize that in all these works, the existence and uniqueness spaces differ, due to the lack of dissipation in the vertical variable.

On the other hand, we recall that the classical notion of well-posedness in the sense of Hadamard for a Cauchy problem on $[0,T]$ requires
\begin{itemize}
\item[(1)] existence and uniqueness of the solution on $[0,T]$;
\item[(2)] continuous dependence of the solution on the initial data.
\end{itemize}
For the classical Navier–Stokes equations, the data-to-solution map is Lipschitz continuous with respect to the initial data (see \cite{fujitakato}). The third author and Zhu \cite{ZZhu1} proved that solutions of $(ANS)$ depend continuously on the initial data in $H^{0,s}$ for $s > \frac12$, yet the data-to-solution map is not uniformly continuous on bounded subset of  $H^{0,s}$ with  $s > \frac12$. Prior to the present work, continuity in the critical space $\dot{B}^{0,\frac12}$ remained an open problem.

The goal of this paper is to establish the continuity of the data-to-solution map in the critical spaces. Our main result gives the continuity in the space $\dot{B}^{0,\frac12}$:

\begin{thm} \label{thm2}
{\sl Let $u^\infty$ be a strong solution of $(ANS)$ on $[0,T]$ with initial data $u^\infty_0 \in \dot{B}^{0,\frac12}$, and let $\{ u^n_0 \}_{n \in \mathbb{N}}$ be a sequence of functions in $\dot{B}^{0,\frac12}$. If $u^n_0$ converges to $u^\infty_0$ in $\dot{B}^{0,\frac12}$ as $n \to \infty$, then there exists $N \in \mathbb{N}$ such that for all $n > N$, there exists a unique solution $u^n$ of $(ANS)$ on $[0,T]$ with initial data $u^n_0$, and
\begin{equation}\label{S1eq4}
\lim_{n \to \infty} \Bigl( \|(u^n - u^\infty)\|_{L^\infty_T(\dot{B}^{0,\frac12})} + \|\nabla_{\mathrm{h}} (u^n - u^\infty)\|_{L^2_T(\dot{B}^{0,\frac12})} \Bigr) = 0.
\end{equation}}
\end{thm}

We shall sketch the main ideas used in the proof of Theorem~\ref{thm2} in Section \ref{Sect5}.
\medskip

\begin{rmk}
We conclude  by comparing our approach with previous continuity results.
\begin{itemize}
  \item[(1)]  As shown in the abstract framework of \cite{ABITZ}, continuity in a space $X^s$ can be derived from a Lipschitz estimate in a weaker space $X^{s_1}$ together with a tame estimate in a smoother space $X^{s_2}$. The constants in these estimates need to depend only on the $X^s$ norm, and the argument can be viewed as a nonlinear interpolation for $s_1 < s < s_2$. The continuity result in $H^{0,s}$ for $s > \frac12$ obtained in \cite{ZZhu1} follows this strategy, which is basically an elaboration of a classical argument by Bona and Smith~\cite{BoSm} (see also the blog article by T. Tao~\cite{TT})
  
  \item[(2)] However, in the critical space $\dot{B}^{0,\frac12}$, this general method is difficult to apply due to the lack of a suitable Lipschitz estimate in weaker norms. To avoid the loss of vertical derivatives, one naturally considers the weak space $\dot{B}^{0,-\frac12}$. The only available continuity result at this level is \eqref{eq:prop1.3} obtained in \cite{Paicu}, which roughly speaking shows that the norm in the weaker space of the difference of two solutions, $N(t)$,  satisfy
  $$N(t) \leq e^{\log (N(t_0))^{e^{-C|t-t_0|}}}$$ which still implies uniqueness ($e^{-\infty } =0$) but is significantly weaker than a Lipschitz estimate of the form 
  $$ N(t) \leq C_t  N(t_0)$$ 
  required in \cite{ABITZ}.
  
  \item[(3)] Our approach is fundamentally different. Since we show that the initial data actually belong to the smoother space $\dot{B}^{0,\frac12}_{\rm b}$ (see Definition \ref{def:dBb} below), there is no derivative loss when treating the high-frequency part in $\dot{B}^{0,\frac12}$. As a result, we do not rely on abstract interpolation, and any type of continuity in a weaker space suffices to control the low-frequency part, performing in some sense a direct elementary interpolation. In particular, Proposition \ref{prop1.6}, despite being weaker than a Lipschitz estimate, is sufficient for our argument.  Our approach is in fact very general and we believe it should apply to other cases where the Lipschitz contraction estimates does not hold.
\end{itemize}
\end{rmk}

Let us end this section with some notations that will be used throughout this paper.\smallskip

\noindent \textbf{Notations:} For $a \lesssim b$, we mean that there is a uniform constant $C$, which may differ at each occurrence, such that $a \leq C b$. We denote by $(a \mid b)$ the $L^2(\mathbb{R}^3)$ inner product of $a$ and $b$. The sequences $(d_j)_{j \in \mathbb{Z}}$ and $(c_j)_{j \in \mathbb{Z}}$ denote generic elements on the unit spheres of $\ell^1(\mathbb{Z})$ and $\ell^2(\mathbb{Z})$, respectively, i.e., $\sum_{j\in\mathbb{Z}} d_j = \sum_{j\in\mathbb{Z}} c_j^2 = 1$. Finally, we denote by $L^r_T(L^p_{\mathrm{h}}(L^q_{\mathrm{v}}))$ the space $L^r(0,T; L^p(\mathbb{R}_{x_1} \times \mathbb{R}_{x_2}; L^q(\mathbb{R}_{x_3})))$, and set $\nabla_{\mathrm{h}} \eqdefa (\partial_{x_1}, \partial_{x_2})$, $\operatorname{div}_{\mathrm{h}} \vec{f} = \partial_{x_1} f_1 + \partial_{x_2} f_2$ for $\vec{f} = (f_1, f_2)$.
\smallskip

\section{Ideas and structures of  the proof of Theorem~\ref{thm2}}\label{Sect5}

In this section, we sketch the main ideas used in the proof of Theorem~\ref{thm2}. The key idea is the following: given a sequence of initial data $u^n_0 \to u^\infty_0$ in $\dot{B}^{0,\frac12}$, we show that this convergence actually holds in a slightly smoother space than $\dot{B}^{0,\frac12}$. More precisely, we introduce the following Fourier-weighted Besov space:

\begin{defi}\label{def:dBb}
Given an arithmetic function ${\rm b} = \{ b_j \}_{j \in \mathbb{Z}}$ with positive values, we define the anisotropic Besov-type space $\dot{B}^{0,\frac12}_{\rm b}$ as the space of tempered distributions $f$ satisfying
\begin{equation}\label{eqdef:dBb}
\|f\|_{\dot{B}^{0,\frac12}_{\rm b}} \eqdefa \sum_{j\in \mathbb{Z}} b_j2^{\frac{j}{2}}  \| \dot{\Delta}_j^{\mathrm{v}} f \|_{L^2}.
\end{equation}
To exploit the optimal regularity for parabolic equations, for any $T>0,$  we also need  the following Chemin-Lerner type norm $\|\cdot\|_{\widetilde L^p_T(\dB_{\rm b}^{0,\f12})}$ (\cite{CL95}):
\begin{equation}\label{CLdBb}\|f\|_{\widetilde L^p_T(\dB^{0,\f12}_{\rm b})}\eqdefa\sum_{j\in\Z}
b_j 2^{\frac{j}{2}}\| \dDv_j f\|_{L^p_T(L^2)}.\end{equation}
\end{defi}

We first prove that there exists a suitable choice of the arithmetic function ${\rm b}$ such that the initial data also converge in $\dot{B}^{0,\frac12}_{\rm b}$.

\begin{prop}\label{prop1.4}
{\sl If a sequence of functions $\{ u^n_0 \}_{n \in \mathbb{N}} \in \dot{B}^{0,\frac12}$ converges to $u^\infty_0$ in $\dot{B}^{0,\frac12}$ as $n \to \infty$, then there exists a non-decreasing arithmetic function ${\rm b} = \{ b_j \}_{j \in \mathbb{Z}}$ of positive numbers satisfying $\lim_{j \to \infty} b_j = +\infty$, such that for all $n \in \mathbb{N}$, $u^n_0$  belongs to $\dot{B}^{0,\frac12}_{\rm b}$ and $\{u^n_0\}_{n\in\N}$ converges to $u^\infty_0$ in $\dot{B}^{0,\frac12}_{\rm b}$ as $n \to \infty$.}
\end{prop}

\begin{rmk}
One can further choose $\rb$ to be a smaller function growing sufficiently slowly. In particular, throughout this paper, we assume that $\rb$ satisfies $b_j \leq b_{j+1} \leq 2b_j$.
\end{rmk}

The proof of the elementary Proposition \ref{prop1.4} will be presented in Section \ref{Sect2}.  We then establish the following estimate to propagate the $\dot{B}^{0,\frac12}_\rb$ regularity, the proof of which will be postponed in Section \ref{Sect3}.

\begin{prop}[Tame estimate in $\dot{B}^{0,\frac12}_{\rm b}$]\label{prop1.5}
{\sl Let ${\rm b} = \{ b_j \}_{j \in \mathbb{Z}}$ be a non-decreasing arithmetic function of positive numbers satisfying $b_j \leq b_{j+1} \leq 2b_j$. If $u(t)$ is a solution of $(ANS)$ on $[0,T]$ with initial data $u_0 \in \dot{B}^{0,\frac12}_{\rm b}$, then for all $t \in [0,T]$, one has
\begin{equation}\label{eq:prop1.5}
\begin{aligned}
&\|u\|_{\wt{L}^\infty_t(\dot{B}^{0,\frac12}_{\rm b})} + \|\nabla_{\mathrm{h}} u\|_{\wt{L}^2_t(\dot{B}^{0,\frac12}_{\rm b})}  \leq C e^{C {\rm A}(t)} \|u_0\|_{\dot{B}^{0,\frac12}_\rb} \\
&\text{with } {\rm A}(t) \eqdefa \int_0^t \left( 1 + \|u(\tau)\|_{\dot{B}^{0,\frac12}}^2 \right) \left( 1 + \|\nabla_{\mathrm{h}} u(\tau)\|_{\dot{B}^{0,\frac12}}^2 \right) \, d\tau.
\end{aligned}
\end{equation}}
\end{prop}

Next, we want to derive a continuity estimate in some weaker space. The $H^{0,-\frac12}$ norm can overcome the vertical derivative loss for high frequencies, while for low frequencies, it is better to consider the $\dot{B}^{0,\frac12}$ norms, since the initial data does not belong to any space with lower regularity. More precisely, we first introduce two seminorms concerning  different frequencies:
\begin{equation}\label{2.4}
  \|f\|_{\dBl^{0,\f12}}\eqdefa\sum_{j\leq -1} 2^\f{j}2\|\dDv_j f\|_{L^2} 
  \andf
  \|f\|_{\dHh^{0,-\f12}}\eqdefa \Bigl( \sum_{j\geq 0} 2^{-j}{\|\dDv_j f\|_{L^2}^2}\Bigr)^\f12,
\end{equation}
and the norm $\|\cdot\|_X:$
\begin{equation}\label{S1eq1a}
\begin{split}
  \|f\|_{X} \eqdefa \|f\|_{\dBl^{0,\f12}} + \|f\|_{\dHh^{0,-\frac12}}& = \sum_{j\leq -1} 2^{\frac{j}{2}} \|\dot{\Delta}_j^{\mathrm{v}} f\|_{L^2} + \Bigl( \sum_{j\geq 0} 2^{-j} \|\dot{\Delta}_j^{\mathrm{v}} f\|_{L^2}^2 \Bigr)^{\frac12}.
  \end{split}
\end{equation}

Although the Osgood argument works for the high frequency seminorm $\dHh^{0,-\f12}$, we cannot apply Gronwall type argument for the low frequency part in Chemin-Lerner  type Besov spaces. For this reason,  as in \cite{PZ11}, we need to use time-weighted Chemin-Lerner type norm.  Motivated by \cite{CL95},  for any $p\in [1,\infty],$ we introduce the associated Chemin–Lerner seminorms:
\begin{equation}\label{S1eq2aq}
  \|f\|_{\wt L^p_T(\dBl^{0,\f12})}\eqdefa\sum_{j\leq -1} 2^\f{j}2\|\dDv_j f\|_{L^p_T(L^2)} 
  \andf
  \|f\|_{\wt L^p_T(\dHh^{0,-\f12})}\eqdefa \Bigl( \sum_{j\geq 0} 2^{-j}{\|\dDv_j f\|_{L^p_T(L^2)}^2}\Bigr)^\f12.
\end{equation}
and denote their summation by:
\begin{equation}\label{S1eq2a}
\begin{split}
\|f\|_{\widetilde L^p_T(X)} &\eqdefa \|f\|_{\widetilde L^p_T(\dBl^{0,\frac12})} + \|f\|_{\widetilde L^p_T(\dHh^{0,-\frac12})}.
\end{split}
\end{equation}

Then we present the estimates for the difference of any two solutions  of $(ANS)$ in the time-weighted $\wt{L}^\infty_T(X)\cap\wt{L}^2_T(X)$ space, the proof of which will be postponed to Section \ref{Sect4}:

\begin{prop}[Weak continuity in $X$]\label{prop1.6}
{\sl Let $u(t)$ and $v(t)$ be two solutions of $(ANS)$ on $[0,T]$ with initial data $u_0$ and $v_0$, respectively. If $\|u - v\|_{\widetilde L^\infty_T(X)} \leq 2^{-10}$, then for all $t \in [0,T]$, there holds
\begin{equation}\label{eq:prop1.6}
\begin{aligned}
 &\|\fg(u-v)\|_{\widetilde L^\infty_t(X)}^2 + \|\fg\nabla_{\mathrm{h}} (u - v)\|_{\widetilde L^2_t(X)}^2 \\
 &\leq C \|u_0-v_0\|_{X}^2+C\int_0^t \B'(\tau)\Bigl(-\ln\|\left(\fg(u-v)\right)\|_{\widetilde L^\infty_\tau(X)}^2\Bigr)\\
 &\qquad\times\ln\Bigl(-\ln\|\left(\fg(u-v)\right)\|_{\widetilde L^\infty_\tau(X)}^2\Bigr) \|\left(\fg(u-v)\right)\|_{\widetilde L^\infty_\tau(X)}^2\, d\tau, \with \\
\frak{g}(t)\eqdefa &e^{-C{\rm B}(t)}\andf {\rm B}(t) \eqdefa \int_0^t \Bigl( 1 + \|(u,v)(\tau)\|_{\dot{B}^{0,\frac12}}^2 \Bigr)
\Bigl( 1 + \|\nabla_{\mathrm{h}} (u,v)(\tau)\|_{\dot{B}^{0,\frac12}}^2 \Bigr) \, d\tau.
\end{aligned}
\end{equation}}
\end{prop}

\begin{rmk}
  We have several comments in order concerning the estimate \eqref{eq:prop1.6}:
  
  \begin{itemize}
    \item[(1)] Paicu \cite{Paicu}  established the following  $H^{0,-\frac12}$ estimate for the difference of any two solutions of $(ANS)$:
    \begin{equation}\label{eq:prop1.3}
\begin{aligned}
\frac{d}{dt} \|(u - v)(t)\|_{H^{0,-\frac12}}^2 &+ \|\nabla_{\mathrm{h}} (u - v)(t)\|_{H^{0,-\frac12}}^2 \leq C \frak{B}'(t) \|(u- v)(t)\|_{H^{0,-\frac12}}^2 \\
& \times \Bigl( 1 - \ln\|(u- v)(t)\|_{H^{0,-\frac12}}^2 \Bigr)
   \ln\Bigl( 1 - \ln\|(u- v)(t)\|_{H^{0,-\frac12}}^2 \Bigr), \\
\text{with } \frak{B}(t) &\eqdefa \int_0^t \Bigl( 1 + \|(u,v)(\tau)\|_{H^{0,\frac12}}^2 \Bigr)
                     \Bigl( 1 + \|\nabla_{\mathrm{h}} (u,v)(\tau)\|_{H^{0,\frac12}}^2 \Bigr) \, d\tau.
\end{aligned}
\end{equation}
The estimate \eqref{eq:prop1.6} can be understood as an integrated counterpart of  \eqref{eq:prop1.3}, yet the low frequency part of the difference in the $X$ norm  is  $\dot{B}^{0,\frac12}$ norm, and the  time-weighted function $\fg(t)$ depends only on the critical  $\dot{B}^{0,\frac12}$ norm of the solutions.

\item[(2)]  We get, by  applying Osgood's lemma to \eqref{eq:prop1.6}, that
\begin{equation}\notag
  \begin{aligned}
    \|&\fg(u-v) \|_{\widetilde L^\infty_t(X)}^2 + \|\fg\nabla_{\mathrm{h}} (u - v)\|_{\widetilde L^2_t(X)}^2\leq \exp \Bigl(- \Bigl(\ln \frac{1}{C \|u_0-v_0\|_{X}^2}\Bigr)^{\fg(t)}  \Bigr),
  \end{aligned}
\end{equation}
from which and the definition of $\fg(t)$ in \eqref{eq:prop1.6}, we infer
\begin{equation}\label{eq1.9}
  \|(u-v) \|_{\widetilde L^\infty_t(X)}^2 + \|\nabla_{\mathrm{h}} (u - v)\|_{\widetilde L^2_t(X)}^2 \leq e^{C {\rm B}(t)} \exp \Bigl(- \Bigl(\ln \frac{1}{C \|u_0-v_0\|_{X}^2}\Bigr)^{\fg(t)}  \Bigr).
\end{equation}
If ${\rm B}(t)$ is bounded and $\|u_0-v_0\|_{X} \leq C\|u_0-v_0\|_{\dot{B}^{0,\frac12}},$ which is sufficiently small, then the left-hand side of \eqref{eq1.9} is small enough as well.
In particular, the estimate \eqref{eq1.9} yields the uniqueness of solution of $(ANS)$ in the space  $\dot{B}^{0,\frac12}.$
  \end{itemize}
\end{rmk}

Finally to prove Theorem~\ref{thm2}, we decompose the difference $u^n(t) - u^\infty(t)$ into high-frequency and low-frequency parts.  Since Proposition \ref{prop1.5} ensures boundedness of $\{ u^n(t) \}$ in $\dot{B}^{0,\frac12}_\rb$ with $\lim_{j \to \infty} b_j = \infty,$ the high-frequency part is automatically small in $\dot{B}^{0,\frac12}$. For the low-frequency part, Proposition \ref{prop1.6} yields continuity in a weaker norm, and also strong norm due to the low frequency condition.

We now admit Propositions \ref{prop1.4}-\ref{prop1.6} for the time being and continue our  proof of Theorem~\ref{thm2}.

\begin{proof}[Proof of Theorem~\ref{thm2}]
  As $u^\infty(t)$ is a solution of $(ANS)$ with initial data $u_0^\infty \in \dot{B}^{0,\frac12}$ on $[0,T]$, we deduce from Theorem~1 of \cite{Paicu} that there exists some positive constant $\rC_1$ such that
\begin{equation}\label{eq6.1}
\begin{split}
\|u^\infty\|_{L^\infty_T(\dot{B}^{0,\frac12})} + \|\nabla_{\mathrm{h}} u^\infty\|_{L^2_T(\dot{B}^{0,\frac12})}
\leq \rC_1.
\end{split}
\end{equation}
Since $u_0^n$ converges to $u_0^\infty$ in $\dot{B}^{0,\frac12}$, the sequence $\{ u^n_0 \}_{n \in \mathbb{N}}$ is uniformly bounded in $\dot{B}^{0,\frac12}$. Therefore, by taking $\rC_1$ large enough, we may assume without loss of generality that $\rC_1 \geq \sup_{n \in \mathbb{N}} \|u^n_0\|_{\dot{B}^{0,\frac12}}$.

On the other hand, corresponding to the initial data $u^n_0 \in \dot{B}^{0,\frac12}$, we deduce from Theorem~1 of \cite{Paicu} that the system $(ANS)$ has a unique solution $u^n(t)$ on $[0, T^*_n)$ for some maximal existence time $T^*_n$. Then, for every $n \in \mathbb{N}$, we introduce the following timespan:
\begin{equation}\label{eq:def Tn d}
\begin{split}
T_n^\star \eqdefa \sup \Bigl\{ t \in (0, \min(T, T^*_n)) \ : \ & \|u^n\|_{L^\infty_t(\dot{B}^{0,\frac12})} + \|\nabla_{\mathrm{h}} u^n\|_{L^2_t(\dot{B}^{0,\frac12})} \leq 2\rC_1 \Bigr\}.
\end{split}
\end{equation}
It is easy to observe that $T^\star_n > 0$ because $\rC_1 \geq \sup_{n \in \mathbb{N}} \|u^n_0\|_{\dot{B}^{0,\frac12}}$.

Thanks to \eqref{eq:def Tn d}, we have
\begin{equation}\label{eq6.3}
  \int_0^{T_n^\star} \Bigl(1 + \|u^n(t)\|_{\dot{B}^{0,\frac12}}^2\Bigr) \Bigl(1 + \|\nabla_{\mathrm{h}} u^n(t)\|_{\dot{B}^{0,\frac12}}^2\Bigr) \, dt \leq {(1 + 4\rC_1^2)(T + 4\rC_1^2)}.
\end{equation}

Since $u^n_0 \to u^\infty_0$ in $\dot{B}^{0,\frac12}$, we get, by applying Proposition \ref{prop1.4}, 
that there exists arithmetic function $\rb=\{b_j\}_{j\in\mathbb{Z}}$ such that $\lim_{j \to \infty} b_j = +\infty$, each $u^n_0$ for $n \in \mathbb{N}$ belongs to $\dot{B}^{0,\frac12}_\rb$, and the sequence $\{ u^n_0 \}_{n\in\N}$ converges to $u^\infty_0$ in $\dot{B}^{0,\frac12}_\rb$ as $n \to \infty$.

We decompose the difference $u^n - u^\infty$ as
\begin{equation} \label{S6eq1}
u^n - u^\infty = \dot{S}_J^{\mathrm{v}} (u^n - u^\infty) + (\mathrm{Id} - \dot{S}_J^{\mathrm{v}})(u^n - u^\infty),
\end{equation}
where $J \in \mathbb{N}$ is a sufficiently large integer to be chosen later on.

Below we consider the low- and high-frequency parts separately.

\smallskip
\noindent $\Large\bullet$ Estimates for the high-frequency part $(\mathrm{Id} - \dot{S}_J^{\mathrm{v}})(u^n - u^\infty)$.\smallskip

Since $u^n_0 \to u^\infty_0$ in $\dot{B}^{0,\frac12}_\rb$ as $n \to \infty$, the sequence of initial data $u_0^n$ is uniformly bounded in $\dot{B}^{0,\frac12}_\rb$. Let us denote
\[
\rC_2 \eqdefa \sup_{n \in \mathbb{N}} \|u^n_0\|_{\dot{B}_\rb^{0,\frac12}}.
\]

Thanks to \eqref{eq:prop1.5} and \eqref{eq6.3}, we get, by applying Proposition \ref{prop1.5} for $u^n$ on {$[0, T_n^\star]$,} that
\begin{align}\label{S6eq2}
  \|u^n\|_{L^\infty_{T_n^\star}(\dot{B}^{0,\frac12}_b)} + \|\nabla_{\mathrm{h}} u^n\|_{L^2_{T_n^\star}(\dot{B}^{0,\frac12}_b)}
\leq C e^{C {(1 + 4\rC_1^2)(T + 4\rC_1^2)}} \|u_0^n\|_{\dot{B}^{0,\frac12}_\rb}^2.
\end{align}

Observing that for any function $f \in \dot{B}_\rb^{0,\frac12}$, the monotonicity of $b_j$ implies
\[
\| (\mathrm{Id} - \dot{S}_J^{\mathrm{v}}) f \|_{\dot{B}^{0,\frac12}} = \sum_{j=J}^\infty b_j^{-1} \bigl( b_j 2^{j/2} \| \dot{\Delta}_j^{\mathrm{v}} f \|_{L^2} \bigr) \leq b_J^{-1} \| f \|_{\dot{B}^{0,\frac12}_\rb},
\]
from which, and \eqref{S6eq2}, we infer
\begin{equation}\notag
  \begin{aligned}
    &\| (\mathrm{Id} - \dot{S}_J^{\mathrm{v}})(u^n - u^\infty) \|_{L^\infty_{T_n^\star}(\dot{B}^{0,\frac12})}
      + \| (\mathrm{Id} - \dot{S}_J^{\mathrm{v}}) \nabla_{\mathrm{h}} (u^n - u^\infty) \|_{L^2_{T_n^\star}(\dot{B}^{0,\frac12})} \\
    &\leq b_J^{-1} \Bigl( \| (u^n, u^\infty) \|_{L^\infty_{T_n^\star}(\dot{B}^{0,\frac12}_\rb)}
      + \| \nabla_{\mathrm{h}} (u^n, u^\infty) \|_{L^2_{T_n^\star}(\dot{B}^{0,\frac12}_\rb)} \Bigr) \\
    &\leq b_J^{-1} \, 2C \rC_2 e^{C {(1 + 4\rC_1^2)(T + 4\rC_1^2)}}.
  \end{aligned}
\end{equation}

Given any $\varepsilon > 0$, we may choose $J = J_\varepsilon$ sufficiently large such that
\begin{equation}\label{eq6.4}
  b_{J_\varepsilon}^{-1} \, 2C \rC_2 e^{C {(1 + 4\rC_1^2)(T + 4\rC_1^2)}} \leq \frac{\varepsilon}{2},
\end{equation}
and then we conclude that we have (uniformly with respect to $n$), 
\begin{equation}\label{eq6.5}
  \| (\mathrm{Id} - \dot{S}_{J_\varepsilon}^{\mathrm{v}})(u^n - u^\infty) \|_{L^\infty_{T_n^\star}(\dot{B}^{0,\frac12})}
  + \| (\mathrm{Id} - \dot{S}_{J_\varepsilon}^{\mathrm{v}}) \nabla_{\mathrm{h}} (u^n - u^\infty) \|_{L^2_{T_n^\star}(\dot{B}^{0,\frac12})} \leq \frac{\varepsilon}{2}.
\end{equation}

\noindent ${\Large\bullet}$ Estimates for the low-frequency part $\dot{S}_J^{\mathrm{v}}(u^n - u^\infty)$.

For the low-frequency part, we first observe from \eqref{S1eq1a}  that
\[
\| \dot{S}_J^{\mathrm{v}} f \|_{\dot{B}^{0,\frac12}} \leq \sum_{j\leq J-1} 2^{j/2}  \| \dot{\Delta}_j^{\mathrm{v}} f \|_{L^2} 
\leq \|f\|_{\dBl^{0,\frac12}} + \Bigl( \sum_{j=0}^{J-1} 2^{2j} \Bigr)^{1/2} \| f\|_{\dHh^{0,-\frac12}} \leq 2^J \| f \|_{X},
\]
which implies that
\begin{equation}\label{S6eq3}
\begin{aligned}
\| \dot{S}_J^{\mathrm{v}} (u^n - u^\infty) \|_{L^\infty_{T_n^\star}(\dot{B}^{0,\frac12})}
&+ \| \dot{S}_J^{\mathrm{v}} \nabla_{\mathrm{h}} (u^n - u^\infty) \|_{L^2_{T_n^\star}(\dot{B}^{0,\frac12})} \\
&\leq 2^J \Bigl( \| u^n - u^\infty \|_{L^\infty_{T_n^\star}(X)}
+ \| \nabla_{\mathrm{h}} (u^n - u^\infty) \|_{L^2_{T_n^\star}(X)} \Bigr).
\end{aligned}
\end{equation}

By virtue of \eqref{eq6.1} and \eqref{eq6.3}, for $t \leq T_n^\star$, we have
\begin{align*}
{\rm B}_n(t) \eqdefa &\int_0^t \Bigl( 1 + \| (u^n, u^\infty)(\tau) \|_{\dot{B}^{0,\frac12}}^2 \Bigr)
 \Bigl( 1 + \| \nabla_{\mathrm{h}} (u^n, u^\infty)(\tau) \|_{\dot{B}^{0,\frac12}}^2 \Bigr) \, d\tau\\
 \leq& {\left(1 + 9\rC_1^2\right)\left(T + 9\rC_1^2\right).}
\end{align*}
Then we deduce from Osgood's lemma and \eqref{eq1.9}
 that if $\| u_0^n - u^\infty_0 \|_{X}$ is sufficiently small, then
\begin{equation}\label{S6eq4}
\begin{aligned}
\| (u^n - &u^\infty) \|_{ L^\infty_{T_n^\star}(X)}^2 + \| \nabla_{\mathrm{h}} (u^n - u^\infty) \|_{ L^2_{T_n^\star}(X)}^2 \\
&\leq Ce^{C {\rm B}_n({T_n^\star})} \exp \Bigl(- \Bigl(\ln \frac{1}{C \|u_0^n - u^\infty_0\|_{X}^2}\Bigr)^{\exp (-C {\rm B}_n({T_n^\star}))}  \Bigr).
\end{aligned}
\end{equation}

Since
\[
\| u_0^n - u^\infty_0 \|_{X} \leq \| u_0^n - u^\infty_0 \|_{\dot{B}^{0,\frac12}} \to 0 \quad \text{as } n \to \infty,
\]
we deduce from \eqref{S6eq3} and \eqref{S6eq4} that for  $J = J_\varepsilon$, which is determined by \eqref{eq6.4}, there exists a large enough integer $N_\varepsilon \in \mathbb{N}$ such that for all $n > N_\varepsilon$,
\begin{equation}\label{eq6.6}
\| \dot{S}_{J_\varepsilon}^{\mathrm{v}} (u^n - u^\infty) \|_{L^\infty_{T_n^\star}(\dot{B}^{0,\frac12})}
+ \| \dot{S}_{J_\varepsilon}^{\mathrm{v}} \nabla_{\mathrm{h}} (u^n - u^\infty) \|_{L^2_{T_n^\star}(\dot{B}^{0,\frac12})} \leq \frac{\varepsilon}{2}.
\end{equation}

By combining \eqref{eq6.5} with \eqref{eq6.6}, we conclude that for all $n > N_\varepsilon$,
\begin{equation}\label{eq6.7}
\| (u^n - u^\infty) \|_{L^\infty_{T_n^\star}(\dot{B}^{0,\frac12})}
+ \| \nabla_{\mathrm{h}} (u^n - u^\infty) \|_{L^2_{T_n^\star}(\dot{B}^{0,\frac12})} \leq \varepsilon,
\end{equation}
which together with \eqref{eq6.1} ensures
\begin{equation}\label{eq6.8}
\| u^n \|_{L^\infty_{T_n^\star}(\dot{B}^{0,\frac12})} + \| \nabla_{\mathrm{h}} u^n \|_{L^2_{T_n^\star}(\dot{B}^{0,\frac12})} \leq \rC_1 + \varepsilon,\quad \mbox{for all} \ n > N_\varepsilon.
\end{equation}

By taking $\varepsilon$ to be so small that $\varepsilon < \frac{\rC_1}{2} $ and using \eqref{eq:def Tn d} and \eqref{eq6.8}, a standard continuity argument ensures that $T^*_n > T_n^\star = T$ for all $n > N_{\varepsilon_0}$. This means that the lifespans of the solutions $u^n$ are all larger than $T$ for $n > N_{\varepsilon_0}$. Furthermore, \eqref{S1eq4} follows from \eqref{eq6.7}. This completes the proof of Theorem~\ref{thm2}.
\end{proof}
\smallskip

\section{The improved convergence for the initial data sequence}\label{Sect2}

In this section, we first show that given any initial data sequence $\{u^n_0\}_{n\in\N}$ which converges to $u_0^\infty$ in $\dot{B}^{0,\frac12}$, one can construct a special choice of the arithmetic function $\rb$ such that $u_0^n$ belongs to $\dot{B}^{0,\frac12}_\rb$ for {all} $n \in \mathbb{N}$. Then we present the proof of Proposition \ref{prop1.4}.

Since the Fourier weight only makes sense for high frequencies, it is sufficient to construct a non-decreasing sequence in $\N,$ instead of a function on $\mathbb{Z}$. We begin with the simple case that given any $\ell^1(\mathbb{N})$ sequence, it belongs to some weighted $\ell^1(\N)$ space.

\begin{lem}\label{lem2.1}
{\sl Let ${\rm d} = \{ d_j \}_{j \in \mathbb{N}}$ be a sequence such that $d_j \geq 0$ and $\sum_{j=0}^\infty d_j < \infty$. Then there exists a non-decreasing positive sequence $\rb = \{ b_j \}_{j \in \mathbb{N}}$ such that
\beq\label{S2eqa}
\lim_{j \to \infty} b_j = \infty \quad \text{and} \quad \sum_{j=0}^\infty b_j d_j \leq 2 \sum_{j=0}^\infty d_j.
\eeq}
\end{lem}

\begin{proof}
We define the sequence $\{ J_k \}_{k \in \mathbb{N}}$ by induction:
\[
J_0 \eqdefa 0, \qquad
J_{k+1} \eqdefa \min \Bigl\{ J \in \mathbb{N} : J > J_k,\; \sum_{j=J}^\infty d_j \leq 4^{-(k+1)} \sum_{j=0}^\infty d_j \Bigr\},
\]
which implies that for all $k$,
\begin{equation}\label{eq2.1}
\sum_{j=J_k}^\infty d_j \leq 4^{-k} \sum_{j=0}^\infty d_j.
\end{equation}
Then we define $b_j = k+1$ for $J_k \leq j < J_{k+1}$.

Now we compute
\[
\sum_{j=0}^\infty b_j d_j \leq \sum_{k=0}^\infty (k+1) \sum_{j=J_k}^{J_{k+1}-1} d_j
\leq \left( \sum_{k=0}^\infty (k+1) 4^{-k} \right) \sum_{j=0}^\infty d_j \leq 2 \sum_{j=0}^\infty d_j,
\]
where we used \eqref{eq2.1} and $\sum_{k=0}^\infty (k+1) 4^{-k} \leq \sum_{k=0}^\infty 2^{-k} \leq 2$. This finishes the proof of \eqref{S2eqa}.
\end{proof}

Next, we turn to the case with a sequence of $\ell^1(\mathbb{N})$ elements:

\begin{lem}\label{lem2.2}
{\sl For each $n \in \mathbb{N},$ let ${\rm d}^n = \{ d_j^n \}_{j \in \mathbb{N}}$  be sequences of positive numbers satisfying
\[
\lim_{n \to \infty} \sum_{j=0}^\infty d_j^n = 0.
\]
Then there exists a non-decreasing positive sequence $\rb = \{ b_j \}_{j \in \mathbb{N}}$ such that
\beq
\label{S2eqb}
\lim_{j \to \infty} b_j = \infty \quad \text{and} \quad \lim_{n \to \infty} \sum_{j=0}^\infty b_j d_j^n = 0.
\eeq}
\end{lem}

\begin{proof}
First, we define the sequence $\{ N_k \}_{k \in \mathbb{N}}$ by induction:
\[
N_0 \eqdefa 0, \qquad
N_{k+1} \eqdefa \min \Bigl\{ N \in \mathbb{N} : N > N_k,\; \sum_{j=0}^\infty d_j^n \leq 4^{-(k+1)},\ \forall n > N \Bigr\},
\]
where $N_k$ is well-defined since $\lim_{n \to \infty} \sum_{j=0}^\infty d_j^n = 0$. In particular, for all $k \geq 1$, the $N_k$ defined above satisfies
\begin{equation}\label{eq2.2}
\sum_{j=0}^\infty d_j^n \leq 4^{-k}, \qquad \forall n > N_k.
\end{equation}

Next  we define another sequence $\{ J_k \}_{k \in \mathbb{N}}$ by induction:
\[
J_0 \eqdefa 0, \qquad
J_{k+1}\eqdefa \min \Bigl\{ J \in \mathbb{N} : J > J_k,\; \sum_{j=J}^\infty d_j^n \leq 4^{-(k+1)},\ \forall n \leq N_{k+1} \Bigr\},
\]
which implies that for all $k$,
\begin{equation}\label{eq2.3}
\sum_{j=J_k}^\infty d_j^n \leq 4^{-k}, \qquad \forall n \leq N_k.
\end{equation}
Again, we define $b_j = k+1$ for $J_k \leq j < J_{k+1}$.

For any given $n$, there exists some $k_0$ such that $N_{k_0} < n \leq N_{k_0+1}$. We compute
\[
\sum_{j=0}^\infty b_j d_j^n \leq \sum_{k=0}^\infty (k+1) \sum_{j=J_k}^{J_{k+1}-1} d_j^n
\leq \left( \sum_{k=0}^{k_0} (k+1) \right) \sum_{j=0}^\infty d_j^n
+ \sum_{k=k_0+1}^\infty (k+1) \sum_{j=J_k}^\infty d_j^n.
\]
For $n > N_{k_0}$, we get, applying \eqref{eq2.2}, that
\[
\left( \sum_{k=0}^{k_0} (k+1) \right) \sum_{j=0}^\infty d_j^n \leq k_0(k_0+1) 4^{-k_0}.
\]
Whereas for all $k \geq k_0+1$, we have ${N_k \geq N_{k_0+1} \geq n}$, so that we deduce from \eqref{eq2.3}  that
\[
\sum_{k=k_0+1}^\infty (k+1) \sum_{j=J_k}^\infty d_j^n \leq \sum_{k=k_0+1}^\infty (k+1) 4^{-k} \leq 2^{-k_0}.
\]

By summarizing the above estimates, we conclude that
\[
\sum_{j=0}^\infty b_j d_j^n \leq k_0(k_0+1) 4^{-k_0} + 2^{-k_0}, \qquad \forall n \in (N_{k_0}, N_{k_0+1}].
\]
Noting that $N_{k_0} \to \infty$ as $k_0 \to \infty$, we obtain
\[
\lim_{n \to \infty} \sum_{j=0}^\infty b_j d_j^n \leq \lim_{k_0 \to \infty} \bigl( k_0(k_0+1) 4^{-k_0} + 2^{-k_0} \bigr) = 0,
\]
which finishes the proof of Lemma \ref{lem2.2}.
\end{proof}

Now we present the proof of Proposition \ref{prop1.4}.

\begin{proof}[Proof of Proposition \ref{prop1.4}]
First, we denote $d_j \eqdefa 2^{\frac{j}{2}} \| \dot{\Delta}_j^{\mathrm{v}} u_0^\infty \|_{L^2}$, which belongs to $\ell^1(\mathbb{Z})$ because $u_0^\infty \in \dot{B}^{0,\frac12}$. Then we get, by applying Lemma \ref{lem2.1} to  $\{d_j\}_{j \in \mathbb{N}}$,  that there exists a non-decreasing positive sequence $\rb^1 = \{ b_j^1 \}_{j \in \mathbb{N}}$ such that
\begin{equation}\label{eq2.6}
\lim_{j \to \infty} b_j^1 = \infty \quad \text{and} \quad
\sum_{j=0}^\infty b_j^1 d_j \leq 2 \sum_{j=0}^\infty d_j \leq 2 \| u_0^\infty \|_{\dot{B}^{0,\frac12}}.
\end{equation}

Let us denote $d_j^n \eqdefa 2^{\frac{j}{2}} \| \dot{\Delta}_j^{\mathrm{v}} (u_0^n - u_0^\infty) \|_{L^2}$, which satisfies
\[
\lim_{n \to \infty} \sum_{j=0}^\infty d_j^n \leq \lim_{n \to \infty} \| u_0^n - u_0^\infty \|_{\dot{B}^{0,\frac12}} = 0.
\]
We then get, by apply Lemma \ref{lem2.2}, that there exists  a non-decreasing positive sequence $\rb^2 = \{ b_j^2 \}_{j \in \mathbb{N}}$ such that
\begin{equation}\label{eq2.7}
  \lim_{j \to \infty} b_j^2 = \infty \quad \text{and} \quad
  \lim_{n \to \infty} \sum_{j=0}^\infty b_j^2 d_j^n = 0.
\end{equation}

Finally, we define the arithmetic function $b_j = \min( b_j^1, b_j^2 )$ for $j \geq 0$, and extend it as $b_j = b_0$ for all $j < 0$. It is clear that $\rb$ is also a non-decreasing arithmetic function and $\lim_{j \to \infty} b_j = \infty$. Then it follows from  the estimates \eqref{eq2.6}, \eqref{eq2.7} and $u^n_0 \to u_0^\infty$ in $\dot{B}^{0,\frac12}$ that
\begin{equation}
  \|u^\infty_0\|_{\dot{B}^{0,\frac12}_\rb} \leq b_0 \sum_{j<0} 2^{\frac{j}{2}} \| \dot{\Delta}_j^{\mathrm{v}} u^\infty_0 \|_{L^2} + \sum_{j\geq 0} b^1_j d_j \leq (b_0 + 2) \| u^\infty_0 \|_{\dot{B}^{0,\frac12}},
\end{equation}
and
\begin{equation}
  \begin{aligned}
    \lim_{n \to \infty} \| u^n_0 - u^\infty_0 \|_{\dot{B}^{0,\frac12}_\rb}
    &\leq b_0 \lim_{n \to \infty} \sum_{j<0} 2^{\frac{j}{2}} \| \dot{\Delta}_j^{\mathrm{v}} (u^n_0 - u^\infty_0) \|_{L^2} + \lim_{n \to \infty} \sum_{j\geq 0} b_j d_j^n \\
    &\leq b_0 \lim_{n \to \infty} \| u^n_0 - u^\infty_0 \|_{\dot{B}^{0,\frac12}} + \lim_{n \to \infty} \sum_{j\geq 0} b_j^2 d_j^n = 0.
  \end{aligned}
\end{equation}
The boundedness and convergence of $\{ u_0^n \}_{n \in \mathbb{N}}$ in $\dot{B}^{0,\frac12}_\rb$ are direct consequences of the above construction. This finishes the proof of Proposition \ref{prop1.4}.
\end{proof}

\smallskip

\section{Propagation of the slightly higher regularity}\label{Sect3}

In this section, we aim to present the proof of Proposition \ref{prop1.5}, namely, the propagation of the $\dB^{0,\f12}_\rb$ regularity of the initial data.  The key ingredient lies in 
the following time-weighted energy estimate:

\begin{prop}\label{prop3.1}
{\sl Let $u$ be a smooth enough solution of  $(ANS)$ on $[0,T].$ Then there exists some { universal large} constant $C_0>0$ such that
\begin{equation}\label{eq:prop3.1}
\begin{aligned}
&\|\frak{f}u\|_{\widetilde L^\infty_T(B^{0,\f12}_\rb)} + \| \ff\nabla_\h u\|_{\widetilde L^2_T(B^{0,\f12}_\rb)}
\leq C \|u_0\|_{B^{0,\f12}_\rb}\with \frak{f}(t)\eqdefa e^{-C_0 {\rm A}(t)},
\end{aligned}
\end{equation}
where ${\rm A}(t)$ is defined as in \eqref{eq:prop1.5}.
}
\end{prop}

\begin{proof} We get, by
first  applying $\dDv_j$ to the momentum equations of $(ANS)$ and then taking the $L^2$ inner product of the resulting equation with $\dDv_j u$, that
\begin{eqnarray}\label{eq3.2}
\frac12 \frac{d}{dt} \|\dDv_j u(t)\|_{L^2}^2 + \|\nabla_\h \dDv_j u(t)\|_{L^2}^2 = - \bigl( \dDv_j (u\cdot\nabla u) \mid \dDv_j u \bigr).
\end{eqnarray}

We remark that to overcome the difficulty that one cannot use a Gronwall type argument in the framework of Chemin-Lerner type norms, below we shall use a time-weighted Chemin-Lerner type norm (see \cite{PZ11}). Indeed in view of the definition of the $\ff(t)$ given by \eqref{eq:prop3.1}, one has 
\begin{align*}
\int_0^t \ff^2(\tau) \frac{d}{d\tau} \|\dDv_j u(\tau)\|_{L^2}^2 \, d\tau =& \ff^2(t) \|\dDv_j u(t)\|_{L^2}^2 - \|\dDv_j u_0\|_{L^2}^2 \\
&+ 2C_0 \int_0^t \A'(\tau) \ff^2(\tau) \|\dDv_j u(\tau)\|_{L^2}^2 \, d\tau,
\end{align*}
so that we get, by first multiplying \eqref{eq3.2} by {$2\ff^2(t)$} and then integrating the resulting equation over $[0,t],$ that for all $t \in [0,T]$,
\begin{equation}\label{eq3.3}
\begin{aligned}
\ff^2(t)& \|\dDv_j u(t)\|_{L^2}^2 + 2C_0 \int_0^t \A'(\tau) \ff^2(\tau) \|\dDv_j u(\tau)\|_{L^2}^2 \, d\tau+ 2 \int_0^t \ff^2(\tau) \|\nabla_\h \dDv_j u(\tau)\|_{L^2}^2 \, d\tau \\
\leq& \|\dDv_j u_0\|_{L^2}^2 + 2 \int_0^t \ff^2(\tau)\bigl| \bigl( \dDv_j (u\cdot\nabla u) \mid \dDv_j u \bigr) \bigr| \, d\tau.
\end{aligned}
\end{equation}

Let us turn to the estimate of the nonlinear terms in \eqref{eq3.3}. Due to the anisotropic dissipation in system $(ANS),$ we write $u\cdot \nabla u = u^{\rm h} \cdot \nabla_\h u + u^3 \partial_3 u$ and handle these two terms separately.\smallskip

\noindent $\bullet$ Estimate of $\big| \bigl( \dDv_j (u^{\rm h} \cdot \nabla_\h u) \mid \dDv_j u \bigr) \big|$.\smallskip

By applying Bony's decomposition \eqref{homo bony v} to $u^{\rm h} \cdot \nabla_\h u$ in the vertical variable, we write
\[
u^{\rm h} \cdot \nabla_\h u = \dTv_{u^{\rm h}} \nabla_\h u + \dTv_{\nabla_\h u} u^{\rm h} + \dRv(u^{\rm h}, \nabla_\h u).
\]

Considering the support properties of the Fourier transform of the terms in {$\dTv_{u^{\rm h}} \nabla_\h u$,} we find
\[
\big| \bigl( \dDv_j (\dTv_{u^{\rm h}} \nabla_\h u) \mid \dDv_j u \bigr) \big|
\lesssim \sum_{|k-j|\leq 4} \| \dSv_{k-1} u^{\rm h} \|_{L^4_{\rm h}(L^\infty_{\rm v})} \| \dDv_k \nabla_\h u \|_{L^2} \| \dDv_j u \|_{L^4_{\rm h}(L^2_{\rm v})}.
\]
Yet it follows from Lemma \ref{lemBern}, the 2-D interpolation inequality:
\begin{equation}\label{interpolation inequality}
  \|a\|_{L^4(\R^2)}\lesssim \|a\|_{L^2(\R^2)}^{\f12}\|\na_\h a\|_{L^2(\R^2)}^{\f12},
\end{equation} 
 and the definition of $A(t)$ that
\begin{equation}\label{eq3.4a}
\begin{aligned}
\| \dSv_{k-1} u^{\rm h}(t) \|_{L^4_{\rm h}(L^\infty_{\rm v})}
&\lesssim \sum_{\ell\leq k-2} 2^{\frac{\ell}{2}} \| \dDv_\ell u^{\rm h}(t) \|_{L^4_{\rm h}(L^2_{\rm v})} \\
&\lesssim \sum_{\ell\leq k-2} 2^{\frac{\ell}{2}} \| \dDv_\ell u^{\rm h}(t) \|_{L^2}^{\frac12} \| \nabla_\h \dDv_\ell u^{\rm h}(t) \|_{L^2}^{\frac12} \\
&\lesssim \| u^{\rm h}(t) \|_{\dB^{0,\f12}}^{\frac12} \| \nabla_\h u^{\rm h}(t) \|_{\dB^{0,\f12}}^{\frac12} \lesssim \bigl( \A'(t) \bigr)^{\frac14},
\end{aligned}
\end{equation}
from which we infer
\begin{align*}
\int_0^t &\ff^2(\tau)\big| \bigl( \dDv_j (\dTv_{u^{\rm h}} \nabla_\h u) \mid \dDv_j u \bigr) \big| \, d\tau \\
\lesssim &\sum_{|k-j|\leq 4} \int_0^t \ff^2(\tau)\bigl( \A'(\tau) \bigr)^{\frac14} \| \dDv_k \nabla_\h u(\tau) \|_{L^2} \| \dDv_j u(\tau) \|_{L^2}^{\frac12} \| \nabla_\h \dDv_j u(\tau) \|_{L^2}^{\frac12} \, d\tau \\
\lesssim &\sum_{|k-j|\leq 4} \| \ff \dDv_k \nabla_\h u \|_{L^2_t(L^2)} \| \bigl( \A' \bigr)^{\frac12}
 \ff \dDv_j u\|_{L^2_t(L^2)}^{\frac12} \|\ff\dDv_j \nabla_\h u \|_{L^2_t(L^2)}^{\frac12} \\
\lesssim &\Bigl( \sum_{|k-j|\leq 4} d_k b_k^{-1} 2^{-\frac{k}{2}} \Bigr) d_j b_j^{-1} 2^{-\frac{j}{2}} \| \ff \nabla_\h u \|_{\widetilde L^2_t(\dB^{0,\f12}_\rb)}^{\frac32}  \| \bigl( \A'\bigr)^{\frac12} \ff u\|_{\widetilde L^2_t(\dB^{0,\f12}_\rb)}^{\frac12}.
\end{align*}
Due to the fact that $b_j \leq b_{j+1} \leq 2b_j$, we have $b_k^{-1} \lesssim b_j^{-1}$ for $|k-j|\leq 4$. As a consequence, we deduce that
\begin{equation}\label{eq3.4}
\begin{aligned}
&\int_0^t \ff^2(\tau) \big| \bigl( \dDv_j (\dTv_{u^{\rm h}} \nabla_\h u) \mid \dDv_j u \bigr) \big| \, d\tau \lesssim d_j^2 b_j^{-2} 2^{-j} \| \ff \nabla_\h u \|_{\widetilde L^2_t(\dB^{0,\f12}_\rb)}^{\frac32} \| \bigl( \A' \bigr)^{\frac12} \ff u \|_{\widetilde L^2_t(\dB^{0,\f12}_\rb)}^{\frac12}.
\end{aligned}
\end{equation}

Similarly, it follows from Lemma \ref{lemBern} and the definition of $A(t)$ that
\begin{equation}\label{eq3.5}
\begin{split}
\| \dSv_{k-1} \nabla_\h u(t) \|_{L^2_{\rm h}(L^\infty_{\rm v})}
&\lesssim \sum_{\ell\leq k-2} 2^{\frac{\ell}{2}} \| \dDv_\ell \nabla_\h u(t) \|_{L^2} \\
&\lesssim \| \nabla_\h u(t) \|_{\dB^{0,\f12}} \lesssim \bigl( \A'(t) \bigr)^{\frac12},
\end{split}
\end{equation}
from which, we infer
\begin{align*}
\int_0^t &\ff^2(\tau)\big| \bigl( \dDv_j  (\dTv_{\nabla_\h u} u^{\rm h}) \mid \dDv_j u \bigr) \big| \, d\tau \\
\lesssim &\sum_{|k-j|\leq 4} \int_0^t \ff^2(\tau)\| \dSv_{k-1} \nabla_\h u(\tau) \|_{L^2_{\rm h}(L^\infty_{\rm v})} \| \dDv_k u^{\rm h}(\tau) \|_{L^4_{\rm h}(L^2_{\rm v})} \| \dDv_j u(\tau) \|_{L^4_{\rm h}(L^2_{\rm v})} \, d\tau \\
\lesssim &\sum_{|k-j|\leq 4} \| \bigl( \A'\bigr)^{\frac12} \ff \dDv_k u \|_{L^2_t(L^2)}^{\frac12} \|\ff \dDv_k \nabla_\h u \|_{L^2_t(L^2)}^{\frac12}  \| \bigl( \A' \bigr)^{\frac12} \ff \dDv_j u\|_{L^2_t(L^2)}^{\frac12} \|\ff \dDv_j \nabla_\h u \|_{L^2_t(L^2)}^{\frac12} \\
\lesssim &\Bigl( \sum_{|k-j|\leq 4} d_k b_k^{-1} 2^{-\frac{k}{2}} \Bigr) d_j b_j^{-1} 2^{-\frac{j}{2}} \| \ff\nabla_\h u \|_{\widetilde L^2_t(\dB^{0,\f12}_\rb)} \| \bigl( \A'\bigr)^{\frac12} \ff u \|_{\widetilde L^2_t(\dB^{0,\f12}_\rb)}.
\end{align*}
This gives rise to
\begin{equation} \label{eq3.6}
\begin{aligned}
\int_0^t &\ff^2(\tau) \big| \bigl( \dDv_j (\dTv_{\nabla_\h u} u^{\rm h}) \mid \dDv_j u \bigr) \big| \, d\tau 
\lesssim  d_j^2 b_j^{-2} 2^{-j} \|\ff\nabla_\h u \|_{\widetilde L^2_t(\dB^{0,\f12}_\rb)} \| \bigl( \A'\bigr)^{\frac12} \ff u\|_{\widetilde L^2_t(\dB^{0,\f12}_\rb)}.
\end{aligned}
\end{equation}

Finally, for the remainder term, we deduce from Lemma \ref{lemBern} that
\begin{align*}
\int_0^t &\ff^2(\tau) \big| \bigl( \dDv_j \dRv(u^{\rm h}, \nabla_\h u) \mid \dDv_j u \bigr) \big| \, d\tau \\
\lesssim & 2^{\frac{j}{2}} \sum_{k \geq j-3} \int_0^t \ff^2(\tau) \| \dDv_k u^{\rm h}(\tau) \|_{L^4_{\rm h}(L^2_{\rm v})} \| \wt{\D}^\v_k \nabla_\h u(\tau) \|_{L^2} \| \dDv_j u(\tau) \|_{L^4_{\rm h}(L^2_{\rm v})} \, d\tau.
\end{align*}
It follows from a similar derivation of \eqref{eq3.5} that
\begin{align*}
\| \wt{\D}^\v_k \nabla_\h u(t) \|_{L^2} \lesssim 2^{-\frac{k}{2}} \| \nabla_\h u(t) \|_{\dB^{0,\f12}} \lesssim 2^{-\frac{k}{2}} \bigl( \A'(t) \bigr)^{\frac12},
\end{align*}
from which, we infer
\begin{equation}\label{eq3.7}
\begin{aligned}
\int_0^t &\ff^2(\tau) \big| \bigl( \dDv_j \dRv(u^{\rm h}, \nabla_\h u) \mid \dDv_j u \bigr) \big| \, d\tau \\
\lesssim & d_j b_j^{-1} \sum_{k \geq j-3} d_k b_k^{-1} 2^{-k} \| \ff \nabla_\h u \|_{\widetilde L^2_t(\dB^{0,\f12}_\rb)} \| \bigl( \A'\bigr)^{\frac12} \ff u\|_{\widetilde L^2_t(\dB^{0,\f12}_\rb)} \\
\lesssim & d_j^2 b_j^{-2} 2^{-j} \|\ff \nabla_\h u\|_{\widetilde L^2_t(\dB^{0,\f12}_\rb)} \| \bigl( \A'\bigr)^{\frac12} \ff u\|_{\widetilde L^2_t(\dB^{0,\f12}_\rb)},
\end{aligned}
\end{equation}
where in the last step we used $b_k^{-1}\leq b_{j-3}^{-1}\leq 8 b_j^{-1} $ from the fact that $\{ b_k \}$ is non-decreasing and slowly increasing.

By summarizing the estimates \eqref{eq3.4}, \eqref{eq3.6}, and \eqref{eq3.7}, we achieve
\begin{equation}\label{eq3.8}
\begin{aligned}
\int_0^t &\ff^2(\tau)\big| \bigl( \dDv_j (u^{\rm h} \cdot \nabla_\h u) \mid \dDv_j u \bigr) \big| \, d\tau \\
\lesssim & d_j^2 b_j^{-2} 2^{-j} \Bigl( \|\ff \nabla_\h u\|_{\widetilde L^2_t(\dB^{0,\f12}_\rb)} \| \bigl( \A' \bigr)^{\frac12} \ff u \|_{\widetilde L^2_t(\dB^{0,\f12}_\rb)} + \| \ff \nabla_\h u \|_{\widetilde L^2_t(\dB^{0,\f12}_\rb)}^{\frac32} \| \bigl( \A'\bigr)^{\frac12} \ff u\|_{\widetilde L^2_t(\dB^{0,\f12}_\rb)}^{\frac12} \Bigr).
\end{aligned}
\end{equation}

\noindent $\bullet$ Estimate of $\big| \bigl( \dDv_j (u^3 \partial_3 u) \mid \dDv_j u \bigr) \big|$.

By applying Bony's decomposition \eqref{homo bony v} to $u^3 \partial_3 u$, we write
\[
u^3 \partial_3 u = \dTv_{u^3} \partial_3 u + \dTv_{\partial_3 u} u^3 + \dRv(u^3, \partial_3 u).
\]

We first get,  by using a standard commutator argument (see \cite{CDGG} for instance), that
\begin{equation}\label{eq3.8a}
  \begin{aligned}
\bigl( \dDv_j (T^{\rm v}_{u^3} \partial_3 u) \mid \dDv_j u \bigr)
=&\sum_{|k-j|\leq 4} \Bigl( \bigl( [\dDv_j; \dSv_{k-1} u^3] \dDv_k \partial_3 u \mid \Delta_j^{\rm v} u \bigr) \\
&\qquad\qquad + \bigl( (\dSv_{k-1} u^3 - \dSv_{j} u^3) \dDv_j \dDv_k \partial_3 u \mid \dDv_j u \bigr) \Bigr) \\
&+ \bigl( \dSv_{j} u^3 \dDv_j \partial_3 u \mid \dDv_j u \bigr) \\
\eqdefa & I_1 + I_2 + I_3.
\end{aligned}
\end{equation}

By applying Lemma 2.97 of \cite{BCD}, Lemma \ref{lemBern}, and using $\partial_3 u^3 = -\operatorname{div}_{\rm h} u^{\rm h}$, we deduce from \eqref{eq3.5} that
\begin{align*}
|I_1| &\lesssim \sum_{|k-j|\leq 4} 2^{-j} \| \dSv_{k-1} \partial_3 u^3 \|_{L^2_{\rm h}(L^\infty_{\rm v})} \| \dDv_k \partial_3 u \|_{L^4_{\rm h}(L^2_{\rm v})} \| \dDv_j u \|_{L^4_{\rm h}(L^2_{\rm v})} \\
&\lesssim \sum_{|k-j|\leq 4} 2^{k-j} \bigl( \A'(t) \bigr)^{\frac12} \| \dDv_k u \|_{L^4_{\rm h}(L^2_{\rm v})}\| \dDv_j u \|_{L^4_{\rm h}(L^2_{\rm v})}.
\end{align*}
The term $I_2$ shares the same estimate.

For $I_3$, by using integration by parts and $\partial_3 u^3 = -\operatorname{div}_{\rm h} u^{\rm h}$, we obtain
\begin{align*}
|I_3| \lesssim \| \dSv_j \partial_3 u^3 \|_{L^2_{\rm h}(L^\infty_{\rm v})} \| \dDv_j u \|_{L^4_{\rm h}(L^2_{\rm v})}^2 
\lesssim \bigl( \A'(t) \bigr)^{\frac12} \| \dDv_j u \|_{L^4_{\rm h}(L^2_{\rm v})}^2.
\end{align*}

As a consequence, it comes out
\begin{align*}
\int_0^t &\ff^2(\tau) \big| \bigl( \dDv_j (T^{\rm v}_{u^3} \partial_3 u) \mid \dDv_j u \bigr) \big| \, d\tau \\
\lesssim &\sum_{|k-j|\leq 4} \int_0^t \ff^2(\tau) \bigl( A'(\tau) \bigr)^{\frac12} \| \dDv_k u(\tau) \|_{L^2}^{\frac12} \| \nabla_{\rm h} \dDv_k u(\tau) \|_{L^2}^{\frac12}\| \dDv_j u(\tau) \|_{L^2}^{\frac12} \| \nabla_{\rm h} \dDv_j u(\tau) \|_{L^2}^{\frac12} \, d\tau \\
\lesssim &\Bigl( \sum_{|k-j|\leq 4} d_k b_k^{-1} 2^{-\frac{k}{2}} \Bigr) d_j b_j^{-1} 2^{-\frac{j}{2}} \| \ff \nabla_{\rm h} u\|_{\widetilde L^2_t(\dB^{0,\f12}_\rb)} \| \bigl( \A' \bigr)^{\frac12} \ff u\|_{\widetilde L^2_t(\dB^{0,\f12}_\rb)},
\end{align*}
which implies
\begin{equation}\label{eq3.9}
\begin{aligned}
\int_0^t &\ff^2(\tau) \big| \bigl( \dDv_j (\dTv_{u^3} \partial_3 u) \mid \dDv_j u \bigr) \big| \, d\tau \lesssim d_j^2 b_j^{-2} 2^{-j} \|\ff \nabla_{\rm h} u\|_{\widetilde L^2_t(\dB^{0,\f12}_\rb)} \| \bigl( \A'\bigr)^{\frac12} \ff u \|_{\widetilde L^2_t(\dB^{0,\f12}_\rb)}.
\end{aligned}
\end{equation}

Next, we compute
\begin{align*}
\int_0^t &\ff^2(\tau) \big| \bigl( \dDv_j (\dTv_{\partial_3 u} u^3) \mid \dDv_j u \bigr) \big| \, d\tau \\
\lesssim &\sum_{|k-j|\leq 4} \int_0^t \ff^2(\tau) \| \dSv_{k-1} \partial_3 u(\tau) \|_{L^4_{\rm h}(L^\infty_{\rm v})} \| \dDv_k u^3(\tau) \|_{L^2} \| \dDv_j u(\tau) \|_{L^4_{\rm h}(L^2_{\rm v})} \, d\tau.
\end{align*}
Along the same line to the derivation of \eqref{eq3.4a}, we get, by using Lemma \ref{lemBern} and the definition of ${\rm A}(t)$,  that 
\begin{equation}\notag
\begin{split}
\| \dSv_{k-1} \partial_3 u(t) \|_{L^4_{\rm h}(L^\infty_{\rm v})}
&\lesssim 2^k \sum_{\ell\leq k-2} 2^{\frac{\ell}{2}} \| \dDv_\ell u(t) \|_{L^2}^{\frac12} \| \dDv_\ell \nabla_{\rm h} u(t) \|_{L^2}^{\frac12} \\
&\lesssim 2^k \| u(t) \|_{\dB^{0,\f12}}^{\frac12} \| \nabla_{\rm h} u(t) \|_{\dB^{0,\f12}}^{\frac12} \lesssim 2^k \bigl( \A'(t) \bigr)^{\frac14}.
\end{split}
\end{equation}
Whereas it follows from Lemma \ref{lemBern} and $\partial_3 u^3 = -\operatorname{div}_{\rm h} u^{\rm h}$ that
\begin{equation}\label{eq3.11}
\| \dDv_k u^3 \|_{L^2} \lesssim 2^{-k} \| \dDv_k \partial_3 u^3 \|_{L^2} \lesssim 2^{-k} \| \dDv_k \nabla_{\rm h} u^\h \|_{L^2}.
\end{equation}
Therefore, we deduce that
\begin{align*}
\int_0^t &\ff^2(\tau) \big| \bigl( \dDv_j (T^{\rm v}_{\partial_3 u} u^3) \mid \dDv_j u \bigr) \big| \, d\tau \\
\lesssim &\sum_{|k-j|\leq 4} \| \ff \dDv_k \nabla_{\rm h} u \|_{L^2_t(L^2)} \| \bigl( \A'\bigr)^{\frac12} \ff \dDv_j u \|_{L^2_t(L^2)}^{\frac12}\|\ff \dDv_j \nabla_{\rm h} u\|_{L^2_t(L^2)}^{\frac12} \\
\lesssim &\Bigl( \sum_{|k-j|\leq 4} d_k b_k^{-1} 2^{-\frac{k}{2}} \Bigr) d_j b_j^{-1} 2^{-\frac{j}{2}}
 \|\ff \nabla_{\rm h} u\|_{\widetilde L^2_t(\dB^{0,\f12}_\rb)}^{\frac32}
  \| \bigl( \A' \bigr)^{\frac12} \ff u\|_{\widetilde L^2_t(\dB^{0,\f12}_\rb)}^{\frac12},
\end{align*}
which implies that
\begin{equation}\label{eq3.12}
\begin{aligned}
\int_0^t &\ff^2(\tau)\big| \bigl( \dDv_j (T^{\rm v}_{\partial_3 u} u^3) \mid \dDv_j u \bigr) \big| \, d\tau\lesssim d_j^2 b_j^{-2} 2^{-j} \| \ff \nabla_{\rm h} u\|_{\widetilde L^2_t(\dB^{0,\f12}_\rb)}^{\frac32} \| \bigl( \A'\bigr)^{\frac12} \ff u \|_{\widetilde L^2_t(\dB^{0,\f12}_\rb)}^{\frac12}.
\end{aligned}
\end{equation}

Finally, for the remainder term, we deduce from Lemma \ref{lemBern} that
\begin{align*}
\int_0^t &\ff^2(\tau) \big| \bigl( \dDv_j \dRv(u^3, \partial_3 u) \mid \dDv_j u \bigr) \big| \, d\tau \\
\lesssim & 2^{\frac{j}{2}} \sum_{k \geq j-3} \int_0^t \ff^2(\tau) \|\dot {\D}^\v_k u^3(\tau) \|_{L^2} \| \wt{\D}^\v_k \partial_3 u(\tau) \|_{L^4_{\rm h}(L^2_{\rm v})} \| \dDv_j u(\tau) \|_{L^4_{\rm h}(L^2_{\rm v})} \, d\tau.
\end{align*}
It follows from  \eqref{eq3.11} that
\begin{equation}\label{dkv}
\| \dDv_k u^3(t) \|_{L^2} \lesssim 2^{-k} \| \dDv_k \nabla_{\rm h} u^\h(t) \|_{L^2} \lesssim 2^{-\frac{3}{2}k} \| \nabla_{\rm h} u^\h(t) \|_{\dB^{0,\f12}} \lesssim 2^{-\frac{3}{2}k} \bigl( \A'(t) \bigr)^{\frac12},
\end{equation}
 from which and Lemma \ref{lemBern}, we infer
\begin{align*}
\int_0^t &\ff^2(\tau)\big| \bigl( \dDv_j \dRv(u^3, \partial_3 u) \mid \dDv_j u \bigr) \big| \, d\tau \\
\lesssim & 2^{\frac{j}{2}} \sum_{k \geq j-3} \int_0^t \ff^2(\tau) 2^{-\frac{3}{2}k} \bigl( A'(\tau) \bigr)^{\frac12} 2^k \| \wt{\D}^\v_k u(\tau) \|_{L^4_{\rm h}(L^2_{\rm v})} \| \dDv_j u(\tau) \|_{L^4_{\rm h}(L^2_{\rm v})} \, d\tau \\
\lesssim & d_j b_j^{-1} \sum_{k \geq j-3} d_k b_k^{-1} 2^{-k} \| \ff \nabla_{\rm h} u \|_{\widetilde L^2_t(\dB^{0,\f12}_\rb)} \| \bigl( \A'\bigr)^{\frac12} \ff u \|_{\widetilde L^2_t(\dB^{0,\f12}_\rb)}.
\end{align*}
We thus obtain
\begin{equation}\label{eq3.13}
\begin{aligned}
\int_0^t &\ff^2(\tau) \big| \bigl( \dDv_j \dRv(u^3, \partial_3 u) \mid \dDv_j u \bigr) \big| \, d\tau\lesssim d_j^2 b_j^{-2} 2^{-j} \|\ff \nabla_{\rm h} u \|_{\widetilde L^2_t(\dB^{0,\f12}_\rb)} \| \bigl( \A'\bigr)^{\frac12} \ff  u\|_{\widetilde L^2_t(\dB^{0,\f12}_\rb)}.
\end{aligned}
\end{equation}

By summarizing the estimates \eqref{eq3.9}, \eqref{eq3.12}, and \eqref{eq3.13}, we achieve
\begin{equation}\label{eq3.14}
\begin{aligned}
\int_0^t &\ff^2(\tau) \big| \bigl( \dDv_j (u^3 \partial_3 u) \mid \dDv_j u \bigr) \big| \, d\tau \\
\lesssim & d_j^2 b_j^{-2} 2^{-j} \Bigl( \| \ff \nabla_{\rm h} u \|_{\widetilde L^2_t(\dB^{0,\f12}_\rb)} \| \bigl( \A'\bigr)^{\frac12} \ff u\|_{\widetilde L^2_t(\dB^{0,\f12}_\rb)} + \| \ff \nabla_{\rm h} u \|_{\widetilde L^2_t(\dB^{0,\f12}_\rb)}^{\frac32} \| \bigl( \A'\bigr)^{\frac12} \ff u \|_{\widetilde L^2_t(\dB^{0,\f12}_\rb)}^{\frac12} \Bigr).
\end{aligned}
\end{equation}

By substituting the estimates \eqref{eq3.8} and \eqref{eq3.14} into \eqref{eq3.3}, and then taking supremum over $t \in [0,T]$, we obtain
\begin{align*}
&\sup_{t \in [0,T]} \bigl( \ff^2(t) \| \dDv_j u(t) \|_{L^2}^2 \bigr) + 2C_0 \int_0^T \A'(t) \ff^2(t) \| \dDv_j u(t) \|_{L^2}^2 \, dt + 2 \int_0^T \ff^2(t) \| \nabla_{\rm h} \dDv_j u(t) \|_{L^2}^2 \, dt \\
&\leq \| \dDv_j u_0 \|_{L^2}^2+ C d_j^2 b_j^{-2} 2^{-j} \Bigl( \| \ff \nabla_{\rm h} u \|_{\widetilde L^2_T(\dB^{0,\f12}_\rb)} \| \bigl( \A' \bigr)^{\frac12} \ff u \|_{\widetilde L^2_T(\dB^{0,\f12}_\rb)} \\
&\qquad\qquad\qquad\qquad\qquad\qquad\quad+ \| \ff \nabla_{\rm h} u\|_{\widetilde L^2_T(\dB^{0,\f12}_\rb)}^{\frac32} \| \bigl( \A'\bigr)^{\frac12} \ff u \|_{\widetilde L^2_T(\dB^{0,\f12}_\rb)}^{\frac12} \Bigr).
\end{align*}
By taking the square root of the above inequality, then multiplying it by $b_j 2^{\frac{j}{2}}$, and finally summing up the resulting inequalities  over $j \in \Z$, we arrive at
\begin{equation}\label{eq3.15}
\begin{aligned}
&\| \ff u \|_{\widetilde L^2_T(\dB^{0,\f12}_\rb)} + \sqrt{2C_0} \| \bigl( \A' \bigr)^{\frac12} \ff u \|_{\widetilde L^2_T(\dB^{0,\f12}_\rb)} + \sqrt{2} \|\ff \nabla_{\rm h} u \|_{\widetilde L^2_T(\dB^{0,\f12}_\rb)} \\
&\leq C  \|u_0\|_{\dB^{0,\f12}_\rb}  + C \Bigl( \| \ff \nabla_{\rm h} u \|_{\widetilde L^2_T(\dB^{0,\f12}_\rb)}^{\f12} \| \bigl( \A'\bigr)^{\frac12} \ff u \|_{\widetilde L^2_T(\dB^{0,\f12}_\rb)}^{\f12} \\
 &\qquad\qquad\qquad\qquad\quad+ \| \ff \nabla_{\rm h} u \|_{\widetilde L^2_T(\dB^{0,\f12}_\rb)}^{\frac34} \| \bigl( \A'\bigr)^{\frac12} \ff u \|_{\widetilde L^2_T(\dB^{0,\f12}_\rb)}^{\frac14} \Bigr).
\end{aligned}
\end{equation}
Yet we get, by using Young's inequality,  that
\begin{equation}\label{eq3.16}
\begin{aligned}
&C \Bigl( \| \ff \nabla_{\rm h} u \|_{\widetilde L^2_T(\dB^{0,\f12}_\rb)}^{\f12} \| \bigl( \A'\bigr)^{\frac12} \ff u \|_{\widetilde L^2_T(\dB^{0,\f12}_\rb)}^{\f12} + \| \ff\nabla_{\rm h} u \|_{\widetilde L^2_T(\dB^{0,\f12}_\rb)}^{\frac34} \| \bigl( \A'\bigr)^{\frac12} \ff u \|_{\widetilde L^2_T(\dB^{0,\f12}_\rb)}^{\frac14} \Bigr) \\
&\leq (\sqrt{2} - 1) \| \ff \nabla_{\rm h} u\|_{\widetilde L^2_T(\dB^{0,\f12}_\rb)}  + \Bigl( \frac{1+\sqrt{2}}{2} C^2 +\f{3^3(1+\sqrt{2})^3}{32} C^{4} \Bigr) \| \bigl( \A'\bigr)^{\frac12} \ff u \|_{\widetilde L^2_T(\dB^{0,\f12}_\rb)}.
\end{aligned}
\end{equation}
We remark that the constant $C$ appearing in the proof is independent of $C_0$. Therefore, by taking $C_0$ large enough so that
\[
\sqrt{2C_0} \geq  \Bigl( \frac{1+\sqrt{2}}{2} C^2 +\f{3^3(1+\sqrt{2})^3}{32} C^{4} \Bigr),
\]
we deduce \eqref{eq:prop3.1} from \eqref{eq3.15}. This completes the proof of Proposition \ref{prop3.1}.
\end{proof}

Now, we are in a position to complete the proof of Proposition \ref{prop1.5}.

\begin{proof}[Proof of Proposition \ref{prop1.5}]
For simplicity, we present only the {\it a priori} estimate. Let $u$ be a smooth enough solution of $(ANS)$
and $\ff(t)$ be defined by \eqref{eq:prop3.1}. Then in view of \eqref{eq:prop3.1}, we observe that $\ff(t)\leq \ff(\tau)$ for $0 \leq \tau \leq t \leq T$. Then, using Minkowski's inequality, we obtain
\[
\|u\|_{L^\infty_t(\dB^{0,\f12}_\rb)} + \|\nabla_\h u\|_{L^2_t(\dB^{0,\f12}_\rb)}
\leq e^{C_0 \A(t)} \Bigl( \| \ff u \|_{\widetilde L^\infty_t(\dB^{0,\f12}_\rb)} + \|\ff \nabla_\h u \|_{\widetilde L^2_t(\dB^{0,\f12}_\rb)} \Bigr).
\]
Substituting the estimate \eqref{eq:prop3.1} into the above inequality leads to \eqref{eq:prop1.5}. 
This completes the proof of Proposition \ref{prop1.5}.
\end{proof}

\smallskip

\section{The continuity estimates in the weak space}\label{Sect4}

In this section, we shall derive the estimate for the difference of any two solutions, which belong to
the homogeneous Besov space $\dB^{0,\f12},$ of $(ANS)$  via the weak norm $\|\cdot\|_X$ defined by \eqref{S1eq1a}, and present the proof of Proposition \ref{prop1.6}. 

Let $(u,p)$ and $(v,q)$ be two solutions of $(ANS)$. Then $w\eqdefa u-v$ solves the equations:
\begin{equation}\label{eqs:w}
 \left\{
\begin{array}{l}
\displaystyle \partial_t w + u \cdot \nabla w+w\cdot \nabla v - \Delta_{\mathrm{h}} w = -\nabla \frak{p}, \qquad (t,x) \in \mathbb{R}^+ \times \mathbb{R}^3, \\[4pt]
\displaystyle \operatorname{div} w = 0, \\[4pt]
\displaystyle w|_{t=0} = w_0,
\end{array}
\right.
\end{equation}
where $\frak{p}=p-q$ and $w_0=u_0-v_0$.

The term $w^3\p_3 v$ in \eqref{eqs:w}  causes the loss of one vertical derivative for $w.$ That is the reason why we have to  estimate the high-frequency part of $w$ in $H^{0,-\f12}$ by following ideas in \cite{Paicu}. For the low frequency part, as there is no derivative loss, we shall deal with the 
  estimate in $\dB^{0,\f12}.$
  

Our first estimate is concerned with the low frequency part of $w.$

\begin{prop}\label{S4prop1}
  {\sl
    Let $u,v$ be smooth enough solutions of $(ANS)$ on $[0,T]$ and $\B(t)$ be defined by \eqref{eq:prop1.6}. Then there exists some large enough constant $C>0$ such that for any {$C_3>0$} and $\fg_1(t)\eqdefa e^{-{C_3} \B(t)},$\begin{equation}\label{eq:prop5.1}
\begin{aligned}
    &\|\fg_1 w\|_{\wt L^\oo_T(\dBl^{0,\f12})}^2+\|\fg_1\nabla_\h w(t)\|_{\wt L^2_T(\dBl^{0,\f12})}^2+{C_3}\|\bigl(\B'\bigr)^\f12\fg_1 w\|_{\wt L^2_T(\dBl^{0,\f12})}^2 \\
&\leq {C} \|w_{0}\|_{\dBl^{0,\f12}}^2+C \sum_{\kappa=1}^3 \| \fg_1 \nabla_\h w \|_{\widetilde L^2_T(X)}^{\f\kappa2}\|\bigl(\B'\bigr)^\f12\fg_1 w\|_{\wt L^2_T(X)}^{2-\f{\kappa}2},
\end{aligned}
\end{equation}
where the semi-norm $\wt{L}^p_T(\dBl^{0,\f12})$  and the norm $\wt{L}^p_T(X)$ are given respectively by \eqref{S1eq2aq} and \eqref{S1eq2a}.
  }
\end{prop}

\begin{proof}
By applying $\dDv_j$ to the momentum equations of \eqref{eqs:w} and then taking the $L^2$ inner product of the resulting equation with $\dDv_j w$, we obtain
\begin{eqnarray}\label{S5eqdj}
\frac12 \frac{d}{dt} \|\dDv_j w(t)\|_{L^2}^2 + \|\nabla_\h \dDv_j w(t)\|_{L^2}^2 = - \bigl( \dDv_j (u\cdot\nabla w+w\cdot\nabla v) \mid \dDv_j w \bigr).
\end{eqnarray}
To overcome the difficulty that one cannot use a Gronwall type argument in the framework of Chemin-Lerner type spaces, we shall use once again the time-weighted Chemin-Lerner type norm as in the proof of Proposition \ref{prop3.1}. Indeed, noticing that
\begin{align*}
\int_0^t \fg_1^2(\tau) \frac{d}{d\tau} \|\dDv_j w(\tau)\|_{L^2}^2 \, d\tau =&\fg_1^2(t)\|\dDv_j w(t)\|_{L^2}^2 - \|\dDv_j w_0\|_{L^2}^2\\
&+ 2{C_3} \int_0^t \B'(\tau)\fg_1^2(\tau)\|\dDv_j w(\tau)\|_{L^2}^2 \, d\tau,
\end{align*}
we get, by first multiplying \eqref{S5eqdj} by {$2\fg_1^2(t)$} and then integrating the resulting equation over $[0,t],$ that
\begin{equation}\label{eq5.3}
  \begin{aligned}
   & \fg_1^2(t)\|\dDv_j w(t)\|_{L^2}^2 +2 \|\fg_1\nabla_\h \dDv_j w\|_{L^2_t(L^2)}^2+2{C_3} \|\bigl(\B'\bigr)^\f12 \fg_1 \dDv_j w\|_{L^2_t(L^2)}^2 \\
    &\leq \|\dDv_j w_0\|_{L^2}^2 + 2\int_0^t \fg_1^2(\tau)|\bigl( \dDv_j (u\cdot\nabla w+w\cdot\nabla v) \mid \dDv_j w \bigr)|d\tau.
  \end{aligned}
\end{equation}

Let us turn to the estimate of the nonlinear terms in \eqref{eq5.3} with $j\leq -1$. Due to the feature of anisotropic dissipation in $(ANS),$ we write
 \begin{equation}\label{S4eq1a}
  u\cdot\nabla w+w\cdot\nabla v = u^{\rm h} \cdot \nabla_\h w + u^3 \partial_3 w+ w^{\rm h} \cdot \nabla_\h v + w^3 \partial_3 v. \end{equation}
   Below we shall handle these four parts separately. Precisely
 
 \begin{lem}\label{S4lem1}
 {\sl Under the assumptions of Proposition \ref{S4prop1},  for  $j\leq -1,$ one has
 \begin{equation}\label{eq5.13}
\begin{aligned}
\int_0^t \fg_1^2(\tau) \big| \bigl( \dDv_j (u^{\rm h} \cdot \nabla_\h w) \mid \dDv_j w \bigr) \big| \, d\tau \lesssim &d_j^2  2^{-j}  \|\bigl(\B'\bigr)^\f12 \fg_1 w\|_{\wt L^2_t({ X})}^\f12 \| \fg_1 \nabla_\h w \|_{\widetilde L^2_t(X)}^\f32.
\end{aligned}
\end{equation}
}\end{lem} 
 
 \begin{lem}\label{S4lem2}
 {\sl Under the assumptions of Proposition \ref{S4prop1},   for  $j\leq -1,$ one has
 \begin{equation}\label{eq5.19}
\begin{aligned}
\int_0^t & \fg_1^2(\tau)\big| \bigl( \dDv_j (u^3 \p_3 w) \mid \dDv_j w \bigr) \big| \, d\tau \\
\lesssim& d_j^2  2^{-j} \Bigl(\|\fg_1\nabla_\h w\|_{\widetilde L^2_t(X)}\| \bigl(\B'\bigr)^\f12 \fg_1 w\|_{\wt L^2_t(X)}+\|\fg_1 \nabla_\h w \|_{\widetilde L^2_t({X})} ^\f12  \|\bigl(\B'\bigr)^\f12\fg_1  w\|_{\wt L^2_t(X)}^\f32\Bigr).
\end{aligned}
\end{equation}}\end{lem} 

 \begin{lem}\label{S4lem3}
 {\sl Under the assumptions of Proposition \ref{S4prop1}, for  $j\leq -1,$ one has
 \begin{equation}\label{eq5.25}
\begin{aligned}
\int_0^t&\fg_1^2(\tau) \big| \bigl( \dDv_j (w^{\rm h} \cdot \nabla_\h v) \mid \dDv_j w \bigr) \big| \, d\tau
\lesssim d_j^2  2^{-j}\| \fg_1 \nabla_\h w \|_{\widetilde L^2_t(X)} \|\bigl(\B'\bigr)^\f12 \fg_1 w\|_{\wt L^2_t(X)}.
\end{aligned}
\end{equation} }\end{lem} 

 \begin{lem}\label{S4lem4}
 {\sl Under the assumptions of Proposition \ref{S4prop1}, for  $j\leq -1,$ one has
 \begin{equation}\label{eq5.30}
\begin{aligned}
\int_0^t &\fg_1^2(\tau) \big| \bigl( \dDv_j (w^3\p_3 v) \mid \dDv_j w \bigr) \big| \, d\tau \\
\lesssim& d_j^2  2^{-j} \Bigl(\| \fg_1 \nabla_\h w \|_{\widetilde L^2_t(X)} \|\bigl(\B'\bigr)^\f12\fg_1  w\|_{\wt L^2_t(X)}
+ \|\bigl(\B'\bigr)^\f12 \fg_1 w\|_{\wt L^2_t({X})}^\f12 
\| \fg_1 \nabla_\h w \|_{\widetilde L^2_t(X)}^\f32\Bigr).
\end{aligned}
\end{equation}
}\end{lem} 

We admit the above lemmas for the time being, and continue our proof of Proposition \ref{S4prop1}.\\

In view of \eqref{S4eq1a}, by substituting the estimates (\ref{eq5.13}-\ref{eq5.30}) into \eqref{eq5.3}, and then taking supremum over $t\in [0,T]$, we achieve
  \begin{align*}
    \|&\fg_1\dDv_j w\|_{L^\oo_T(L^2)}^2 +2 \|\fg_1\nabla_\h \dDv_j w\|_{L^2_T(L^2)}^2+2{C_3}\|\bigl(\B'\bigr)^\f12\fg_1 \dDv_jw\|_{L^2_T(L^2)}^2\\
    & \leq \|\dDv_j w_0\|_{L^2}^2 +C d_j^2 2^{-j} \sum_{\kappa=1}^3 \| \fg_1 \nabla_\h w\|_{\widetilde L^2_T(X)}^{\f\kappa2}\|\bigl(\B'\bigr)^\f12\fg_1 w\|_{\wt L^2_T(X)}^{2-\f{\kappa}2}.
  \end{align*}
Thanks to \eqref{2.4} and \eqref{S1eq2aq}, by taking the square root of the above inequality, then multiplying it by $ 2^{\frac{j}{2}}$, and finally summing up the resulting inequalities over $j\leq -1$, we arrive at
  \begin{align*}
    \|\fg_1& w\|_{\wt L^\oo_T(\dBl^{0,\f12})}+\|\fg_1\nabla_\h w\|_{\wt L^2_T(\dBl^{0,\f12})}+{\sqrt{C_3}}\|\bigl(\B'\bigr)^\f12\fg_1 w\|_{\wt L^2_T(\dBl^{0,\f12})} \\
\leq& {C} \|w_{0}\|_{\dBl^{0,\f12}}+C \sum_{\kappa=1}^3 \| \fg_1\nabla_\h 
w \|_{\widetilde L^2_T(X)}^{\f\kappa4}  \|\bigl(\B'\bigr)^\f12\fg_1  w\|_{\wt L^2_T(X)}^{1-\f{\kappa}4},
  \end{align*}
which leads to  \eqref{eq:prop5.1}. This finishes the proof of Proposition \ref{S4prop1}.\end{proof}

We now present the proof of Lemmas \ref{S4lem1}-\ref{S4lem4}.
 
\begin{proof}[Proof of Lemma \ref{S4lem1}] By applying  Bony's decomposition \eqref{homo bony v} to $u^{\rm h} \cdot \nabla_\h w$ in the vertical variable, we write
\begin{equation}\label{S4eq2a}
u^{\rm h} \cdot \nabla_\h w = \dTv_{u^{\rm h}} \nabla_\h w + \dTv_{\nabla_\h w} u^{\rm h} + \dRv(u^{\rm h}, \nabla_\h w).
\end{equation}

Considering the support properties to the Fourier transform of the terms in $\dTv_{u^{\rm h}} \nabla_\h w$, we find
\[
\big| \bigl( \dDv_j (\dTv_{u^{\rm h}} \nabla_\h w) \mid \dDv_j w \bigr) \big|
\lesssim \sum_{|k-j|\leq 4} \| \dSv_{k-1} u^{\rm h} \|_{L^4_{\rm h}(L^\infty_{\rm v})} \| \dDv_k \nabla_\h w \|_{L^2} \| \dDv_j w \|_{L^4_{\rm h}(L^2_{\rm v})}.
\]
Yet it follows from Lemma \ref{lemBern}, the 2-D interpolation inequality \eqref{interpolation inequality}
 and the definition of $\B(t)$  (see \eqref{eq:prop1.6}) that
\begin{equation}\label{eq5.4}
\begin{aligned}
\| \dSv_{k-1} u(t) \|_{L^4_{\rm h}(L^\infty_{\rm v})}
&\lesssim \sum_{\ell=-\oo}^{k-2} 2^{\frac{\ell}{2}} \| \dDv_\ell u(t) \|_{L^4_{\rm h}(L^2_{\rm v})} \\
&\lesssim \sum_{\ell=-\oo}^{k-2} 2^{\frac{\ell}{2}} \| \dDv_\ell u(t) \|_{L^2}^{\frac12} \| \nabla_\h \dDv_\ell u(t) \|_{L^2}^{\frac12} \\
&\lesssim \| u(t) \|_{\dB^{0,\f12}}^{\frac12} \| \nabla_\h u(t) \|_{\dB^{0,\f12}}^{\frac12} \lesssim \bigl( \B'(t) \bigr)^{\frac14},
\end{aligned}
\end{equation}
from which,  \eqref{S1eq2aq} and $j\leq -1$,  we infer 
\begin{align*}
\int_0^t &\fg_1^2(\tau)\big| \bigl( \dDv_j (\dTv_{u^{\rm h}} \nabla_\h w) \mid \dDv_j w \bigr) \big| \, d\tau \\
\lesssim &\sum_{|k-j|\leq 4} \int_0^t  \bigl(\B'(\tau) \bigr)^{\frac14}\fg_1^2(\tau) \| \dDv_k \nabla_\h w(\tau) \|_{L^2} \| \dDv_j w(\tau) \|_{L^2}^{\frac12} \| \nabla_\h \dDv_j w(\tau) \|_{L^2}^{\frac12} \, d\tau \\
\lesssim &\sum_{|k-j|\leq 4} \|\fg_1 \dDv_k \nabla_\h w\|_{L^2_t(L^2)} \| \bigl( \B' \bigr)^{\frac12}
\fg_1\dDv_j w \|_{L^2_t(L^2)}^{\frac12}  \| \fg_1 \dDv_j \nabla_\h w\|_{L^2_t(L^2)}^{\frac12} \\
\lesssim &\Bigl( \sum_{|k-j|\leq 4} d_k 2^{-\frac{k}{2}} \Bigr) d_j  2^{-\frac{j}{2}} \|\fg_1{\dSv_4} \nabla_\h w\|_{\widetilde L^2_t(\dB^{0,\f12})}  \| \bigl(\B'\bigr)^\f12 \fg_1 w\|_{\wt L^2_t(\dBl^{0,\f12})}^\f12\| \fg_1  \nabla_\h w \|_{\widetilde L^2_t(\dBl^{0,\f12})}^\f12.
\end{align*}
Notice from \eqref{S1eq2aq} and \eqref{S1eq2a} that
\begin{equation}\label{eq5.5}
  \begin{aligned}
    \|\fg_1 {\dSv_4} \nabla_\h w \|_{\widetilde L^2_t(\dB^{0,\f12})} 
    &\leq \|\fg_1\nabla_\h w\|_{\widetilde L^2_t(\dBl^{0,\f12})} +C\sum_{j=0}^{3} 2^{\f{j}2} \|\fg_1\dDv_j \nabla_\h w \|_{ L^2_t(L^2)}\\
    &\leq \|\fg_1\nabla_\h w\|_{\widetilde L^2_t(\dBl^{0,\f12})} +C \|\fg_1\nabla_\h w \|_{\widetilde L^2_t(\dHh^{0,-\f12})}\leq C\| \fg_1\nabla_\h w \|_{\widetilde L^2_t(X)},
  \end{aligned}
\end{equation}
we deduce that
\begin{equation}\label{eq5.6}
\begin{aligned}
\int_0^t \fg_1^2(\tau) \big| \bigl( \dDv_j (\dTv_{u^{\rm h}} \nabla_\h w) \mid \dDv_j w \bigr) \big| \, d\tau \lesssim  &d_j^2  2^{-j}  \| \bigl(\B'\bigr)^\f12 \fg_1 w\|_{\wt L^2_t({X})}^\f12\| \fg_1\nabla_\h w \|_{\widetilde L^2_t(X)} ^\f32.
\end{aligned}
\end{equation}

Along the same line, it follows from Lemma \ref{lemBern} that
\begin{equation}
\begin{split}
\| \dSv_{k-1} \nabla_\h w(t) \|_{L^2_{\rm h}(L^\infty_{\rm v})}
&\lesssim \sum_{\ell\in \Z} 2^{\frac{\ell}{2}} \| \dDv_\ell \dSv_{k-1} \nabla_\h w(t) \|_{L^2} \lesssim \| \dSv_{k-1} \nabla_\h w(t) \|_{\dB^{0,\f12}} ,
\end{split}
\end{equation}
and 
\begin{equation}\label{eq5.7}
  \| \dDv_k u(t) \|_{L^4_{\rm h}(L^2_\vv)}  \lesssim d_k(t) 2^{-\f{k}2} \|u(t) \|_{\dB^{0,\f12}}^\f12\|\nabla_\h u(t) \|_{\dB^{0,\f12}}^\f12 \lesssim d_k(t) 2^{-\f{k}2}\bigl(\B'(t)\bigr)^\f14,
\end{equation}
so that we deduce from \eqref{interpolation inequality} and $j\leq -1$ that
\begin{align*}
\int_0^t &\fg_1^2(\tau) \big| \bigl( \dDv_j (\dTv_{\nabla_\h w} u^{\rm h}) \mid \dDv_j w \bigr) \big| \, d\tau \\
\lesssim &\sum_{|k-j|\leq 4} \int_0^t \fg_1^2(\tau) \| \dSv_{k-1} \nabla_\h w(\tau) \|_{L^2_{\rm h}(L^\infty_{\rm v})} \| \dDv_k u^{\rm h}(\tau) \|_{L^4_{\rm h}(L^2_{\rm v})} \| \dDv_j w(\tau) \|_{L^4_{\rm h}(L^2_{\rm v})} \, d\tau \\
\lesssim &\sum_{|k-j|\leq 4}  2^{-\f{k}2}\int_0^t d_k(\tau)\bigl(\B'(\tau)\bigr)^\f14 \fg_1^2(\tau)\| \dSv_4 \nabla_\h w(\tau) \|_{\dB^{0,\f12}} \| \dDv_j w(\tau) \|_{L^2}^\f12 \|\nabla_\h \dDv_j w(\tau) \|_{L^2}^\f12 \, d\tau\\
\lesssim &\Bigl( \sum_{|k-j|\leq 4} d_k 2^{-\frac{k}{2}} \Bigr) d_j  2^{-\frac{j}{2}} \|\fg_1{\dSv_4} \nabla_\h w \|_{\widetilde L^2_t(\dB^{0,\f12})}  \| \bigl(\B'\bigr)^\f12 \fg_1 w\|_{\wt L^2_t(\dBl^{0,\f12})}^\f12\| \fg_1 \nabla_\h w \|_{\widetilde L^2_t(\dBl^{0,\f12})}^\f12,
\end{align*}
which together with \eqref{eq5.5} implies
\begin{equation} \label{eq5.8}
\begin{aligned}
\int_0^t \fg^2_1(\tau)\big| \bigl( \dDv_j (\dTv_{\nabla_\h w} u^{\rm h}) \mid \dDv_j w \bigr) \big| \, d\tau \lesssim &d_j^2  2^{-j} \|\bigl(\B'\bigr)^\f12 \fg_1  w\|_{\wt L^2_t({X})}^\f12
\|\fg_1\nabla_\h w\|_{\widetilde L^2_t(X)}^\f32.
\end{aligned}
\end{equation}

Finally, for the remainder term with $j \leq -1$, we deduce from Lemma \ref{lemBern} that
\begin{align*}
\int_0^t &\fg_1^2(\tau) \big| \bigl( \dDv_j \dRv(u^{\rm h}, \nabla_\h w) \mid \dDv_j w \bigr) \big| \, d\tau \\
\lesssim & 2^{\frac{j}{2}} \sum_{k \geq j-3} \int_0^t \fg_1^2(\tau)\| \wt{\D}^\v_k u^{\rm h}(\tau) \|_{L^4_{\rm h}(L^2_{\rm v})} \| \dDv_k \nabla_\h w(\tau) \|_{L^2} \| \dDv_j w(\tau) \|_{L^4_{\rm h}(L^2_{\rm v})} \, d\tau.
\end{align*}
Similar to  \eqref{eq5.7}, one has
\begin{equation}\label{eq5.8a}
\| \wt{\D}^\v_k u(t) \|_{L^4_{\rm h}(L^2_{\rm v})} 
\lesssim d_k(t) 2^{-\frac{k}{2}} \bigl(\B'(t) \bigr)^{\frac14},
\end{equation}
from which and \eqref{interpolation inequality}, we infer 
\begin{align*}
\int_0^t &\fg_1^2(\tau) \big| \bigl( \dDv_j \dRv(u^{\rm h}, \nabla_\h w) \mid \dDv_j w \bigr) \big| \, d\tau \\
\lesssim & 2^{\frac{j}{2}} \sum_{k \geq j-3} \int_0^t  d_k(\tau) 2^{-\frac{k}{2}} \bigl( \B'(\tau) \bigr)^{\frac14}\fg_1^2(\tau) \|\dDv_k \nabla_\h w(\tau)\|_{L^2} \| \dDv_j w(\tau) \|_{L^2}^\f12\| \nabla_\h \dDv_j w(\tau) \|_{L^2}^\f12 \, d\tau\\
\lesssim & d_j \Bigl( \| \fg_1\nabla_\h w \|_{\widetilde L^2_t(\dBl^{0,\f12})}\sum_{k = j-3}^{-1}  d_k 2^{-k} +  \|\fg_1\nabla_\h w \|_{\widetilde L^2_t(\dHh^{0,-\f12})}\sum_{k = 0}^{\oo}  d_k  \Bigr) \\
&\qquad\qquad\qquad\qquad\qquad\qquad\times  \|\bigl(\B'\bigr)^\f12 \fg_1 w \|_{\wt L^2_t(\dBl^{0,\f12})}^\f12\|\fg_1\nabla_\h w \|_{\wt L^2_t(\dBl^{0,\f12})}^\f12 .
\end{align*}
Notice that $2^{j}\in \ell^1(-\N)$, we have 
\begin{equation}\label{eq5.9}
  \sum_{k = j-3}^{-1}  d_k 2^{-k} +\sum_{k = 0}^{\oo}  d_k \lesssim  2^{-j} d_j+1 \lesssim 2^{-j} d_j, \qquad \text{for}\quad j\leq -1.
\end{equation} 
Therefore, we conclude that 
\begin{equation} \label{eq5.11}
\begin{aligned}
\int_0^t \fg_1^2(\tau) \big| \bigl( \dDv_j \dRv(u^{\rm h}, &\nabla_\h w) \mid \dDv_j w \bigr) \big| \, d\tau \lesssim d_j^2  2^{-j}  \|\bigl(\B'\bigr)^\f12 \fg_1w\|_{\wt L^2_t({X})}^\f12
 \|\fg_1 \nabla_\h w \|_{\widetilde L^2_t(X)}^\f32.
\end{aligned}
\end{equation}

By summarizing the estimates \eqref{eq5.6}, \eqref{eq5.8}, and \eqref{eq5.11}, we obtain \eqref{eq5.13}. This completes the proof of Lemma \ref{S4lem1}. \end{proof}

\begin{proof}[Proof of Lemma \ref{S4lem2}]
By applying Bony's decomposition \eqref{homo bony v} once again to $u^3\p_3 w$ in the vertical variable, we write
\begin{equation} \label{S4eq3a}
u^3\p_3 w = \dTv_{u^3} \p_3 w + \dTv_{\p_3 w} u^3 + \dRv(u^3, \p_3 w).
\end{equation}

 Since we focus on the low frequency part of $w$, we do not need to use the commutator argument as the proof of \eqref{eq3.9}. Precisely, we get, by using the support properties of the Fourier transform of the terms in $\dTv_{u^3} \p_3 w$ and \eqref{eq5.4},  that
\begin{align*}
\int_0^t &\fg_1^2(\tau) \big| \bigl( \dDv_j (\dTv_{u^3} \p_3 w) \mid \dDv_j w \bigr) \big| \, d\tau \\
\lesssim &\sum_{|k-j|\leq 4} \int_0^t  \bigl( \B'(\tau)\bigr)^{\frac14} \fg_1^2(\tau)\| \dDv_k \p_3 w(\tau) \|_{L^2} \| \dDv_j w(\tau) \|_{L^2}^{\frac12} \| \nabla_\h \dDv_j w(\tau) \|_{L^2}^{\frac12} \, d\tau \\
\lesssim &\sum_{|k-j|\leq 4} \int_0^t  \bigl( \B'(\tau) \bigr)^{\frac34}\fg_1^2(\tau) \| \dDv_k  w(\tau) \|_{L^2} \| \dDv_j w(\tau) \|_{L^2}^{\frac12} \| \nabla_\h \dDv_j w(\tau) \|_{L^2}^{\frac12} \, d\tau \\
\lesssim &\Bigl( \sum_{|k-j|\leq 4} d_k 2^{-\frac{k}{2}} \Bigr) d_j  2^{-\frac{j}{2}}  \|\bigl(\B'\bigr)^{\f12}\fg_1 {\dSv_4} w\|_{\widetilde L^2_t(\dB^{0,\f12})}   \|\bigl(\B'\bigr)^\f12 \fg_1  w\|_{\wt L^2_t(\dBl^{0,\f12})}^\f12\| \fg_1\nabla_\h w \|_{\widetilde L^2_t(\dBl^{0,\f12})}^\f12,
\end{align*}
where we used Lemma \ref{lemBern} so that $\| \dDv_k \p_3 w(\tau) \|_{L^2} \lesssim \| \dDv_k  w(\tau) \|_{L^2} $ for $k\leq 3$ and $1\lesssim B'(t)$. 

While it follows from a similar derivation of \eqref{eq5.5} that
\begin{equation}\label{eq5.14}
  \begin{aligned}
    \|\bigl(\B'\bigr)^\f12  \fg_1{\dSv_4 } w \|_{\widetilde L^2_t(\dB^{0,\f12})}
    &\leq \| \bigl(\B'\bigr)^\f12 \fg_1 w\|_{\widetilde L^2_t(\dBl^{0,\f12})} +C\sum_{j=0}^{3} 2^{\f{j}2} \| \bigl(\B'\bigr)^\f12 \fg_1\dDv_j  w\|_{ L^2_t(L^2)}\\
    &\leq \| {\bigl(\B'\bigr)^\f12}\fg_1 w \|_{\widetilde L^2_t(\dBl^{0,\f12})} +C \| \bigl(\B'\bigr)^\f12 
    \fg_1 w\|_{\widetilde L^2_t(\dHh^{0,-\f12})},
  \end{aligned}
\end{equation}
we thus deduce that
\begin{equation}\label{eq5.15}
\begin{aligned}
\int_0^t \fg_1^2(\tau) \big| \bigl( \dDv_j (\dTv_{u^3} \p_3 w)& \mid \dDv_j w \bigr) \big| \, d\tau \lesssim d_j^2  2^{-j} \| \fg_1 \nabla_\h w \|_{\widetilde L^2_t({\color{red}X})} ^\f12 \|\bigl(\B'\bigr)^\f12 \fg_1 w\|_{\wt L^2_t(X)}^\f32.
\end{aligned}
\end{equation}

Along the same line to  the proof of \eqref{eq5.8}, we get, by using $\| \dSv_{k-1} \p_3 w(\tau) \|_{L^2_{\rm h}(L^\infty_{\rm v})}\lesssim \| \dSv_4 w(\tau) \|_{L^2_{\rm h}(L^\infty_{\rm v})}$ for $k\leq 3$, $1\lesssim B'(t)$ and \eqref{eq5.7}, that
\begin{align*}
\int_0^t &\fg_1^2(\tau) \big| \bigl( \dDv_j (\dTv_{\p_3 w} u^3) \mid \dDv_j w \bigr) \big| \, d\tau \\
\lesssim &\sum_{|k-j|\leq 4} \int_0^t \fg_1^2(\tau) \| \dSv_{k-1} \p_3 w(\tau) \|_{L^2_{\rm h}(L^\infty_{\rm v})} \| \dDv_k u^3(\tau) \|_{L^4_{\rm h}(L^2_{\rm v})} \| \dDv_j w(\tau) \|_{L^4_{\rm h}(L^2_{\rm v})} \, d\tau \\
\lesssim &\sum_{|k-j|\leq 4}  2^{-\f{k}2}\int_0^td_k(\tau) \bigl(\B'\bigr)^\f34\fg_1^2(\tau)\| \dSv_4 w(\tau) \|_{\dB^{0,\f12}} \| \dDv_j w(\tau) \|_{L^2}^\f12 \|\nabla_\h \dDv_j w(\tau) \|_{L^2}^\f12 \, d\tau\\
\lesssim &\Bigl( \sum_{|k-j|\leq 4} d_k 2^{-\frac{k}{2}} \Bigr) d_j  2^{-\frac{j}{2}}  \| \bigl(\B'\bigr)^\f12 \fg_1 {\dSv_4 } w \|_{\widetilde L^2_t(\dB^{0,\f12})}  \|\bigl(\B'\bigr)^\f12 \fg_1 w\|_{\wt L^2_t(\dBl^{0,\f12})}^\f12\| \fg_1\nabla_\h w \|_{\widetilde L^2_t(\dBl^{0,\f12})}^\f12,
\end{align*}
which together with \eqref{eq5.14} implies
\begin{equation} \label{eq5.17}
\begin{aligned}
\int_0^t  \fg_1^2(\tau)\big| \bigl( \dDv_j (\dTv_{\p_3 w} u^3) &\mid \dDv_j w \bigr) \big| \, d\tau \lesssim d_j^2  2^{-j} \| \fg_1\nabla_\h w \|_{\widetilde L^2_t({X})} ^\f12
\|\bigl(\B'\bigr)^\f12 \fg_1 w\|_{\wt L^2_t(X)}^\f32.
\end{aligned}
\end{equation}

Finally, for the remainder term with $j \leq -1$, we deduce from Lemma \ref{lemBern} that
\begin{align*}
\int_0^t &\fg_1^2(\tau) \big| \bigl( \dDv_j \dRv(u^3, \p_3 w) \mid \dDv_j w \bigr) \big| \, d\tau \\
\lesssim & 2^{\frac{j}{2}} \sum_{k \geq j-3} \int_0^t \fg_1^2(\tau)\| \wt{\D}^\v_k u^3(\tau) \|_{L^2} \| \dDv_k \p_3 w(\tau) \|_{L^4_\h(L^2_\v)} \| \dDv_j w(\tau) \|_{L^4_{\rm h}(L^2_{\rm v})} \, d\tau\\ 
\lesssim & 2^{\frac{j}{2}} \sum_{k \geq j-3} \int_0^t \fg_1^2(\tau) \| \wt{\D}^\v_k \p_3 u^3(\tau) \|_{L^2} \| \dDv_k  w(\tau) \|_{L^4_\h(L^2_\v)} \| \dDv_j w(\tau) \|_{L^4_{\rm h}(L^2_{\rm v})} \, d\tau.
\end{align*}
Observing that
\begin{align*}
\| \wt{\D}^\v_k \p_3u^3 (t) \|_{L^2} \lesssim d_k(t) 2^{-\frac{k}{2}} \| \dive_\h u^\h (t) \|_{\dB^{0,\f12}} \lesssim d_k(t) 2^{-\frac{k}{2}} \bigl( \B'(t) \bigr)^{\frac12},
\end{align*}
from which and \eqref{interpolation inequality}, we infer 
\begin{align*}
\int_0^t &\fg_1^2(\tau)\big| \bigl( \dDv_j \dRv(u^3, \p_3 w) \mid \dDv_j w \bigr) \big| \, d\tau \\
\lesssim & 2^{\frac{j}{2}} \sum_{k \geq j-3} \int_0^t  d_k(\tau) 2^{-\frac{k}{2}} \bigl( \B'(\tau) \bigr)^{\frac12} \fg_1^2(\tau) \|\dDv_k  w(\tau)\|_{L^2}^\f12 \|\dDv_k \nabla_\h w(\tau)\|_{L^2}^\f12 \\
&\qquad\qquad\times\| \dDv_j w(\tau) \|_{L^2}^\f12\| \nabla_\h \dDv_j w(\tau) \|_{L^2}^\f12 \, d\tau\\
\lesssim & d_j \Bigl( \|\fg_1 \nabla_\h w \|_{\widetilde L^2_t(\dBl^{0,\f12})}^\f12 \|\bigl(\B'\bigr)^\f12 \fg_1 w\|_{\wt L^2_t(\dBl^{0,\f12})}^\f12\sum_{k = j-3}^{-1}  d_k 2^{-k} \\
&\qquad+  \|\fg_1 \nabla_\h w\|_{\widetilde L^2_t(\dHh^{0,-\f12})}^\f12 \|\bigl(\B'\bigr)^\f12
\fg_1 w\|_{\wt L^2_t(\dHh^{0,-\f12})}^\f12 \sum_{k = 0}^{\oo}  d_k  \Bigr) \\
&\quad\times   \|\bigl(\B'\bigr)^\f12 \fg_1 w\|_{\wt L^2_t(\dBl^{0,\f12})}^\f12\|\fg_1\nabla_\h w \|_{\wt L^2_t(\dBl^{0,\f12})}^\f12 .
\end{align*}
Together with \eqref{eq5.9}, we conclude that 
\begin{equation} \label{eq5.18}
\begin{aligned}
\int_0^t & \fg_1^2(\tau)\big| \bigl( \dDv_j \dRv(u^3, \p_3 w) \mid \dDv_j w \bigr) \big| \, d\tau \lesssim d_j^2  2^{-j}  \| \fg_1 \nabla_\h w \|_{\widetilde L^2_t(X)} 
\|\bigl(\B'\bigr)^\f12 \fg_1 w\|_{\wt L^2_t(X)}.
\end{aligned}
\end{equation}

By summarizing the estimates \eqref{eq5.15}, \eqref{eq5.17}, and \eqref{eq5.18}, we conclude the proof of \eqref{eq5.19}. \end{proof}

\begin{proof}[Proof of Lemma \ref{S4lem3}]
By applying the  Bony's decomposition \eqref{homo bony v} to $w^\h \cdot \nabla_\h v$ in the vertical variable, we write
\begin{equation} \label{S4eq5a}
w^\h \cdot \nabla_\h v = \dTv_{w^{\rm h}} \nabla_\h v + \dTv_{\nabla_\h v} w^{\rm h} + \dRv(w^{\rm h}, \nabla_\h v).
\end{equation}

Considering the support properties of the Fourier transform of the terms in $\dTv_{w^{\rm h}} \nabla_\h v$, we find
\[
\big| \bigl( \dDv_j (\dTv_{w^{\rm h}} \nabla_\h v) \mid \dDv_j w \bigr) \big|
\lesssim \sum_{|k-j|\leq 4} \| \dSv_{k-1} w^{\rm h} \|_{L^4_{\rm h}(L^\infty_{\rm v})} \| \dDv_k \nabla_\h v \|_{L^2} \| \dDv_j w \|_{L^4_{\rm h}(L^2_{\rm v})}.
\]
Yet it follows from the definition of $\B(t)$ (see \eqref{eq:prop1.6}) that 
\begin{equation}\label{eq5.20}
   \| \dDv_k \nabla_\h v(t) \|_{L^2} \lesssim d_k(t) 2^{-\f{k}2}\|\nabla_\h v(t)\|_{\dB^{0,\f12}}
   \lesssim d_k (t)2^{-\f{k}2} \bigl(\B'(t)\bigr)^\f12
\end{equation}
from which, Lemma \ref{lemBern} and \eqref{interpolation inequality}, we deduce that for $j\leq -1$
\begin{align*}
\int_0^t &\fg_1^2(\tau) \big| \bigl( \dDv_j (\dTv_{w^{\rm h}} \nabla_\h v) \mid \dDv_j w \bigr) \big| \, d\tau \\
\lesssim &\sum_{|k-j|\leq 4} d_k 2^{-\f{k}2} \int_0^t  \bigl( \B' \bigr)^{\frac12} \fg_1^2(\tau) \| {\dSv_4} \nabla_\h w(\tau) \|_{L^2_\h(L^\oo_\v)}^\f12 \| {\dSv_4} w(\tau) \|_{L^2_\h(L^\oo_\v)}^{\frac12} \\
&\qquad\qquad\qquad\quad\times \| \nabla_\h \dDv_j w(\tau) \|_{L^2}^{\frac12}\|  \dDv_j w(\tau) \|_{L^2}^{\frac12} \, d\tau \\
\lesssim &\Bigl( \sum_{|k-j|\leq 4} d_k 2^{-\frac{k}{2}} \Bigr) d_j  2^{-\frac{j}{2}} \|\fg_1{\dSv_4} \nabla_\h w\|_{\widetilde L^2_t(\dB^{0,\f12})}^\f12  \|\bigl(\B'\bigr)^\f12\fg_1 {\dSv_4} w\|_{\wt L^2_t(\dB^{0,\f12})}^\f12\\
&\qquad\qquad\qquad\quad \times  \|\bigl(\B'\bigr)^\f12\fg_1  w\|_{\wt L^2_t(\dBl^{0,\f12})}^\f12 \| \fg_1\nabla_\h w \|_{\widetilde L^2_t(\dBl^{0,\f12})}^\f12,
\end{align*}
which together with \eqref{eq5.5} and \eqref{eq5.14} ensures  that
\begin{equation}\label{eq5.21}
\begin{aligned}
\int_0^t &\fg_1^2(\tau)\big| \bigl( \dDv_j (\dTv_{w^{\rm h}} \nabla_\h v) \mid \dDv_j w \bigr) \big| \, d\tau 
\lesssim d_j^2  2^{-j}\|\fg_1 \nabla_\h w \|_{\widetilde L^2_t(X)}  \|\bigl(\B'\bigr)^\f12 \fg_1 w\|_{\wt L^2_t(X)}.
\end{aligned}
\end{equation}

On the other hand, it follows from Lemma \ref{lemBern} and the definition of $\B(t)$ (see \eqref{eq:prop1.6})  that
\begin{equation}\label{eq5.22}
\begin{split}
\| \dSv_{k-1} \nabla_\h v(t) \|_{L^2_{\rm h}(L^\infty_{\rm v})}
&\lesssim \sum_{\ell\leq k-2} 2^{\frac{\ell}{2}} \| \dDv_\ell  \nabla_\h v(t) \|_{L^2} \lesssim \|  \nabla_\h v(t) \|_{\dB^{0,\f12}} \lesssim \bigl(\B'(t)\bigr)^\f12 ,
\end{split}
\end{equation}
so that we deduce from \eqref{interpolation inequality} that
\begin{align*}
\int_0^t &\fg_1(\tau) \big| \bigl( \dDv_j (\dTv_{\nabla_\h v} w^{\rm h}) \mid \dDv_j w \bigr) \big| \, d\tau \\
\lesssim &\sum_{|k-j|\leq 4} \int_0^t \fg_1(\tau) \| \dSv_{k-1} \nabla_\h v(\tau) \|_{L^2_{\rm h}(L^\infty_{\rm v})} \| \dDv_k w^{\rm h}(\tau) \|_{L^4_{\rm h}(L^2_{\rm v})} \| \dDv_j w(\tau) \|_{L^4_{\rm h}(L^2_{\rm v})} \, d\tau \\
\lesssim &\sum_{|k-j|\leq 4} \int_0^t  (\B'(\tau))^\f12 \fg_1^2(\tau)\| \dDv_k w(\tau) \|_{L^2}^\f12 \|\nabla_\h \dDv_k w(\tau) \|_{L^2}^\f12 \| \dDv_j w(\tau) \|_{L^2}^\f12 \|\nabla_\h \dDv_j w(\tau) \|_{L^2}^\f12 \, d\tau\\
\lesssim &\Bigl( \sum_{|k-j|\leq 4} d_k 2^{-\frac{k}{2}} \Bigr) d_j  2^{-\frac{j}{2}}\|\bigl(\B'\bigr)^\f12
\fg_1 {\dSv_4} w\|_{\wt L^2_t(\dB^{0,\f12})}^\f12 \|\fg_1 {\dSv_4} \nabla_\h w\|_{\widetilde L^2_t(\dB^{0,\f12})}^\f12 \\
&\qquad \qquad\qquad\qquad\qquad\quad\times  \|\bigl(\B'\bigr)^\f12 \fg_1  w\|_{\wt L^2_t(\dBl^{0,\f12})}^\f12 \|\fg_1\nabla_\h w \|_{\widetilde L^2_t(\dBl^{0,\f12})}^\f12,
\end{align*}
which along with \eqref{eq5.5} and \eqref{eq5.14} ensures  that
\begin{equation}\label{eq5.23}
\begin{aligned}
\int_0^t &\fg_1^2(\tau) \big| \bigl( \dDv_j (\dTv_{\nabla_\h v} w^\h) \mid \dDv_j w \bigr) \big| \, d\tau \lesssim d_j^2  2^{-j}\|\fg_1\nabla_\h w\|_{\widetilde L^2_t(X)} \|\bigl(\B'\bigr)^\f12\fg_1 w\|_{\wt L^2_t(X)}.
\end{aligned}
\end{equation}

Finally, for the remainder term with $j \leq -1$, we deduce from Lemma \ref{lemBern} that
\begin{align*}
\int_0^t & \fg_1^2(\tau)\big| \bigl( \dDv_j \dRv(w^{\rm h}, \nabla_\h v) \mid \dDv_j w \bigr) \big| \, d\tau \\
\lesssim & 2^{\frac{j}{2}} \sum_{k \geq j-3} \int_0^t \fg_1^2(\tau)\| \dDv_k w^{\rm h}(\tau) \|_{L^4_{\rm h}(L^2_{\rm v})} \| \wt{\D}^\v_k \nabla_\h v(\tau) \|_{L^2} \| {\dDv_j } w(\tau) \|_{L^4_{\rm h}(L^2_{\rm v})} \, d\tau.
\end{align*}
Yet similar to  \eqref{eq5.20}, one has
$$
\| \wt{\D}^\v_k \nabla_\h v(t) \|_{L^2}  \lesssim d_k(t) 2^{-\frac{k}{2}} \bigl( \B'(t) \bigr)^{\frac12},
$$
from which and \eqref{interpolation inequality}, we infer 
\begin{align*}
\int_0^t & \fg_1^2(\tau)\big| \bigl( \dDv_j \dRv(w^{\rm h}, \nabla_\h v) \mid \dDv_j w \bigr) \big| \, d\tau \\
\lesssim & 2^{\frac{j}{2}} \sum_{k \geq j-3} \int_0^t  d_k(\tau) 2^{-\frac{k}{2}} \bigl( \B'(\tau) \bigr)^{\frac12}\fg_1^2(\tau)\|\dDv_k  w(\tau)\|_{L^2}^\f12\|\dDv_k \nabla_\h w(\tau)\|_{L^2}^\f12 \\
&\qquad\qquad\times\| \dDv_j w(\tau) \|_{L^2}^\f12\| \nabla_\h \dDv_j w(\tau) \|_{L^2}^\f12 \, d\tau\\
\lesssim & d_j \Bigl( \|\fg_1\nabla_\h w\|_{\widetilde L^2_t(\dBl^{0,\f12})}^\f12\|\bigl(\B'\bigr)^\f12
\fg_1 w\|_{\wt L^2_t (\dBl^{0,\f12})}^\f12\sum_{k = j-3}^{-1}  d_k 2^{-k} \\
&\qquad\quad+  \| \fg_1 \nabla_\h w\|_{\widetilde L^2_t(\dHh^{0,-\f12})}^\f12\|\bigl(\B'\bigr)^\f12
\fg_1w\|_{\wt L^2_t (\dHh^{0,-\f12})}^\f12\sum_{k = 0}^{\oo}  d_k  \Bigr) \\
&\quad\times  \| \bigl(\B'\bigr)^\f12\fg_1 w \|_{\wt L^2_t(\dBl^{0,\f12})}^\f12\|\fg_1 \nabla_\h w\|_{\wt L^2_t(\dBl^{0,\f12})}^\f12 .
\end{align*}
Then we get, by applying \eqref{eq5.9}, that 
\begin{equation} \label{eq5.24}
\begin{aligned}
\int_0^t &\fg_1^2(\tau) \big| \bigl( \dDv_j \dRv(w^{\rm h}, \nabla_\h v) \mid \dDv_j w \bigr) \big| \, d\tau 
\lesssim  d_j^2  2^{-j}\| \fg_1 \nabla_\h w\|_{\widetilde L^2_t(X)}  \|\bigl(\B'\bigr)^\f12 \fg_1 w\|_{\wt L^2_t(X)}.
\end{aligned}
\end{equation}

By summarizing the estimates \eqref{eq5.21}, \eqref{eq5.23}, and \eqref{eq5.24}, we conclude the proof of \eqref{eq5.25}. \end{proof}

\begin{proof}[Proof of Lemma \ref{S4lem4}]
Once again, by applying  Bony's decomposition \eqref{homo bony v} to $w^3\p_3 v$ in the vertical variable, we write
\begin{equation} \label{S4eq6a}
w^3\p_3 v = \dTv_{w^3} \p_3 v + \dTv_{\p_3 v} w^3 + \dRv(w^3, \p_3 v).
\end{equation}

 As we focus on the low frequency part, there holds 
 \begin{equation}\label{eq5.26}
   \|\dDv_k \p_3 v(t)\|_{L^2}\lesssim \|\dDv_k v(t)\|_{L^2}\lesssim d_k(t) 2^{-\f{k}2}\| v(t)\|_{\dB^{0,\f12}} \lesssim d_k(t) 2^{-\f{k}2} \bigl(\B'(t)\bigr)^\f12, \  \text{for} \quad k\leq 4.
 \end{equation}
 Then we get, by using the support properties of the Fourier transform of the terms in $\dTv_{u^3} \p_3 w$ and \eqref{interpolation inequality}, that
\begin{align*}
\int_0^t &\fg_1^2(\tau)\big| \bigl( \dDv_j (\dTv_{w^3} \p_3 v) \mid \dDv_j w \bigr) \big| \, d\tau \\
\lesssim &\sum_{|k-j|\leq 4} \int_0^t \fg_1^2(\tau) \| \dSv_{k-1}  w(\tau) \|_{L^4_\h {(L^\oo_\v)}} \|\dDv_k \p_3 v\|_{L^2} \| \dDv_j w(\tau) \|_{L^4_\h(L^2_\v)}  \, d\tau \\
\lesssim &\sum_{|k-j|\leq 4} d_k 2^{-\f{k}2}\int_0^t  \bigl(\B'(\tau) \bigr)^{\frac12} \fg_1^2(\tau)\| {\dSv_4}  w(\tau) \|_{\dB^{0,\f12}}^\f12 \| \nabla_\h {\dSv_4 } w(\tau) \|_{\dB^{0,\f12}}^\f12 \\
&\qquad\qquad\qquad\quad\times \| \dDv_j w(\tau) \|_{L^2}^{\frac12} \| \nabla_\h \dDv_j w(\tau) \|_{L^2}^{\frac12} \, d\tau \\
\lesssim &\Bigl( \sum_{|k-j|\leq 4} d_k 2^{-\frac{k}{2}} \Bigr) d_j  2^{-\frac{j}{2}}  \|\bigl(\B'\bigr)^\f12
\fg_1{\dSv_4 } w\|_{\widetilde L^2_t(\dB^{0,\f12})}^\f12 \|\fg_1{\dSv_4}\nabla_\h w\|_{\wt L^2_t(\dB^{0,\f12})}^\f12 \\
&\qquad\qquad\qquad \times  \|\bigl(\B'\bigr)^\f12\fg_1  w\|_{\wt L^2_t(\dBl^{0,\f12})}^\f12\| \fg_1\nabla_\h w\|_{\widetilde L^2_t(\dBl^{0,\f12})}^\f12,
\end{align*}
which together with \eqref{eq5.5} and \eqref{eq5.14} ensures that 
\begin{equation}\label{eq5.27}
\begin{aligned}
\int_0^t& \fg_1^2(\tau) \big| \bigl( \dDv_j (\dTv_{w^3} \p_3 v) \mid \dDv_j w \bigr) \big| \, d\tau \lesssim d_j^2  2^{-j}\| \fg_1 \nabla_\h w \|_{\widetilde L^2_t(X)} \|\bigl(\B'\bigr)^\f12\fg_1 w\|_{\wt L^2_t(X)}.
\end{aligned}
\end{equation}

It follows from a similar  derivation  of \eqref{eq5.26} that 
\begin{equation}\notag
   \| \dSv_{k-1} \p_3 v(t) \|_{L^2_{\rm h}(L^\infty_{\rm v})}\lesssim \| v(t)\|_{L^2_\h(L^\oo_\v)}\lesssim \| v(t)\|_{\dB^{0,\f12}} \lesssim  \bigl(\B'(t)\bigr)^\f12, \quad \text{for} \quad k\leq 4,
 \end{equation}
from which and  \eqref{interpolation inequality}, we infer
\begin{align*}
\int_0^t & \fg_1^2(\tau)\big| \bigl( \dDv_j (\dTv_{\p_3 v} w^3) \mid \dDv_j w \bigr) \big| \, d\tau \\
\lesssim &\sum_{|k-j|\leq 4} \int_0^t \fg_1^2(\tau) \| \dSv_{k-1} \p_3 v(\tau) \|_{L^2_{\rm h}(L^\infty_{\rm v})} \| \dDv_k w^3(\tau) \|_{L^4_{\rm h}(L^2_{\rm v})} \| \dDv_j w(\tau) \|_{L^4_{\rm h}(L^2_{\rm v})} \, d\tau \\
\lesssim &\sum_{|k-j|\leq 4}\int_0^t \bigl(\B'\bigr)^\f12 \fg_1^2(\tau) \| \dDv_k w(\tau) \|_{L^2}^\f12 \|\nabla_\h \dDv_k  w(\tau) \|_{L^2}^\f12 \\
&\qquad\qquad\times\| \dDv_j w(\tau) \|_{L^2}^\f12 \|\nabla_\h \dDv_j w(\tau) \|_{L^2}^\f12 \, d\tau\\
\lesssim &\Bigl( \sum_{|k-j|\leq 4} d_k 2^{-\frac{k}{2}} \Bigr) d_j  2^{-\frac{j}{2}}  \|\bigl(\B'\bigr)^\f12\fg_1 {\dSv_4}  w \|_{\widetilde L^2_t(\dB^{0,\f12})}^\f12 \|\fg_1{\dSv_4} w\|_{\wt L^2_t(\dB^{0,\f12})}^\f12 \\
&\qquad\qquad\qquad \times  \|\bigl(\B'\bigr)^\f12 \fg_1 w\|_{\wt L^2_t(\dBl^{0,\f12})}^\f12\| \fg_1\nabla_\h w\|_{\widetilde L^2_t(\dBl^{0,\f12})}^\f12.
\end{align*}
By using \eqref{eq5.5} and \eqref{eq5.14}, we arrive at
\begin{equation} \label{eq5.28}
\begin{aligned}
\int_0^t &\fg_1^2(\tau) \big| \bigl( \dDv_j (\dTv_{\p_3 w} u^3) \mid \dDv_j w \bigr) \big| \, d\tau \lesssim d_j^2  2^{-j}\| \fg_1 \nabla_\h w\|_{\widetilde L^2_t(X)} \|\bigl(\B'(\tau)\bigr)^\f12\fg_1  w\|_{\wt L^2_t(X)}.
\end{aligned}
\end{equation}

Finally, for the remainder term with $j \leq -1$, we deduce from Lemma \ref{lemBern} and $\p_3 w_3=-\dive_\h w_\h$ that
\begin{align*}
\int_0^t &\fg_1^2(\tau) \big| \bigl( \dDv_j \dRv(w^3, \p_3 v) \mid \dDv_j w \bigr) \big| \, d\tau \\
\lesssim & 2^{\frac{j}{2}} \sum_{k \geq j-3} \int_0^t \fg_1^2(\tau)\| \dDv_k w^3(\tau) \|_{L^2} \| \wt{\D}^\v_k \p_3 v(\tau) \|_{L^4_\h(L^2_\v)} \| \dDv_j w(\tau) \|_{L^4_{\rm h}(L^2_{\rm v})} \, d\tau\\ 
\lesssim & 2^{\frac{j}{2}} \sum_{k \geq j-3}\int_0^t \fg_1^2(\tau)\| \dDv_k \p_3w^3(\tau) \|_{L^2} \| \wt{\D}^\v_k  v(\tau) \|_{L^4_\h(L^2_\v)} \| \dDv_j w(\tau) \|_{L^4_{\rm h}(L^2_{\rm v})} \, d\tau\\
\lesssim & 2^{\frac{j}{2}} \sum_{k \geq j-3}\int_0^t \fg_1^2(\tau)\| \dDv_k \nabla_\h w^\h(\tau) \|_{L^2} \| \wt{\D}^\v_k  v(\tau) \|_{L^4_\h(L^2_\v)} \| \dDv_j w(\tau) \|_{L^4_{\rm h}(L^2_{\rm v})} \, d\tau.
\end{align*}
Then we get, by repeating the proof of \eqref{eq5.11}, that
\begin{equation} \label{eq5.29}
\begin{aligned}
\int_0^t  \fg_1^2(\tau)\big| \bigl( \dDv_j \dRv(w^3, \p_3 v) &\mid \dDv_j w \bigr) \big| \, d\tau \lesssim d_j^2  2^{-j} \|\bigl(\B'\bigr)^\f12 \fg_1 w\|_{\wt L^2_t({X})}^\f12\| \fg_1\nabla_\h w \|_{\widetilde L^2_t(X)}^\f32.
\end{aligned}
\end{equation}

By summarizing the estimates \eqref{eq5.27}, \eqref{eq5.28}, and \eqref{eq5.29}, we achieve \eqref{eq5.30}.  This completes the proof of Lemma \ref{S4lem4}.\end{proof}

Our second estimate is concerned with the high frequency part of $w.$

\begin{prop}\label{S4prop2}
  {\sl  Let $u,v$ be smooth enough solutions of $(ANS)$ on $[0,T]$ and $w\eqdefa u-v$. Let $\B(t)$ be defined by \eqref{eq:prop1.6}. If $\|w(t)\|_{X}\leq 2^{-10}$ for all $0\leq t\leq T$, then there exists some constant $C>0$ such that for any {$C_3>0$} and $\fg_1(t)\eqdefa e^{-{C_3} \B(t)},$
\begin{equation}\label{eq:prop5.2}
\begin{aligned}
    &\|\fg_1 w\|_{\wt L^\oo_T(\dHh^{0,-\f12})}^2 +2 \|\fg_1\nabla_\h  w\|_{\wt L^2_T(\dHh^{0,-\f12})}^2+2{C_3} \|\bigl(\B'\bigr)^\f12 \fg_1  w\|_{\wt L^2_T(\dHh^{0,-\f12})}^2 \\
    &\leq \| w_{0}\|_{\dHh^{0,-\f12}}^2 + C\sum_{\kappa=1}^3 \|\fg_1\nabla_\h w\|_{\wt L^2_T(X)}^{2-\f\kappa2}\\
  &\qquad\times \Bigl(\int_0^T \B'(t)\bigl( -\ln\|\fg_1w\|_{\wt L^\oo_t(X)}^2\bigr) \ln\bigl( -\ln\|\fg_1 w\|_{\wt L^\oo_t(X)}^2\bigr)\|\fg_1 w\|_{\wt L^\oo_t(X)}^2\,dt\Bigr)^\f{\kappa}4.
\end{aligned}
\end{equation}
  }
\end{prop}

\begin{proof}
 As we are dealing with the classical Sobolev type norm $\|w\|_{\dHh^{0,-\frac12}}$, which brings the advantage of avoiding the technical use of time-weighted Chemin–Lerner type norm. More precisely, instead of analyzing the nonlinear terms in \eqref{eq5.3} for each $j\geq 0,$ we first take the supremum of \eqref{eq5.3} over $t \in [0,T]$, then multiply it by $2^{-j}$, and finally sum up the resulting inequalities for $j \geq 0$ to obtain
\begin{equation}\label{eq5.32}
\begin{aligned}
&\|\fg_1 w\|_{\widetilde L^\infty_T(\dHh^{0,-\frac12})}^2 + 2 \|\fg_1 \nabla_{\mathrm{h}} w\|_{L^2_T(\dHh^{0,-\frac12})}^2  + 2{C_3} \| (\B'(t))^{\frac12} \fg_1 w\|_{L^2_T(\dHh^{0,-\frac12})}^2 \\
&\leq \| w_{0}\|_{\dHh^{0,-\frac12}}^2 + 2 \int_0^T \fg_1^2(t)\sum_{j \geq 0} 2^{-j} \bigl| \bigl( \dot{\Delta}_j^{\mathrm{v}} (u \cdot \nabla w + w \cdot \nabla v) \mid \dot{\Delta}_j^{\mathrm{v}} w \bigr) \bigr| \, dt.
\end{aligned}
\end{equation}
Here we used the fact that the summation over $j$ commutes with the integration over time, which enables us to consider the inner product for nonlinear terms in $\dot{H}^{0,-\frac12}$.   By virtue of \eqref{S4eq1a},  below we handle  the nonlinear term in \eqref{eq5.32} by dealing with the following four  parts separately.\smallskip

 \begin{lem}\label{S4lem5}
 {\sl Under the assumptions of Proposition \ref{S4prop2}, one has
 \begin{equation}\label{eq5.41}
\begin{aligned}
  \sum_{j\geq 0}2^{-j}\big| \bigl( \dDv_j (u^{\rm h} \cdot \nabla_\h w) &\mid \dDv_j w \bigr) \big| 
  \lesssim \sum_{\kappa=1}^3 \|\nabla_\h w(t)\|_{X}^{2-\f\kappa2}\\
  &\times \Bigl(\B'(t)\bigl(-\ln \|w(t)\|_{X}^2\bigr)\ln\bigl(-\ln \|w(t)\|_{X}^2\bigr)\| w(t)\|_{X}^2\Bigr)^{\f\kappa4}.
\end{aligned}
\end{equation}
}\end{lem}

 \begin{lem}\label{S4lem6}
 {\sl Under the assumptions of Proposition \ref{S4prop2}, one has 
 \begin{equation}\label{eq5.47}
\begin{aligned}
 &  \sum_{j\geq 0}2^{-j} \big| \bigl( \dDv_j (u^3 \p_3 w) \mid \dDv_j w \bigr) \big|  \lesssim \sum_{\kappa=2}^3 \|\nabla_\h w(t)\|_{X}^{2-\f\kappa2}
\Bigl(\B'(t)\bigl(-\ln \|w(t)\|_{X}^2\bigr)\|w(t)\|_{X}^2\Bigr)^{\f\kappa4}.
\end{aligned}
\end{equation}}
\end{lem}

 \begin{lem}\label{S4lem7}
 {\sl Under the assumption of Proposition \ref{S4prop2}, one has
 \begin{equation}\label{eq5.52}
\begin{aligned}
  \sum_{j\geq 0}2^{-j}   \big| \bigl( \dDv_j ({ w^{\rm h} \cdot \nabla_\h u}) \mid \dDv_j w \bigr) \big|\lesssim &\sum_{\kappa=2}^3\|\nabla_\h w(t)\|_{X}^{2-\f\kappa2}\\
  &\times\Bigl(\B'(t)\bigl(-\ln \|w(t)\|_{X}^2\bigr)\|w(t)\|_{X}^2\Bigr)^{\f\kappa4}.
\end{aligned}
\end{equation} }\end{lem}

 \begin{lem}\label{S4lem8}
 {\sl Under the assumptions of Proposition \ref{S4prop2}, one has
 \begin{equation}\label{eq5.59}\begin{aligned}
  \sum_{j\geq 0}2^{-j} \big| \bigl( \dDv_j (w^3 \p_3 v) &\mid \dDv_j w \bigr) \big| 
 \lesssim  \sum_{\kappa=1}^3\|\nabla_\h w(t)\|_{X}^{2-\f\kappa2}\\
  &\times\Bigl(\B'(t)\bigl(-\ln \|w(t)\|_{X}^2\bigr)\ln\bigl(-\ln \|w(t)\|_{X}^2\bigr)\|w(t)\|_{X}^2\Bigr)^{\f\kappa4}.
\end{aligned}
\end{equation} }\end{lem} 
  
We postpone the proof of the above lemmas after we finish the proof of Proposition \ref{S4prop2}.\\

By substituting the estimates (\ref{eq5.41}-\ref{eq5.59}) (with $w$ there being replaced by $\fg_1w$) into \eqref{eq5.32},  we obtain
  \begin{equation}\label{eq5.60}
  \begin{aligned}
    &\|\fg_1 w\|_{\wt L^\oo_T(\dHh^{0,-\f12})}^2 +2 \|\fg_1 \nabla_\h  w\|_{\wt L^2_T(\dHh^{0,-\f12})}^2+2{C_3} \|\bigl(\B'\bigr)^\f12 \fg_1  w\|_{\wt L^2_T(\dHh^{0,-\f12})}^2 \\
    &\leq \| w_{0}\|_{\dHh^{0,-\f12}}^2 + C\sum_{\kappa=1}^3\int_0^T  \|\fg_1\nabla_\h w(t)\|_{X}^{2-\f\kappa2}\\
  &\qquad\times\Bigl(\B'(t)\bigl(-\ln \|\fg_1 w(t)\|_{X}^2\bigr)\ln\bigl(-\ln \|\fg_1 w(t)\|_{X}^2\bigr)
  \|\fg_1w(t)\|_{X}^2\Bigr)^{\f\kappa4}\,dt.
  \end{aligned}
\end{equation}
Observing that \beq \label{P1t}
\Phi_1(\tau)\eqdefa-\tau\ln\tau\ln\left(-\ln\tau\right)\quad\mbox{is an increasing function of }\ \ \tau \ \ \mbox{if} \ \ 0<\tau<2^{-10},
\eeq
from which and the Minkousi inequality, we infer
\begin{align*}
  &\int_0^T \sum_{\kappa=1}^3\|\fg_1\nabla_\h w(t)\|_{X}^{2-\f\kappa2}\Bigl(\B'(t)\bigl(-\ln \|\fg_1w(t)\|_{X}^2\bigr)\ln\bigl(-\ln \|\fg_1w(t)\|_{X}^2\bigr)\|\fg_1w(t)\|_{X}^2\Bigr)^{\f\kappa4}\,dt\\
  &\lesssim \sum_{\kappa=1}^3\Bigl(\int_0^T \B'(t)\bigl(-\ln\|\fg_1w(t)\|_{X}^2\bigr) \ln\bigl(-\ln \|\fg_1w(t)\|_{X}^2\bigr)\|\fg_1 w(t)\|_{X}^2\,dt\Bigr)^\f{\kappa}4  \|\fg_1\nabla_\h w\|_{ L^2_T(X)}^{2-\f\kappa2}\\
  &\lesssim \sum_{\kappa=1}^3\Bigl(\int_0^T \B'(t)\bigl( -\ln\|\fg_1 w\|_{\wt L^\oo_t(X)}^2\bigr) \ln\bigl( -\ln\|\fg_1w\|_{\wt L^\oo_t(X)}^2\bigr) \|\fg_1 w\|_{\wt L^\oo_t(X)}^2\,dt\Bigr)^\f{\kappa}4 
  \|\fg_1\nabla_\h w{\|}_{\wt L^2_T(X)}^{2-\f\kappa2}.
\end{align*}

By substituting the above estimate into \eqref{eq5.60}, we arrive at \eqref{eq:prop5.2}, which finishes the proof of Proposition \ref{S4prop2}.
\end{proof}

Let us now present the proof of Lemmas \ref{S4lem5}-\ref{S4lem8}.

  \begin{proof}[Proof of  Lemma \ref{S4lem5}]  We first get, by applying  Bony's decomposition \eqref{homo bony v} to write $u^{\rm h} \cdot \nabla_\h w$  as in \eqref{S4eq2a}.

Considering the support properties to the Fourier transform of the terms in $\dTv_{u^{\rm h}} \nabla_\h w$, we get, by using  \eqref{eq5.4} and \eqref{interpolation inequality}, that for $j\geq 0$,
\begin{equation}\notag
\begin{aligned}
  \big| \bigl( &\dDv_j (\dTv_{u^{\rm h}} \nabla_\h w) \mid \dDv_j w \bigr) \big| \,  \\
\lesssim &\sum_{|k-j|\leq 4} \|\dSv_{k-1}u^\h(t)\|_{L^4_\h(L^\oo_\v)} \| \dDv_k \nabla_\h w(t) \|_{L^2} \| \dDv_j w(t) \|_{L^2}^{\frac12} \| \nabla_\h \dDv_j w(t) \|_{L^2}^{\frac12}  \\
\lesssim &\sum_{|k-j|\leq 4}   \bigl(\B'(t) \bigr)^{\frac14} \| \dDv_k \nabla_\h w(t) \|_{L^2} \| \dDv_j w(t) \|_{L^2}^{\frac12} \| \nabla_\h \dDv_j w(t) \|_{L^2}^{\frac12}\\
\lesssim &\Bigl( \sum_{|k-j|\leq 4} c_k(t) 2^{\frac{k}{2}} \Bigr)  c_j(t)  2^{\frac{j}{2}}\bigl( \B'(t) \bigr)^{\frac14} \|(\mathrm{Id}- \dSv_{-5}) \nabla_\h w(t) \|_{\dH^{0,-\f12}}  \|  w(t)\|_{\dHh^{0,-\f12}}^\f12 \|  \nabla_\h w(t) \|_{\dHh^{0,-\f12}}^\f12.
\end{aligned}
\end{equation}
Notice that
\begin{equation}\label{eq5.33}
  \begin{aligned}
    \|(\mathrm{Id}- \dSv_{-5}) \nabla_\h w(t) \|_{\dH^{0,-\f12}}
    &\leq \|  \nabla_\h w(t) \|_{\dHh^{0,-\f12}} +C\sum_{j=-5}^{-1} 2^{-\f{j}2} \|\dDv_j \nabla_\h w(t) \|_{ L^2}\\
    &\leq \|  \nabla_\h w(t) \|_{\dHh^{0,-\f12}} +C \| \nabla_\h w(t) \|_{ \dBl^{0,\f12}}\leq C \| \nabla_\h w(t) \|_{X},
  \end{aligned}
\end{equation}
we thus deduce that
\begin{equation}\label{eq5.34}
\begin{aligned}
  \big| \bigl( \dDv_j (\dTv_{u^{\rm h}} &\nabla_\h w) \mid \dDv_j w \bigr) \big|  \lesssim c_j^2(t)  2^{j}  \bigl(\B'(t)\bigr)^\f14 \|  w(t)\|_{{\color{red}X}}^\f12 
  \|\nabla_\h w(t) \|_{X}^\f32.
\end{aligned}
\end{equation}

Wheres by applying \eqref{eq5.7} and \eqref{interpolation inequality}, we find
\begin{align*}
 \big| \bigl(& \dDv_j (\dTv_{\nabla_\h w} u^{\rm h}) \mid \dDv_j w \bigr) \big|  \\
\lesssim &\sum_{|k-j|\leq 4}  \| \dSv_{k-1} \nabla_\h w(t) \|_{L^2_{\rm h}(L^\infty_{\rm v})} \| \dDv_k u^{\rm h}(t) \|_{L^4_{\rm h}(L^2_{\rm v})} \| \dDv_j w(t) \|_{L^4_{\rm h}(L^2_{\rm v})}  \\
\lesssim &\sum_{|k-j|\leq 4} d_k(t) 2^{-\f{k}2} \bigl(\B'(t)\bigr)^\f14\| \dSv_{k-1} \nabla_\h w(t) \|_{L^2_{\rm h}(L^\infty_{\rm v})} \| \dDv_j w(t) \|_{L^2}^\f12 \|\nabla_\h \dDv_j w(t) \|_{L^2}^\f12 \\
\lesssim &\Bigl( \sum_{|k-j|\leq 4} d_k(t) 2^{\frac{k}{2}} \Bigr) c_j(t)  2^{\frac{j}{2}} \bigl(\B'(t)\bigr)^\f14 \sup_{k\geq -4}\Bigl(2^{-k}\| \dSv_{k-1} \nabla_\h w(t) \|_{L^2_{\rm h}(L^\infty_{\rm v})}\Bigr)\\
& \qquad\qquad\qquad\qquad\qquad\qquad\qquad\qquad\qquad\times \|  w(t)\|_{\dHh^{0,-\f12}}^\f12\|  \nabla_\h w(t) \|_{\dHh^{0,-\f12}}^\f12,
\end{align*}
Yet it follows from Lemma \ref{lemBern} that for $k\geq -4$,
\begin{equation}\label{eq5.36}
\begin{split}
\| \dSv_{k-1} \nabla_\h w(t) \|_{L^2_{\rm h}(L^\infty_{\rm v})}
&\lesssim \sum_{\ell\leq k-2} 2^{\frac{\ell}{2}} \|  \dDv_{\ell} \nabla_\h w(t) \|_{L^2} \\
& \lesssim \|\nabla_\h w(t)\|_{\dBl^{0,\f12}}+\|\nabla_\h w(t)\|_{\dHh^{0,-\f12}}\sum_{\ell=0}^{\max(k-2,0)} c_\ell(t) 2^\ell \\
&\lesssim 2^{k} \Bigl(\|\nabla_\h w(t)\|_{\dBl^{0,\f12}}+\|\nabla_\h w(t)\|_{\dHh^{0,-\f12}}\Bigr) ,
\end{split}
\end{equation}
which together with the fact: $\sum_{|k-j|\leq 4} d_k 2^{\f{k}2}\lesssim c_j 2^{\f{j}2},$ implies
\begin{equation} \label{eq5.37}
\begin{aligned}
  \big| \bigl( \dDv_j (\dTv_{\nabla_\h w} u^{\rm h})& \mid \dDv_j w \bigr) \big| \lesssim c_j^2(t)  2^{j} \bigl(\B'(t)\bigr)^\f14  \|  w(t)\|_{{\color{red}X}}^\f12 
  \|\nabla_\h w(t) \|_{X}^\f32.
\end{aligned}
\end{equation}

Finally, for the remainder term with $j \geq 0$, we deduce from Lemma \ref{lemBern} 
that
\begin{equation}\notag
  \begin{aligned}
  \big| \bigl( \dDv_j \dRv(u^{\rm h},& \nabla_\h w) \mid \dDv_j w \bigr) \big|  \\
\lesssim &  2^{\frac{j}{2}} \sum_{k \geq j-3} \inf_{p\in [4,\oo[} \Bigl(\| \wt{\D}^\v_k u^{\rm h}(t) \|_{L^p_{\rm h}(L^2_{\rm v})} \| \dDv_k \nabla_\h w(t) \|_{L^2} \| \dDv_j w(t) \|_{L^{\f{2p}{p-2}}_{\rm h}(L^2_{\rm v})}\Bigr).
\end{aligned}
\end{equation}
Recall the Sobolev embedding inequality with sharp coeffient:
\begin{equation}\label{sharp interpolation}
  \|f\|_{L^p(\R^2)} \leq C \sqrt{p} \|f\|_{L^2(\R^2)}^\f{2}p \|\nabla_\h f\|_{L^2(\R^2)}^{1-\f{2}p},
\end{equation}
we obtain for any $p\geq 4$ that,
\begin{equation}\label{eq5.38a}
  \| \wt{\D}^\v_k u^{\rm h}(t) \|_{L^p_{\rm h}(L^2_{\rm v})}\lesssim d_k(t)\sqrt{p}2^{-\f{k}2} \|u(t)\|_{\dB^{0,\f12}}^\f{2}p \|\nabla_\h u(t)\|_{\dB^{0,\f12}}^{1-\f2p}\lesssim d_k(t)\sqrt{p} 2^{-\f{k}2} \bigl(\B'(t)\bigr)^{\f12-\f1p},
\end{equation}
and 
\begin{equation}\label{eq5.38b}
  \| \dDv_j w(t) \|_{L^{\f{2p}{p-2}}_{\rm h}(L^2_{\rm v})} \lesssim \| \dDv_j w(t) \|_{L^2}^{1-\f2p} \| \dDv_j \nabla_\h w(t) \|_{L^2}^{\f2p}.
\end{equation}
As a result, we arrive at
\begin{equation}\label{eq5.38}
  \begin{aligned}
 \big| \bigl( &\dDv_j \dRv(u^{\rm h}, \nabla_\h w) \mid \dDv_j w \bigr) \big|  
\lesssim   2^{\frac{j}{2}} \sum_{k \geq j-3} d_k(t) 2^{-\f{k}2} \\
&\qquad\times\inf_{p\in [4,\oo[}\Bigl(\sqrt{p}\bigl(\B'(t)\bigr)^{\f12-\f1p}\| \dDv_k \nabla_\h w(t) \|_{L^2} \| \dDv_j w(t) \|_{L^2}^{1-\f2p}\| \nabla_\h \dDv_j w(t) \|_{L^2}^{\f2p} \Bigr).
\end{aligned}
\end{equation}
 Motivated by \cite{Paicu}, for any fixed $t,$ we introduce  some integer $J(t)$ such that 
\begin{equation}\label{def:J(t)}
  2^{-J(t)}\leq \|w(t)\|_{\dHh^{0,-\f12}}\leq 2^{1-J(t)} \quad\Longleftrightarrow \quad J(t) \approx -\ln  \|w(t)\|_{\dHh^{0,-\f12}}^2.
\end{equation}
Under the assumption that  $\|w(t)\|_X\leq 2^{-10}$, we always have $J(t)\geq 10$. For $0\leq j\leq J(t)$ 
and $p\geq 4$, we have
\begin{align*}
  \sqrt{p}&2^{\frac{j}{2}} \sum_{k \geq j-3} d_k(t) 2^{-\f{k}2}  \bigl(\B'(t)\bigr)^{\f12-\f1p}\| \dDv_k \nabla_\h w(t) \|_{L^2} \| \dDv_j w(t) \|_{L^2}^{1-\f2p}\| \nabla_\h \dDv_j w(t) \|_{L^2}^{\f2p} \\
  \lesssim &\sqrt{p} c_j(t) 2^j\bigl(\B'(t)\bigr)^{\f12-\f1p}\Bigl(  \|\nabla_\h w(t)\|_{\dBl^{0,\f12}} \sum_{-3\leq k\leq -1} d_k(t) 2^{-k} +  \|\nabla_\h w(t)\|_{\dHh^{0,-\f12}} \sum_{ k\geq0} d_k(t) c_k(t) \Bigr) \\
  &\qquad\qquad\times \| w(t)\|_{\dHh^{0,-\f12}}^{1-\f2p} \|\nabla_\h w(t)\|_{\dHh^{0,-\f12}}^\f2p,
  \end{align*}
  from which { and $\sum_{j=0}^{J(t)}c_j(t)\lesssim (J(t))^\f12$}, for any $p\geq 4,$ we infer
  \begin{align*}
 & \sum_{j=0}^{J(t)}2^{-j} \big| \bigl( \dDv_j \dRv(u^{\rm h}, \nabla_\h w) \mid \dDv_j w \bigr) \big| \\  
  &\lesssim \|\nabla_\h w(t)\|_{X} \Bigl(\B'(t)\bigl(J(t)p\bigr)^{\f{p}{p-2}}\| w(t)\|_{\dHh^{0,-\f12}}^2\Bigr)^{\f12-\f1p} \|\nabla_\h w(t)\|_{\dHh^{0,-\f12}}^\f2p.
\end{align*}
Yet by taking the special choice $p=\ln J(t)+2,$  one has
$$
  \bigl(J(t)p\bigr)^{\f{p}{p-2}} \lesssim \bigl(J(t)\ln J(t)\bigr)^{1+\f2{\ln J(t)}}\lesssim J(t)\ln J(t) \Bigl(J(t)\ln J(t)\Bigr)^{\f2{\ln J(t)}} \lesssim J(t)\ln J(t),
$$
where we used $(e^\alpha \alpha)^{\f2\alpha}= e^2 \alpha^{\f2\alpha}\leq e^{2+\f2e} $ with $\alpha=\ln J(t)$. Together with \eqref{def:J(t)} and the Young's inequality, we conclude that \begin{align*}
 \sum_{j=0}^{J(t)}&2^{-j} \big| \bigl( \dDv_j \dRv(u^{\rm h}, \nabla_\h w) \mid \dDv_j w \bigr) \big| \lesssim \sum_{\kappa=1}^2\|\nabla_\h w(t)\|_{X}^{2-\f\kappa2}\\
  &\qquad\qquad\times\Bigl(\B'(t)\bigl(-\ln \|w(t)\|_{\dHh^{0,-\f12}}^2\bigr)\ln\bigl(-\ln \|w(t)\|_{\dHh^{0,-\f12}}^2\bigr)\| w(t)\|_{\dHh^{0,-\f12}}^2\Bigr)^{\f\kappa4} .
\end{align*}

On the other hand, when $j\geq J(\tau)$, due to $w=u-v$, we have
\begin{align*}
  \sum_{k \geq j-3} d_k(t) 2^{-\f{k}2}  \| \dDv_k \nabla_\h w(t) \|_{L^2}
  \lesssim \|\nabla_\h (u,v)(t)\|_{\dB^{0,\f12}} \sum_{k\geq j-3}d_k(t) 2^{-k}\lesssim c_j(t) 2^{-J(t)} \bigl(\B'(t)\bigr)^\f12.
\end{align*}
then by taking 
 $p=4$ in \eqref{eq5.38}, we find
\begin{align*}
  \big| \bigl( \dDv_j \dRv(u^{\rm h}, \nabla_\h w) \mid \dDv_j w \bigr) \big|
  &\lesssim c_j^2(t) 2^j 2^{-J(t)} \bigl(\B'(t)\bigr)^\f34  \| w(t)\|_{\dHh^{0,-\f12}}^\f12 \|\nabla_\h w(t)\|_{\dHh^{0,-\f12}}^\f12 \\
  &\lesssim c_j^2(t) 2^j \bigl(\B'(t)\bigr)^\f34   \| w(t)\|_{\dHh^{0,-\f12}}^\f32 \|\nabla_\h w(t)\|_{\dHh^{0,-\f12}}^\f12.
\end{align*}

By combining these two estimates for the different regimes of $j$, we deduce from the Minkowski inequality that 
\begin{align*}
 &\sum_{j\geq 0}2^{-j} \big| \bigl( \dDv_j \dRv(u^{\rm h}, \nabla_\h w) \mid \dDv_j w \bigr) \big|  
\lesssim  \sum_{\kappa=1}^3\|\nabla_\h w(t)\|_{X}^{2-\f\kappa2}\\
  &\qquad\qquad\times\Bigl(\B'(t)\bigl(-\ln \|w(t)\|_{\dHh^{0,-\f12}}^2\bigr)\ln\bigl(-\ln \|w(t)\|_{\dHh^{0,-\f12}}^2\bigr)\| w(t)\|_{\dHh^{0,-\f12}}^2\Bigr)^{\f\kappa4}.
\end{align*}
which together with \eqref{P1t} ensures that
\begin{equation}\label{eq5.39}
\begin{aligned}
 &\sum_{j\geq 0}2^{-j} \big| \bigl( \dDv_j \dRv(u^{\rm h}, \nabla_\h w) \mid \dDv_j w \bigr) \big|  
\lesssim  \sum_{\kappa=1}^3\|\nabla_\h w(t)\|_{X}^{2-\f\kappa2}\\
  &\qquad\qquad\qquad\times\Bigl(\B'(t)\bigl(-\ln \|w(t)\|_{X}^2\bigr)\ln\bigl(-\ln \|w(t)\|_{X}^2\bigr)\| w(t)\|_{X}^2\Bigr)^{\f\kappa4}.
\end{aligned}
\end{equation}

By summarizing the estimates \eqref{eq5.34}, \eqref{eq5.37}, and \eqref{eq5.39}, we conclude the proof of  \eqref{eq5.41}.
\end{proof}

 \begin{proof}[Proof of Lemma \ref{S4lem6}]  We get, by applying  Bony's decomposition \eqref{homo bony v} to  decompose $u^3 \p_3 w$  as in \eqref{S4eq3a}
 
For the paraproduct $\dTv_{u^3} \p_3 w$, by applying the commutator argument as in \eqref{eq3.8a} and using \eqref{interpolation inequality},  we find
  \begin{align*}
|\bigl( \dDv_j (T^{\rm v}_{u^3} \partial_3 w) \mid  \dDv_j w \bigr)|
\lesssim&  \sum_{|k-j|\leq 4} \|\nabla_\h u(t)\|_{\dB^{0,\f12}} \| \dDv_k w(t) \|_{L^4_{\rm h}(L^2_{\rm v})} \|  \dDv_j w(t)\|_{L^4_{\rm h}(L^2_{\rm v})}\\
\lesssim &\sum_{|k-j|\leq 4} \bigl( \B'(t) \bigr)^{\frac12} \| \dDv_k w(t) \|_{L^2}^\f12\| \nabla_\h\dDv_k w(t) \|_{L^2}^\f12  \|  \dDv_j w(t) \|_{L^2}^\f12\\
&\qquad\qquad\times \| \nabla_\h \dDv_j w(t) \|_{L^2}^\f12\\
\lesssim& \Bigl( \sum_{|k-j|\leq 4} c_k(t) 2^{\frac{k}{2}} \Bigr) c_j(t)  2^{\frac{j}{2}}  \bigl( B'(t) \bigr)^{\frac12} \| (\mathrm{Id}-\dSv_{-5})  w(t)\|_{\dH^{0,-\f12}}^\f12\\
& \times \|(\mathrm{Id}- \dSv_{-5}) \nabla_\h w(t) \|_{\dH^{0,-\f12}}^\f12   \|  w(t)\|_{\dHh^{0,-\f12}}^\f12\|  \nabla_\h w(t) \|_{\dHh^{0,-\f12}}^\f12.
\end{align*}
Similar  to the derivation of \eqref{eq5.33}, one has
\begin{equation}\label{eq5.42}
  \begin{aligned}
    \|(\mathrm{Id}- \dSv_{-5})  w(t) \|_{\dH^{0,-\f12}}
    &\leq \|  w(t) \|_{\dHh^{0,-\f12}} +C\sum_{j=-5}^{-1} 2^{-\f{j}2} \|\dDv_j  w(t) \|_{L^2}\\
    &\leq \|   w(t) \|_{\dHh^{0,-\f12}} +C \| w(t) \|_{\dBl^{0,\f12}}\leq C \| w(t) \|_{X},
  \end{aligned}
\end{equation}
we thus obtain
\begin{equation}\label{eq5.43}
\begin{aligned}
  \big| \bigl( \dDv_j (\dTv_{u^3} \p_3 w) \mid \dDv_j w \bigr) \big| \lesssim &c_j^2(t)  2^{j}  \bigl( \B'(t) \bigr)^{\frac12} \|  \nabla_\h w(t) \|_{X}\|  w(t)\|_{X}.
\end{aligned}
\end{equation}

Along the same line, we get, by using Lemma \ref{lemBern}, $\p_3 u^3=-\dive_\h u^\h$, \eqref{eq5.20} (with $v$ being replaced by $u$) and \eqref{interpolation inequality}, that
\begin{align*}
  \big| \bigl( &\dDv_j (\dTv_{\p_3 w} u^3) \mid \dDv_j w \bigr) \big|  \\
\lesssim &\sum_{|k-j|\leq 4}   \| \dSv_{k-1} \p_3 w(t) \|_{L^4_{\rm h}(L^\infty_{\rm v})} \| \dDv_k u^3(t) \|_{L^2} \| \dDv_j w(t) \|_{L^4_{\rm h}(L^2_{\rm v})} \\
\lesssim &\sum_{|k-j|\leq 4}  \| \dSv_{k-1}  w(t) \|_{L^4_{\rm h}(L^\infty_{\rm v})} \| \dDv_k \dive_\h u^\h(t) \|_{L^2} \| \dDv_j w(t) \|_{L^4_{\rm h}(L^2_{\rm v})} \\
\lesssim &\sum_{|k-j|\leq 4} d_k(t) 2^{-\f{k}2} \bigl(\B'(\tau)\bigr)^\f12\| \dSv_{k-1}  w(t) \|_{L^4_{\rm h}(L^\infty_{\rm v})} \| \dDv_j w(t) \|_{L^2}^\f12 \|\nabla_\h \dDv_j w(t) \|_{L^2}^\f12 \\
\lesssim &\Bigl( \sum_{|k-j|\leq 4} d_k(t) 2^{\frac{k}{2}} \Bigr) c_j(t)  2^{\frac{j}{2}} \bigl( \B'(t) \bigr)^{\frac12} \sup_{k\geq -4}\Bigl(2^{-k}\| \dSv_{k-1} w(t) \|_{L^4_{\rm h}(L^\infty_{\rm v})}\Bigr) \\
&\qquad\qquad\qquad\qquad\qquad\qquad\qquad\qquad\times\|  w(t)\|_{\dHh^{0,-\f12}}^\f12\|  \nabla_\h w(t) \|_{\dHh^{0,-\f12}}^\f12,
\end{align*}
it follows from a similar derivation of \eqref{eq5.36}  that for $k\geq -4$,
\begin{equation}\label{eq5.44}
\begin{split}
\| \dSv_{k-1}  w(t) \|_{L^4_{\rm h}(L^\infty_{\rm v})}
& \lesssim \sum_{\ell\leq k-2} 2^{\frac{\ell}{2}} \|  \dDv_{\ell} w(t) \|_{L^4_\h(L^2_\v)} \\
&\lesssim \sum_{\ell\leq k-2} 2^{\frac{\ell}{2}} \| \dDv_{\ell}  w(t) \|_{L^2}^\f12 \| \dDv_{\ell} \nabla_\h w(t) \|_{L^2}^\f12  \\
& \lesssim \|w(t)\|_{\dBl^{0,\f12}}^\f12\|\nabla_\h w(t)\|_{\dBl^{0,\f12}}^\f12+\|w(t)\|_{\dHh^{0,-\f12}}^{\f12}\|\nabla_\h w(t)\|_{\dHh^{0,-\f12}}^{\f12}\sum_{\ell=0}^{\max(k-2,0)} c_\ell(t) 2^\ell \\
&\lesssim 2^{k} \Bigl(\|w(t)\|_{\dBl^{0,\f12}}^\f12\|\nabla_\h w(t)\|_{\dBl^{0,\f12}}^\f12+\|w(t)\|_{\dHh^{0,-\f12}}^{\f12}\|\nabla_\h w(t)\|_{\dHh^{0,-\f12}}^{\f12}\Bigr) ,
\end{split}
\end{equation}
which together with $\sum_{|k-j|\leq 4} d_k(t) 2^{\f{k}2}\lesssim c_j(t) 2^{\f{j}2}$ implies
\begin{equation} \label{eq5.45}
\begin{aligned}
  \big| \bigl( \dDv_j (\dTv_{\p_3 w} u^3) \mid \dDv_j w \bigr) \big| &\lesssim c_j^2(t)  2^{j} \bigl( \B'(t) \bigr)^{\frac12} \|  w(t) \|_{X} \|  \nabla_\h w(t) \|_{X}.
\end{aligned}
\end{equation}

Finally, for the remainder term with $j \geq 0$, we deduce from Lemma \ref{lemBern}, $\p_3u^3=-\dive_\h u^\h$ and \eqref{eq5.20} (with $v$  being replaced by $u$) that
\begin{equation}\label{eq5.45a}
  \begin{aligned}
  &\big| \bigl( \dDv_j \dRv(u^3, \p_3 w) \mid \dDv_j w \bigr) \big|  \\
&\lesssim   2^{\frac{j}{2}} \sum_{k \geq j-3} \| \wt{\D}^\v_k u^{3}(t) \|_{L^2} \| \dDv_k \p_3 w(t) \|_{L^4_\h(L^2_\v)} \| \dDv_j w(t) \|_{L^4_{\rm h}(L^2_{\rm v})}  \\
&\lesssim   2^{\frac{j}{2}} \sum_{k \geq j-3} \| \wt{\D}^\v_k \nabla_\h u^\h(t) \|_{L^2} \| \dDv_k w(t) \|_{L^4_\h(L^2_\v)} \| \dDv_j w(t) \|_{L^4_{\rm h}(L^2_{\rm v})}  \\
&\lesssim   2^{\frac{j}{2}} \sum_{k \geq j-3} d_k(t) 2^{-\f{k}2} \bigl(\B'(t)\bigr)^\f12 \| \dDv_k  w(t) \|_{L^2}^\f12\| \dDv_k \nabla_\h w(t) \|_{L^2}^\f12 \| \dDv_j w(t) \|_{L^2}^\f12\| \nabla_\h \dDv_j w(t) \|_{L^2}^\f12 .
\end{aligned}
\end{equation}
For $0\leq j\leq J(t)$ with $J(t)$ being determined by  \eqref{def:J(t)}, we have
\begin{align*}
  &2^{\frac{j}{2}} \sum_{k \geq j-3} d_k(t) 2^{-\f{k}2}  \| \dDv_k  w(t) \|_{L^2}^\f12\| \dDv_k \nabla_\h w(t) \|_{L^2}^\f12 \| \dDv_j w(t) \|_{L^2}^\f12\| \nabla_\h \dDv_j w(t) \|_{L^2}^\f12\\
  &\lesssim c_j(t) 2^j\Bigl(  \| w(t)\|_{\dBl^{0,\f12}}^\f12\|\nabla_\h w(t)\|_{\dBl^{0,\f12}}^\f12 \sum_{-3\leq k\leq -1} d_k(t) 2^{-k} \\
  &\qquad\qquad+  \| w(t)\|_{\dHh^{0,-\f12}}^\f12 \|\nabla_\h w(t)\|_{\dHh^{0,-\f12}}^\f12 \sum_{ k\geq0} d_k(t) c_k(t) \Bigr) \| w(t)\|_{\dHh^{0,-\f12}}^\f12 \|\nabla_\h w(t)\|_{\dHh^{0,-\f12}}^\f12,
  \end{align*}
  so that
  \begin{align*}
  &\sum_{j=0}^{J(t)}2^{-j}\big| \bigl( \dDv_j \dRv(u^3, \p_3 w) \mid \dDv_j w \bigr) \big|   \\
  &\lesssim \sqrt{J(t)}\bigl(\B'(t)\bigr)^\f12{\| w(t)\|_{X}^\f12\|\nabla_\h w(t)\|_{X}^\f12}\| w(t)\|_{\dHh^{0,-\f12}}^\f12 \|\nabla_\h w(t)\|_{\dHh^{0,-\f12}}^\f12.
\end{align*}

On the other hand,  for $j\geq J(t)$, we use $w=u-v$ to get
\begin{align*}
  &\sum_{k \geq j-3} d_k(t) 2^{-\f{k}2}  \| \dDv_k  w(t) \|_{L^2}^\f12\| \dDv_k \nabla_\h w(t) \|_{L^2}^\f12\\
  &\lesssim \| (u,v)(t)\|_{\dB^{0,\f12}}^\f12\|\nabla_\h (u,v)(t)\|_{\dB^{0,\f12}}^\f12 \sum_{k\geq j-3}d_k(t) 2^{-k}\\
  &\lesssim c_j(t) 2^{-J(t)} \bigl(\B'(\tau)\bigr)^\f14,
\end{align*}
so that for $j\geq J(t)$, there holds
\begin{align*}
  &2^{\frac{j}{2}} \sum_{k \geq j-3} d_k(t) 2^{-\f{k}2} { \bigl(B'(t)\bigr)^\f12 }\| \dDv_k  w(t) \|_{L^2}^\f12\| \dDv_k \nabla_\h w(t) \|_{L^2}^\f12 \| \dDv_j w(t) \|_{L^2}^\f12\| \nabla_\h \dDv_j w(t) \|_{L^2}^\f12\\
  &\lesssim c_j^2(t) 2^j 2^{-J(t)} {\bigl(\B'(t)\bigr)^\f34 } \| w(t)\|_{\dHh^{0,-\f12}}^\f12 \|\nabla_\h w(t)\|_{\dHh^{0,-\f12}}^\f12 \\
  &\lesssim c_j^2(t) 2^j {\bigl(\B'(t)\bigr)^\f34}   \| w(t)\|_{\dHh^{0,-\f12}}^\f32 \|\nabla_\h w(t)\|_{\dHh^{0,-\f12}}^\f12.
\end{align*}

By combining these two estimates for different regime of $j$, we conclude that 
\begin{align*}
 & \sum_{j\geq 0}2^{-j} \big| \bigl( \dDv_j \dRv(u^3, \p_3 w) \mid \dDv_j w \bigr) \big|  
\lesssim   \bigl(\B'(t)\bigr)^\f34   \| w(t)\|_{\dHh^{0,-\f12}}^\f32 \|\nabla_\h w(t)\|_{\dHh^{0,-\f12}}^\f12  \\
&\qquad+\bigl(\B'(t)\bigr)^\f12 \bigl(-\ln\|w(t)\|_{\dHh^{0,-\f12}}^2\bigr)^\f12\|w(t)\|_{\dHh^{0,-\f12}}^\f12{\|w(t)\|_X^\f12\|\nabla_\h w(t)\|_{X}} .
\end{align*}
Notice that
\beq \label{P2t}
\Phi_2(\tau)\eqdefa-\tau\ln{\tau^2}\quad\mbox{is an increasing function of }\ \ \tau \ \ \mbox{if} \ \ 0<\tau<2^{-10},
\eeq
we thus obtain
\begin{equation}\label{eq5.46}
\begin{aligned}
 & \sum_{j\geq 0}2^{-j} \big| \bigl( \dDv_j \dRv(u^3, \p_3 w) \mid \dDv_j w \bigr) \big|  
\lesssim   \bigl(\B'(t)\bigr)^\f34   \| w(t)\|_{X}^\f32 \|\nabla_\h w(t)\|_{X}^\f12  \\
&\qquad\qquad\qquad\qquad\qquad\qquad+\bigl(\B'(t)\bigr)^\f12 \bigl(-\ln\|w(t)\|_{X}^2\bigr)^\f12\|w(t)\|_{X}\|\nabla_\h w(t)\|_{X}.
\end{aligned}
\end{equation}

By summarizing the estimates \eqref{eq5.43}, \eqref{eq5.45}, and \eqref{eq5.46}, we conclude the proof of  \eqref{eq5.47}. \end{proof}

\begin{proof} [Proof of Lemma \ref{S4lem7}]
We first get, by  applying Bony's decomposition \eqref{homo bony v} to  decompose $w^{\rm h} \cdot \nabla_\h v$  as in \eqref{S4eq5a}.

Considering the support properties to the Fourier transform of the terms in $\dTv_{w^{\rm h}} \nabla_\h v$, we  get, by applying \eqref{eq5.20} and \eqref{interpolation inequality}, that for $j\geq 0$,
\begin{align*}
  \big| \bigl( &\dDv_j (\dTv_{w^{\rm h}} \nabla_\h v) \mid \dDv_j w \bigr) \big| \,  \\
\lesssim &\sum_{|k-j|\leq 4} \|\dSv_{k-1}w^\h(t)\|_{L^4_\h(L^\oo_\v)} \| \dDv_k \nabla_\h v(t) \|_{L^2} \| \dDv_j w(t) \|_{L^2}^{\frac12} \| \nabla_\h \dDv_j w(t) \|_{L^2}^{\frac12}  \\
\lesssim &\sum_{|k-j|\leq 4} d_k(t) 2^{-\f{k}2} \bigl( \B'(t) \bigr)^{\frac12} \|\dSv_{k-1}w(t)\|_{L^4_\h(L^\oo_\v)}\| \dDv_j w(t) \|_{L^2}^{\frac12} \| \nabla_\h \dDv_j w(t) \|_{L^2}^{\frac12}\\
\lesssim &\Bigl( \sum_{|k-j|\leq 4} d_k(t) 2^{\frac{k}{2}} \Bigr) c_j(t)  2^{\frac{j}{2}} \bigl( \B'(t) \bigr)^{\frac12} \sup_{k\geq -4}\Bigl(2^{-k}\| \dSv_{k-1} w(t) \|_{L^4_{\rm h}(L^\infty_{\rm v})}\Bigr)\\
&\qquad\qquad\qquad\qquad\qquad\qquad\qquad\qquad\times \|  w(t)\|_{\dHh^{0,-\f12}}^\f12\|  \nabla_\h w(t) \|_{\dHh^{0,-\f12}}^\f12,
\end{align*}
which together with \eqref{eq5.44} implies
\begin{equation}\label{eq5.48}
\begin{aligned}
  \big| \bigl( \dDv_j (\dTv_{w^{\rm h}} \nabla_\h v) \mid \dDv_j w \bigr) \big|  &\lesssim c_j^2(t)  2^{j} \bigl( \B'(t) \bigr)^{\frac12} \|  w(t) \|_{X} \|  \nabla_\h w(t) \|_{X}. \end{aligned}
\end{equation}

Whereas by applying \eqref{eq5.22} and \eqref{interpolation inequality}, we find
\begin{align*}
 \big| \bigl( &\dDv_j (\dTv_{\nabla_\h v} w^{\rm h}) \mid \dDv_j w \bigr) \big|  \\
\lesssim &\sum_{|k-j|\leq 4}  \| \dSv_{k-1} \nabla_\h v(t) \|_{L^2_{\rm h}(L^\infty_{\rm v})} \| \dDv_k w^{\rm h}(t) \|_{L^4_{\rm h}(L^2_{\rm v})} \| \dDv_j w(t) \|_{L^4_{\rm h}(L^2_{\rm v})}  \\
\lesssim &\sum_{|k-j|\leq 4}  \bigl(\B'(t)\bigr)^\f12\| \dDv_{k}  w(t) \|_{L^2}^\f12\| \dDv_{k} \nabla_\h w(t) \|_{L^2}^\f12 \| \dDv_j w(t) \|_{L^2}^\f12 \|\nabla_\h \dDv_j w(t) \|_{L^2}^\f12 \\
\lesssim &\Bigl( \sum_{|k-j|\leq 4} c_k(t) 2^{\frac{k}{2}} \Bigr) c_j (t) 2^{\frac{j}{2}}  \bigl( \B'(t) \bigr)^{\frac12} \| (\mathrm{Id}-\dSv_{-5})  w(t)\|_{\dH^{0,-\f12}}^\f12  \|(\mathrm{Id}- \dSv_{-5}) \nabla_\h w(t) \|_{\dH^{0,-\f12}}^\f12\\
&\qquad \times   \|  w(t)\|_{\dHh^{0,-\f12}}^\f12\|  \nabla_\h w(t) \|_{\dHh^{0,-\f12}}^\f12,
\end{align*}
from which, \eqref{eq5.33} and \eqref{eq5.42}, we infer
\begin{equation} \label{eq5.50}
\begin{aligned}
  \big| \bigl( \dDv_j (\dTv_{\nabla_\h v} w^{\rm h}) \mid \dDv_j w \bigr) \big| \lesssim &c_j^2(t)  2^{j}  \bigl( \B'(t) \bigr)^{\frac12} \|  \nabla_\h w(t) \|_{X}  \|  w(t)\|_{X}.
\end{aligned}
\end{equation}

Finally, for the remainder term, we deduce from Lemma \ref{lemBern}, \eqref{eq5.20} and \eqref{interpolation inequality} that
\begin{align*}
 \big| \bigl( &\dDv_j \dRv(w^{\rm h}, \nabla_\h v) \mid \dDv_j w \bigr) \big|  \\
\lesssim &  2^{\frac{j}{2}} \sum_{k \geq j-3} \| \dDv_k w^{\rm h}(t) \|_{L^4_{\rm h}(L^2_{\rm v})} \| \wt{\D}^\v_k \nabla_\h v(t) \|_{L^2} \| \dDv_j w(t) \|_{L^4_{\rm h}(L^2_{\rm v})} \\
\lesssim &  2^{\frac{j}{2}} \sum_{k \geq j-3} d_k(t) 2^{-\f{k}2} \bigl(\B'(t)\bigr)^\f12 \| \dDv_k  w(t) \|_{L^2}^\f12\| \dDv_k \nabla_\h w(t) \|_{L^2}^\f12 \| \dDv_j w(t) \|_{L^2}^\f12\| \nabla_\h \dDv_j w(t) \|_{L^2}^\f12 .
\end{align*}
Notice that the above term is exactly the same  as in the last line of \eqref{eq5.45a},  we get, by a simialr derivation of  \eqref{eq5.46}, that 
\begin{equation}\label{eq5.51}
\begin{aligned}
 &  \sum_{j\geq 0}2^{-j}\big| \bigl( \dDv_j \dRv(w^\h, \nabla_\h v) \mid \dDv_j w \bigr) \big|  
\lesssim  \bigl(\B'(t)\bigr)^\f34   \| w(t)\|_{X}^\f32 \|\nabla_\h w(t)\|_{X}^\f12  \\
&\qquad\qquad\qquad\qquad\qquad\qquad+\bigl(\B'(t)\bigr)^\f12 \bigl(-\ln\|w(t)\|_{X}^2\bigr)^\f12\|w(t)\|_{X}\|\nabla_\h w(t)\|_{X}.
\end{aligned}
\end{equation}

By summarizing the estimates \eqref{eq5.48}, \eqref{eq5.50}, and \eqref{eq5.51}, we complete the proof of \eqref{eq5.52}. \end{proof}

  \begin{proof}[Proof of Lemma \ref{S4lem8}]
By applying Bony's decomposition \eqref{homo bony v}, we decompose $w^3 \p_3 v$ as in \eqref{S4eq6a}.

For the high frequency part of the paraproduct $\dTv_{w^3} \p_3 v$, we shall encounter  the difficulty of lossing one vertical derivative,  so that the trick by introducing $J(t)$ is needed once again. Indeed  we deduce from Lemma \ref{lemBern} and \eqref{sharp interpolation} that for any $p\geq 4$,
\begin{align*}
  \| \dDv_k \p_3v(t) \|_{L^p_{\rm h}(L^2_{\rm v})} &\lesssim \sqrt{p} 2^k \|\dDv_k v(t)\|_{L^2}^\f2p \|\dDv_k \nabla_\h v(t)\|_{L^2}^{1-\f2p}\\
  &\lesssim \sqrt{p}d_k(t) 2^{\f{k}2}\|v(t)\|_{\dB^{0,\f12}}^{\f2p}\|\nabla_\h v(t)\|_{\dB^{0,\f12}}^{1-\f2p}\\
  &\lesssim \sqrt{p}d_k(t) 2^{\f{k}2} \bigl(\B'(t)\bigr)^{\f12-\f1p},
\end{align*}
from which and \eqref{interpolation inequality}, we infer
\begin{equation}\label{eq5.53}
  \begin{aligned}
|\bigl( &\dDv_j (T^{\rm v}_{w^3} \partial_3 v) \mid  \dDv_j w \bigr)| \\
\lesssim &\sum_{|k-j|\leq 4} \|\dSv_{k-1} w^3(t)\|_{L^2_\h(L^\oo_\v)} \inf_{p\in [4,\oo[} \Bigl(\| \dDv_k \p_3v(t) \|_{L^p_{\rm h}(L^2_{\rm v})} \|  \dDv_j w(t)\|_{L^{\f{2p}{p-2}}_{\rm h}(L^2_{\rm v})}\Bigr)\\
\lesssim &\sum_{|k-j|\leq 4} d_k(t) 2^{\f{k}2} \| \dSv_k {w^3(t)} \|_{L^2_\h(L^\oo_\h)}\\
&\qquad\qquad\qquad\times\inf_{p\in [4,\oo[} \Bigl(\sqrt{p}\bigl( B'(t) \bigr)^{\f12-\f1p}  \|  \dDv_j w(t) \|_{L^2}^{1-\f2p} \| \nabla_\h \dDv_j w(t) \|_{L^2}^\f2p\Bigr)\\
\lesssim& \Bigl(\sum_{|k-j|\leq 4} d_k(t) 2^{\f{k}2}\Bigr) c_j(t) 2^{\f{j}2}  \sup_{k\geq -5}\| \dSv_k {w^3(t)} \|_{L^2_\h(L^\oo_\h)}\\
&\qquad\qquad\qquad\times\inf_{p\in [4,\oo[} \Bigl(\sqrt{p}\bigl( \B'(t) \bigr)^{\f12-\f1p}  \|   w(t) \|_{\dHh^{0,-\f12}}^{1-\f2p} \| \nabla_\h  w(t) \|_{\dHh^{0,-\f12}}^\f2p\Bigr).
\end{aligned}
\end{equation}
Let us  focus on the term  $\dSv_k {w^3 }$. Recalling the definition of $J(t)$ defined in \eqref{def:J(t)}. In case  $k\leq J(t)+2$, we get, by applying Lemma \ref{lemBern} and $\p_3 w^3=-\dive_\h w^\h,$ that
\begin{align*}
  \|\dSv_{k-1} w^3(t)\|_{L^2_\h(L^\oo_\v)} &\lesssim  \sum_{\ell\leq k-2} 2^{\f{\ell}2} \|\dot{\Delta}^\v_\ell w^3(t)\|_{L^2} \\
  &\lesssim  \sum_{\ell\leq 0} 2^{\f{\ell}2} \|\dot{\Delta}^\v_\ell w^3(t)\|_{L^2}+ \sum_{0\leq\ell\leq \max(k,-1)} 2^{-\f{\ell}2} \|\dot{\Delta}^\v_\ell \p_3 w^3(t) \|_{L^2} \\
  &\lesssim \|w(t)\|_{\dBl^{0,\f12}} \sum_{\ell\leq-1} d_\ell(t) +\|\nabla_\h w(t)\|_{\dHh^{0,-\f12}}\sum_{\ell=0}^{J(t)}  c_\ell(t)\\
  &\lesssim \|w(t)\|_{\dBl^{0,\f12}}+\sqrt{J(t)}\|\nabla_\h w(t)\|_{\dHh^{0,-\f12}}.
\end{align*}
In case $k\geq J(t)+3$, we  get, by using  $\p_3 w^3=-\dive_\h w^\h$ and $w=u-v,$ that
\begin{align*}
  \|(\dSv_{k-1}-\dSv_{J(t)})w^3(t)&\|_{L^2_\h(L^\oo_\v)} \lesssim \sum_{\ell=J(t)+1}^{k-2} 2^{\f{\ell}2}\|\dot{\Delta}^\v_\ell w^3(t)\|_{L^2}\lesssim \sum_{\ell=J(t)+1}^{k-2} 2^{-\f{\ell}2}\|\dot{\Delta}^\v_\ell \nabla_\h w^\h(t)\|_{L^2}\\
  &\lesssim \sum_{\ell=J(t)+1}^{k-2} 2^{-\f{\ell}2}\|\dot{\Delta}^\v_\ell \nabla_\h (u,v)(t)\|_{L^2}\lesssim \|\nabla_\h (u,v)(t)\|_{\dB^{0,\f12}}\sum_{\ell=J(t)+1}^{k-2} d_\ell(t)2^{-\ell} \\
  &\lesssim \bigl(\B'(t)\bigr)^\f12 2^{-J(t)} \lesssim \bigl(\B'(t)\bigr)^\f12 \|w(t)\|_{\dHh^{0,-\f12}}.
\end{align*}

By combining the above two cases and using the fact:  $1\leq {\B'(t)}, $ we conclude that for all $k\geq -5$,
\begin{equation}\notag
  \begin{aligned}
  \|\dSv_{k-1} w^3(t)\|_{L^2_\h(L^\oo_\v)} \lesssim &\bigl(\B'(t)\bigr)^\f12  \|w(t)\|_{X}+\sqrt{J(t)}\|\nabla_\h w(t)\|_{{X}}.
\end{aligned}
\end{equation}
By substituting the above estimate into \eqref{eq5.53}, we find
\begin{equation}\label{eq5.56a}
\begin{aligned}
  \big|& \bigl( \dDv_j (\dTv_{w^3} \p_3 v) \mid \dDv_j w \bigr) \big| \lesssim c_j^2(t)  2^{j}  \bigl(\B'(t)\bigr)^\f34 \|w(t)\|_{X}^\f32 \|\nabla_\h w(t)\|_{{X}}^\f12\\
&+c_j^2(t)  2^{j} \|\nabla_\h w(t)\|_{{X}} \inf_{p\in [4,\oo[} \Bigl(\sqrt{J(t)p}\bigl( \B'(t) \bigr)^{\f12-\f1p}  \|   w(t) \|_{\dHh^{0,-\f12}}^{1-\f2p} \| \nabla_\h  w(t) \|_{\dHh^{0,-\f12}}^\f2p\Bigr),
\end{aligned}
\end{equation}
where we take $p=4$ for the first part without $J(t)$.
Then we get, by  taking the special choice $p=\ln J(t)+2$ in the second line of \eqref{eq5.56a} and using Young's inequality, that
\begin{align*}
 & \sqrt{J(t)p}\bigl( \B'(t) \bigr)^{\f12-\f1p}  \|   w(t) \|_{\dHh^{0,-\f12}}^{1-\f2p} \| \nabla_\h  w(t) \|_{\dHh^{0,-\f12}}^\f2p \\
  &\lesssim  \Bigl( \B'(t)\bigl(J(t)p\bigr)^{\f{p}{p-2}}\|w(t)\|_{\dHh^{0,-\f12}}^2\Bigr)^{\f12-\f1p} \|\nabla_\h w(t)\|_{\dHh^{0,-\f12}}^\f2p\\
  &\lesssim \Bigl( \B'(t)J(t)\ln J(t)\|w(t)\|_{\dHh^{0,-\f12}}^2\Bigr)^{\f12-\f1p} \|\nabla_\h w(t)\|_{\dHh^{0,-\f12}}^\f2p\\
  &\lesssim  \sum_{\ka=1}^2 \Bigl( \B'(t)\Bigl(-\ln\|w(t)\|_{\dHh^{0,-\f12}}^2\Bigr)\ln \Bigl(-\ln\|w(t)\|_{\dHh^{0,-\f12}}^2\Bigr)\|w(t)\|_{\dHh^{0,-\f12}}^2\Bigr)^{\f\kappa4} \|\nabla_\h w(t)\|_{\dHh^{0,-\f12}}^{1-\f\kappa2},
\end{align*}
from which, \eqref{P1t} and \eqref{eq5.56a},  we deduce   that 
\begin{equation}\label{eq5.56}
\begin{aligned}
  \big| \bigl( \dDv_j (\dTv_{w^3} \p_3 v) \mid \dDv_j w \bigr) \big| \lesssim & c_j^2(t)  2^{j} \sum_{\ka=1}^3 \|\nabla_\h w(t)\|_{{X}}^{2-\f\kappa2}\\
  &\times \Bigl( \B'(t)\bigl(-\ln\|w(t)\|_{X}^2\bigr)\ln \bigl(-\ln\|w(t)\|_{X}^2\bigr)\|w(t)\|_{X}^2\Bigr)^{\f\kappa4} .
\end{aligned}
\end{equation}

Whereas we get, by using Lemma \ref{lemBern} and $\p_3 w^3=-\dive_\h w^\h,$ 
and then applying \eqref{eq5.4} (with $u$ being  replaced by $v$) and \eqref{interpolation inequality}, that 
\begin{align*}
  \big| \bigl(& \dDv_j (\dTv_{\p_3 v} w^3) \mid \dDv_j w \bigr) \big|  \\
\lesssim &\sum_{|k-j|\leq 4}   \| \dSv_{k-1} \p_3 v(t) \|_{L^4_{\rm h}(L^\infty_{\rm v})} \| \dDv_k w^3(t) \|_{L^2} \| \dDv_j w(t) \|_{L^4_{\rm h}(L^2_{\rm v})} \\
\lesssim &\sum_{|k-j|\leq 4}  \| \dSv_{k-1}  v(t) \|_{L^4_{\rm h}(L^\infty_{\rm v})} \| \dDv_k \dive_\h w^\h(t) \|_{L^2} \| \dDv_j w(t) \|_{L^4_{\rm h}(L^2_{\rm v})} \\
\lesssim &\sum_{|k-j|\leq 4}  \bigl(\B'(t)\bigr)^\f14\| \dDv_{k} \nabla_\h w(t) \|_{L^2} \| \dDv_j w(t) \|_{L^2}^\f12 \|\nabla_\h \dDv_j w(t) \|_{L^2}^\f12 \\
\lesssim &\Bigl( \sum_{|k-j|\leq 4} d_k(t) 2^{\frac{k}{2}} \Bigr) c_j(t)  2^{\frac{j}{2}} \bigl( \B'(t) \bigr)^{\frac14} \|(\mathrm{Id}-\dSv_{-5}) \nabla_\h w(t) \|_{\dH^{0,-\f12}} \|  w(t)\|_{\dHh^{0,-\f12}}^\f12\|  \nabla_\h w(t) \|_{\dHh^{0,-\f12}}^\f12,
\end{align*}
which together with \eqref{eq5.33} implies
\begin{equation} \label{eq5.57}
\begin{aligned}
  \big| \bigl( \dDv_j (\dTv_{\p_3 v} w^3) \mid \dDv_j w \bigr) \big| \lesssim & c_j^2(t)  2^{j} \bigl( \B'(t) \bigr)^{\frac14} \|  w(t) \|_{{X}}^\f12\|  \nabla_\h w(t) \|_{X}^\f32.
\end{aligned}
\end{equation}

Finally, for the remainder term, we deduce from Lemma \ref{lemBern}, $\p_3u^3=-\dive_\h u^\h$, \eqref{eq5.38a} (with $u$ being replaced by $v$) and \eqref{eq5.38b} that
  \begin{align*}
 \big| \bigl( &\dDv_j \dRv(w^3, \p_3 v) \mid \dDv_j w \bigr) \big|  \\
\lesssim &  2^{\frac{j}{2}} \sum_{k \geq j-3} \| \dDv_k w^{3}(t) \|_{L^2} \inf_{p\in [4,\oo[}\Bigl(\| \wt{\D}^\v_k \p_3 v(t) \|_{L^p_\h(L^2_\v)} \| \dDv_j w(t) \|_{L^\f{2p}{p-2}_{\rm h}(L^2_{\rm v})} \Bigr) \\
\lesssim &  2^{\frac{j}{2}} \sum_{k \geq j-3} \| \dDv_k \nabla_\h w^\h(t) \|_{L^2} \inf_{p\in [4,\oo[}\Bigl( \| \wt{\D}^\v_k v(t) \|_{L^p_\h(L^2_\v)} \| \dDv_j w(t) \|_{L^\f{2p}{p-2}_{\rm h}(L^2_{\rm v})}\Bigr)  \\
\lesssim &  2^{\frac{j}{2}} \sum_{k \geq j-3} d_k(t) 2^{-\f{k}2} \inf_{p\in [4,\oo[}\Bigl(\sqrt{p}\bigl(\B'(t)\bigr)^{\f12-\f1p}\| \dDv_k \nabla_\h w(t) \|_{L^2} \| \dDv_j w(t) \|_{L^2}^{1-\f2p}\| \nabla_\h \dDv_j w(t) \|_{L^2}^{\f2p} \Bigr).
\end{align*}
Observing that the last line coincides with \eqref{eq5.38}, we get, by a similar  derivation of  \eqref{eq5.39}, that
\begin{equation}\label{eq5.58}
\begin{aligned}
 &\sum_{j\geq 0}2^{-j} \big| \bigl( \dDv_j \dRv(w^3, \p_3 v) \mid \dDv_j w \bigr) \big|  
\lesssim  \sum_{\kappa=1}^3\|\nabla_\h w(t)\|_{X}^{2-\f\kappa2}\\
  &\qquad\qquad\qquad\times\Bigl(\B'(t)\bigl(-\ln \|w(t)\|_{X}^2\bigr)\ln\bigl(-\ln \|w(t)\|_{X}^2\bigr)\| w(t)\|_{X}^2\Bigr)^{\f\kappa4}.
\end{aligned}
\end{equation}

By summarizing the estimates \eqref{eq5.56}, \eqref{eq5.57}, and \eqref{eq5.58}, we finish the proof of \eqref{eq5.59}. \end{proof}

We are now in a position to complete the proof of Proposition \ref{prop1.6}.

\begin{proof}[Proof of Proposition \ref{prop1.6}] In view of \eqref{S1eq1a} and \eqref{S1eq2a}, we get,   by summing up  \eqref{eq:prop5.1} and \eqref{eq:prop5.2}, that
  \begin{equation}\label{eq5.64}
\begin{aligned}
    &\|\fg_1 w\|_{\wt L^\oo_T(X)}^2+\|\fg_1\nabla_\h w\|_{\wt L^2_T(X)}^2+C_1\|\bigl(\B'\bigr)^\f12
    \fg_1 w{\|}_{\wt L^2_T(X)}^2 \\
& \leq {C} \|w_0\|_{X}^2+C \sum_{\kappa=1}^3  \| \fg_1 \nabla_\h w \|_{\widetilde L^2_T(X)}^{\f\kappa2}  \|\bigl(\B'\bigr)^\f12\fg_1 w\|_{\wt L^2_T(X)}^{2-\f{\kappa}2}+ C\sum_{\kappa=1}^3 \|\fg_1 \nabla_\h w\|_{\wt L^2_T(X)}^{2-\f\kappa2} \\
&\qquad\quad\times \Bigl(\int_0^T \B'(t)\Bigl( -\ln\|\fg_1w\|_{\wt L^\oo_t(X)}^2\Bigr)  \ln\Bigl( -\ln\|\fg_1w{\|}_{\wt L^\oo_t(X)}^2\Bigr)\|\fg_1w\|_{\wt L^\oo_t(X)}^2dt\Bigr)^\f{\kappa}4 .
\end{aligned}
\end{equation}

By applying Young's inequality, we find
\begin{align*}
  C &\sum_{\kappa=1}^3  \| \fg_1 \nabla_\h w \|_{\widetilde L^2_T(X)}^{\f\kappa2}  \|\bigl(\B'\bigr)^\f12\fg_1  w\|_{\wt L^2_T(X)}^{2-\f{\kappa}2}\\
&+ C\sum_{\kappa=1}^3 \|\fg_1\nabla_\h w\|_{\wt L^2_T(X)}^{2-\f\kappa2} \Bigl(\int_0^T \B'(t)\Bigl( -\ln\|\fg_1 w\|_{\wt L^\oo_t(X)}^2\Bigr)  \\
  &\qquad\times \ln\Bigl( -\ln\|\fg_1 w\|_{\wt L^\oo_t(X)}^2\Bigr)\|\fg_1w\|_{\wt L^\oo_t(X)}^2\,dt\Bigr)^\f{\kappa}4 \\
  \leq &\f12 \|\fg_1 \nabla_\h w\|_{\wt L^2_T(X)}^2 +1000C  \|\bigl(\B'\bigr)^\f12\fg_1 w\|_{\wt L^2_T(X)}^2 \\
  &+1000C \int_0^T \B'(t)\Bigl( -\ln\|\fg_1 w\|_{\wt L^\oo_t(X)}^2\Bigr)  \ln\Bigl( -\ln\|\fg_1w\|_{\wt L^\oo_t(X)}^2\Bigr)\|\fg_1 w\|_{\wt L^\oo_t(X)}^2\,dt. 
\end{align*}

By substituting the above inequality into \eqref{eq5.64} and taking $C_1$ larger than $1000C$, we conclude that 
\begin{align*}
  &\|\fg_1w\|_{\wt L^\oo_T(X)}^2+\|\fg_1\nabla_\h w{\|}_{\wt L^2_T(X)}^2\\
&\leq C \|w_0\|_{X}^2+C \int_0^T \B'(t)\Bigl( -\ln\|\fg_1 w\|_{\wt L^\oo_t(X)}^2\Bigr) \ln\Bigl( -\ln\|\fg_1 w\|_{\wt L^\oo_t(X)}^2\Bigr)\|\fg_1 w\|_{\wt L^\oo_t(X)}^2\,dt,
\end{align*}
which is the $t=T$ case of \eqref{eq:prop1.6}. This finishes the proof of Proposition \ref{prop1.6}.
\end{proof}
\smallskip

\appendix

\setcounter{equation}{0}
\section{Tool box on Littlewood–Paley theory}\label{App}

For the convenience of readers, here we collect some basic facts on Littlewood–Paley theory in this {appendix}. We first recall the following anisotropic Bernstein inequalities from \cite{CZ07, Paicu}:
\begin{lem}\label{lemBern}
{\sl Let ${\bf B}_{\rm v}$ be a ball of $\mathbb{R}_{\rm v}$, and $\mathcal{C}_{\rm v}$ a ring of $\mathbb{R}_{\rm v}$; let $1\leq p\leq \infty$ and $1\leq q_2\leq q_1\leq \infty$. Then there holds
\beno
\begin{aligned}
\mbox{if}\ \ \operatorname{Supp} \widehat{a} \subset 2^\ell{\bf B}_{\rm v}&\Rightarrow
\|\partial_{x_3}^N a\|_{L^{p}_{\rm h}(L^{q_1}_{\rm v})}
\lesssim 2^{\ell\left(N+\frac1{q_2}-\frac1{q_1}\right)} \|
a\|_{L^{p}_{\rm h}(L^{q_2}_{\rm v})};\\
\mbox{if}\ \ \ \operatorname{Supp} \widehat{a} \subset 2^\ell\mathcal{C}_{\rm v}&\Rightarrow
\|a\|_{L^{p}_{\rm h}(L^{q_1}_{\rm v})} \lesssim 2^{-\ell N}
\|\partial_{x_3}^N a\|_{L^{p}_{\rm h}(L^{q_1}_{\rm v})}.
\end{aligned}
\eeno
}
\end{lem}

To deal with the law of product in anisotropic Besov spaces, we shall constantly use Bony's decomposition (see \cite{Bo81}) in the vertical variable, which refers to
\begin{equation}\label{homo bony v}\begin{split}
 &ab = \dot{T}^{\mathrm{v}}_a b + \dot{T}^{\mathrm{v}}_b a + \dot{R}^{\mathrm{v}}(a,b) \quad \text{with } \dot{T}^{\mathrm{v}}_a b \eqdefa \sum_{\ell\in\mathbb{Z}} \dot{S}^{\mathrm{v}}_{\ell-1} a \, \dot{\Delta}^{\mathrm{v}}_\ell b,\\
& \dot{R}^{\mathrm{v}}(a,b) \eqdefa \sum_{\ell\in\mathbb{Z}} \dot{\Delta}^{\mathrm{v}}_\ell a \, \widetilde{\Delta}^{\mathrm{v}}_{\ell} b \quad \text{and} \quad \widetilde{\Delta}^{\mathrm{v}}_{\ell} b \eqdefa \sum_{j=\ell-1}^{\ell+1} \dot{\Delta}^{\mathrm{v}}_j b.
\end{split}\end{equation}

\section*{Acknowledgement}
N. Burq is supported by the European research Council (ERC) under the European Union’s Horizon 2020 research and innovation programme (Grant agreement 101097172 - GEOEDP).

  P. Zhang is partially  supported by National Key R$\&$D Program of China under grant 2021YFA1000800 and by National Natural Science Foundation of China under Grant 12421001, 12494542 and 12288201.

\section*{Declarations}

\subsection*{Conflict of interest} The authors declare that there are no conflicts of interest.

\subsection*{Data availability}
This article has no associated data.

\end{document}